\documentclass[12pt, a4paper]{amsart}
\usepackage{amsmath, amsfonts, amssymb, amsthm}
\usepackage[margin=1.1in, top=1in, bottom=1in]{geometry}
\usepackage{tikz}
\usepackage{enumerate}

\newcommand{\R}{\mathbb{R}}
\newcommand{\C}{\mathbb{C}}

\newcommand{\N}{\mathbb{N}}
\newcommand{\T}{\mathbb{T}}
\newcommand{\Z}{\mathbb{Z}}

\newcommand{\Aut}{\operatorname{Aut}}
\newcommand{\TCL}{\mathcal{T}C^*(\Lambda)}
\newcommand{\TT}{\mathcal{T}}
\newcommand{\OO}{\mathcal{O}}
\newcommand{\NO}{\mathcal{NO}}
\newcommand{\KK}{\mathcal{K}}
\newcommand{\LL}{\mathcal{L}}
\newcommand{\HH}{\mathcal{H}}
\newcommand{\I}{\mathcal{I}}
\newcommand{\J}{\mathcal{J}}
\newcommand{\NT}{\mathcal{NT}}
\newcommand{\cov}{\operatorname{cov}}
\newcommand{\id}{\operatorname{id}}
\newcommand{\supp}{\operatorname{supp}}
\newcommand{\clsp}{\overline{\operatorname{span}}}
\newcommand{\lspan} {\operatorname{span}}

\allowdisplaybreaks

\newtheorem{thm}{Theorem}
\newtheorem{lemma}[thm]{Lemma}
\newtheorem{cor}[thm]{Corollary}
\newtheorem{prop}[thm]{Proposition}
\theoremstyle{definition}

\newtheorem{example}[thm]{Example}

\newtheorem{remark}[thm]{Remark}

\numberwithin{equation}{section}
\numberwithin{thm}{section}

\title[KMS states]{\boldmath{KMS states  on $C^*$-algebras associated to\\ a family of $*$-commuting local homeomorphisms}}
\date{24 March, 2017; with corrections April 2018.}
\author{Zahra Afsar}

\author{Astrid an Huef}

\author{Iain Raeburn}

\address{Zahra Afsar, School of Mathematics and Applied Statistics, University of Wollongong,
NSW 2522, Australia}
\email{zafsar@uow.edu.au}

\address{Astrid an Huef and Iain Raeburn, Department of Mathematics and Statistics, University of Otago, PO Box~56, Dunedin 9054, New Zealand}
\address{Current address for Astrid an Huef and Iain Raeburn: School of Mathematics and Statistics, Victoria University of Wellington, PO Box~56, Wellington 6140, New Zealand}

\email{{astrid.anhuef, iain.raeburn}@vuw.ac.nz}

\thanks{This research was supported by grant 15-UOO-071 from the Marsden Fund of the Royal Society of New Zealand.}

\begin{document}

\begin{abstract}
We consider a family of $*$-commuting local homeomorphisms on a compact space, and build a compactly aligned  product system of Hilbert bimodules. The Nica-Toeplitz algebra of this system carries a gauge 
action of a higher-dimensional torus, and there are many possible dynamics obtained by composing with different embeddings of the real line in this torus. We study the KMS states of these dynamics. For large inverse temperatures including $\infty$,  we describe the  simplex of KMS states on the  Nica-Toeplitz algebra. 
We illustrate our main theorem by considering backward shifts on the infinite-path spaces of a class of $k$-graphs whose shift maps $*$-commute.
\end{abstract}

\keywords{Product system; $*$-commuting homeomorphisms; Toeplitz algebra; Nica--Toeplitz algebra, Cuntz--Pimsner algebra;  KMS state.}

\subjclass[2010]{46L35}

\maketitle

\section{Introduction}

Suppose that $\alpha$ is an action of the real line $\R$ by automorphisms of a $C^*$-algebra $A$. Operator-algebraic dynamical systems such as $(A,\R,\alpha)$  provide operator-algebraic  models for physical systems in quantum statistical physics \cite{BRII}. The equilibrium states of the  physical system are then the states  on $A$ that satisfy a  commutation relation called the KMS condition. This relation involves a
parameter $\beta$,  which is a real number interpreted as the inverse temperature of the physical system.
The KMS$_\beta$ condition makes sense for abstract dynamical systems as well as physical ones, and many authors  have studied KMS states in other purely mathematical contexts \cite{EFW,BC,PWY,EL,LN, KW,aHLRS,K, K2, DFK}. The results have often been fascinating.

Suppose that $A$ is a $C^*$-algebra. Fowler \cite{Fo} defined product systems of Hilbert $A$--$A$ bimodules over semigroups $P$. Following the construction for  individual bimodules in \cite{FR}, Fowler associated to each product system $X$ a Toeplitz algebra $\TT(X)$  which is universal for a family of Toeplitz representations of $X$, and a quotient $\OO(X)$ which is universal for a family of Cuntz--Pimsner covariant representations. When $P$ is the positive cone in a quasi-lattice ordered group $(G,P)$ as in \cite{Ni}, Fowler considered also a smaller family  of Nica-covariant Toeplitz representations, and  the Nica-Toeplitz algebra $\NT(X)$ is the quotient of $\TT(X)$ which is universal for Nica-covariant  representations. The algebra $\NT(X)$ is only tractable for a class of \emph{compactly aligned} product systems, and for such systems, $\OO(X)$ is a quotient of $\NT(X)$ (see Lemma~\ref{fixdef} below). 

There are many interesting examples of product systems over the additive semigroup $\N^k$, including ones associated to the $k$-graphs of Kumjian-Pask \cite{KP} (see \cite[page 1492]{FS}). For systems over $\N^k$, the universal properties of $\NT(X)$ and $\OO(X)$ give strongly continuous gauge actions of the torus $\T^k$. Composing with an embedding of $\R$ in $\T^k$ gives actions of $\R$ on $\NT(X)$ and $\OO(X)$, and we are interested in the equilibrium states  of the resulting dynamical systems.

Our approach is informed by our previous work \cite{AaHR}, where we studied the KMS states on the Toeplitz algebra and the Cuntz--Pimsner algebra of a Hilbert bimodule associated to a local homeomorphism on a compact space $Z$. Here we consider a family of $k$ $*$-commuting local homeomorphisms on the same space $Z$, use them to construct a product system $X$ of Hilbert $C(Z)$--$C(Z)$ bimodules over $\N^k$, and study the KMS states of dynamical systems based on $\NT(X)$  as discussed above. Our approach extends the results of \cite{AaHR} to product systems in a way parallel to the extension of results about the algebras of finite graphs in \cite{aHLRS1} to the algebras of $k$-graphs in \cite{aHLRS}.  

For large inverse temperatures we have very complete results:  we find an explicit  isomorphism between the simplex of  KMS$_\beta$ states of $\NT(X)$ and  a concretely-described simplex of measures on $Z$. The surjectivity of this isomorphism requires a rational independency condition on the dynamics which has also appeared in  \cite{HLS, aHLRS}.
At a critical inverse temperature $\beta_c$ determined by the dynamical properties of the local homeomorphisms,  all we can say is that there is at least one KMS$_{\beta_c}$ state of $\NT(X)$. We have been unable to show, even for a ``preferred dynamics'', that there are KMS$_{\beta_c}$ states which factor through a state of $\OO(X)$.

To analyse the KMS structure of $\NT(X)$, we need to be able to recognise when a given state is a KMS state. This usually involves a 
characterising formula on  nice  spanning elements of a dense $*$-subalgebra of analytic elements.    
In the absence of Toeplitz-Cuntz-Krieger type relations as in the Toeplitz algebras of $k$-graphs in \cite{aHLRS} and  with no orthonormal bases for the fibres of our product system  as in \cite{HLS}, such a  characterisation formula seemed elusive. Our innovation  is the discovery that if the local homeomorphisms $*$-commute, then there are suitable Parseval frames for the fibres of the product system which interact in a complex way. Using  these Parseval frames, we prove Toeplitz-Cuntz-Krieger type relations for  $\NT(X)$ and get our formula.

\subsection*{Outline}

We begin with a section on background material. We first discuss some general properties of Hilbert bimodules,  their  representations and some properties of their Parseval frames. Then in \S\ref{prodsys}, we review Fowler's definition of product systems and the family of compactly aligned product systems. In \S\ref{NTalg}, we discuss the Nica-Toeplitz algebra, and establish some basic properties which are hard to point to in the literature. In particular, we show that the Cuntz--Pimsner algebra, which was defined in \cite{Fo} as a quotient of $\TT(X)$, can also be viewed as a quotient of $\NT(X)$. We then have some short notes about various other topics, including a discussion of the $*$-commuting hypothesis.

In \S\ref{sec3}, we start with $k$ commuting  local homeomorphisms on  a compact Hausdorff space $Z$, and construct a compactly aligned product system $X$ over $\N^k$. This is based on constructions of Larsen \cite{Lar} and Brownlowe \cite[Proposition~3.2]{Br}. 

We next choose a vector $r=(r_j)\in (0,\infty)^k$ and construct a dynamics $\alpha^r:\R\to \Aut\NT(X)$ by composing the gauge action of $\T^k$ with the map $t\mapsto \big(e^{itr_j}\big)$ from $\R\to \T^k$. Every KMS functional restricts to a positive functional on the coefficient algebra $C(Z)$, which is implemented by a measure $\mu$ on $Z$. The KMS condition then imposes constraints on the measure $\mu$, which are analogues of the subinvariance relations arising in the analysis of KMS states on $k$-graph algebras \cite[\S4 and Theorem~6.1(a)]{aHLRS}. 
In \S\ref{sec4}, we describe the subinvariance relations arising here and their solutions (Proposition~\ref{series}). To find these solutions we need to restrict $\beta$ to a range $\beta>\beta_c$, and we describe the \emph{critical inverse temperature} $\beta_c$ in  Proposition~\ref{series}. It is a several-variable analogue of the critical inverse temperature in \cite[Proposition~4.2]{AaHR}.

The second step in our analysis is to find a way of recognising KMS$_\beta$ states in terms of their behaviour on a spanning family (Proposition~\ref{KMSPROP}). 
For this calculation, we need to assume that our local homeomorphisms $*$-commute, and we therefore assume this for the rest of the paper. The proof of Proposition~\ref{KMSPROP} is long and involved, and occupies the whole of \S\ref{sec4}.

In \S\ref{sec6} we prove our results about the KMS$_\beta$ states of $(\NT(X),\alpha^r)$ for $\beta$ larger than the critical value $\beta_c$ (Theorem~\ref{theorem1}). 
When we have only one local homeomorphism, our main theorem recovers that of \cite{AaHR}. The proof of Theorem~\ref{theorem1} occupies most of \S\ref{sec6},  but we also discuss  implications for the KMS states on $(\OO(X),\alpha^r)$ in Corollaries~\ref{cor6} and \ref{cor-rescue}.

In \S\ref{sec8}, we discuss the ground states of $(\TT(X),\alpha^r)$: Proposition~\ref{groundprop} says that they are parametrised by the entire simplex of probability measures on $Z$. For this system, every ground state is a KMS$_\infty$ state, and hence there is no further phase transition at $\infty$, as occurs, for example, in \cite{LR}.

Important examples of local homeomorphisms on compact spaces are the shifts on the path spaces for finite directed graphs. In \cite[\S7]{AaHR}, we showed that there are interesting relationships  between the Toeplitz and Cuntz-Krieger algebras of a finite graph $E$,  and the Toeplitz and Cuntz--Pimsner algebras of the associated shifts. In \S\ref{sec9}, we consider shift maps on the path spaces of $k$-graphs. There are now $k$ different shifts to consider, and the factorisation property of the $k$-graph implies that these shifts commute. They do not always $*$-commute, but there is a family of \emph{$1$-coaligned} graphs for which they do, and then we can  apply our results. 
For a finite $1$-coaligned $k$-graph $\Lambda$,  the infinite path space $\Lambda^\infty$ is compact, and we then show that every KMS$_\beta$ state of the Toeplitz algebra of the  $k$-graph is the restriction of a KMS$_\beta$ state of the Nica--Toeplitz algebra of $X(\Lambda^\infty)$.

Toeplitz algebras and Cuntz--Pimsner algebras of product systems over $\N^k$ have previously appeared in work of Solel \cite{S}. To reconcile our work with that of \cite{S}, we show in an appendix that the canonical representations of our product systems in $\NT(X)$ are ``doubly commuting'' in the sense of \cite{S}.

\section{Notation and conventions}

\subsection{Hilbert bimodules}

Let $A$ be a $C^*$-algebra, and let $X$  be a  right Hilbert $A$--$A$ bimodule. This means that $X$ is a right Hilbert $A$-module with a left action of $A$ implemented by a homomorphism $\varphi:A\to \LL(X)$ from $A$ into the $C^*$-algebra $ \LL(X)$ of adjointable operators on $X$. In other words, $X$ is a correspondence over $A$.  We say that $X$ is \textit{essential} if $X=\clsp\{\varphi(a)x:a\in A,x\in X\}.$  If $A$ is unital with identity $1_A$ and $\varphi(1_A)x=x$ for all $x\in X$,  then $X$ is essential.   We write  $_A{A}_A$ for the standard bimodule with the inner product given by 
$\langle a,b\rangle=ab^*$ and the actions given by multiplication in $A$.      
For $x,y\in X$, we write $\Theta_{x,y}$ for the adjointable operator on $X$ given by $\Theta_{x,y}(z)=x\cdot\langle y,z\rangle$.  We call $\KK(X):=\clsp\{\Theta_{x,y}:x,y\in X\}$ the algebra of \emph{compact operators on $X$}. 

A \emph{representation} $(\psi,\pi)$ of $X$ in a $C^*$-algebra $B$ consists of a linear map $\psi:X\to B$ and a homomorphism $\pi:A\to B$ such that 
\begin{align*}
\psi(a\cdot x\cdot b)=\pi(a)\psi(x)\pi(b)\text{ and }\pi(\langle x, y\rangle)=\psi(x)^*\psi(y)
\end{align*}
for all $x,y\in X$ and $a,b\in A$. Such a  representation  induces a  homomorphism \[\psi^{(1)}:\KK(X)\rightarrow B\] such that $\psi^{(1)}(\Theta_{x,y})=\psi_p(x)\psi_p(y)^*$ (see \cite[page 202]{P}).

Following \cite{FL,EV}, a sequence $\{x_i\}_{i=1}^d$ in $X$ is called a \textit{Parseval frame} for $X$ if 
\begin{align}\label{reconstruction}
	\sum_{i=1}^{d}x_i\cdot\langle x_i,x\rangle_A=x\quad  \text{for all $x\in X$}.
	\end{align}
The formula \eqref{reconstruction} is known as \textit{the reconstruction formula}. 
Let $\{x_i\}_{i=1}^d$ be a Parseval  frame for  $X$ and write $1_X$ for the  identity operator on $X$. The reconstruction formula implies that $\sum_{i=1}^d\Theta_{x_i,x_i}=1_X$. In particular  $1_X$ is compact.
 
If  $(\psi,\pi)$ is a  representation  of $X$ on a Hilbert space  $\HH$, then 
  \cite[Proposition~4.1(1)]{Fo}  gives a unique endomorphism $\alpha^{\psi,\pi}$ of $\pi(A)'$ such that 
\[\alpha^{\psi,\pi}(T)\psi(x)=\psi(x)T \text{  for all } T\in \pi(A)',\, x\in X, \text{ and }\]
\[\alpha^{\psi,\pi}(1)r=0 \text{ for  }r\in (\psi(X)\HH)^\bot.\]

The next lemma shows that if there is  a Parseval frame for $X$, then we can formulate the endomorphism $\alpha^{\psi,\pi}$ easily.
\begin{lemma}\label{nicaeasy}
	Let $X$ be a  right Hilbert $A$--$A$ bimodule. Suppose that $\{x_i\}_{i=1}^d$ is a Parseval frame for $X$. Let $(\psi,\pi)$ be a  representation of $X$ on a Hilbert space $\HH$ and let $\alpha^{\psi,\pi}$ be as above. Then
	\begin{align}\label{hilbermodule}
	 \alpha^{\psi,\pi}(1)=\psi^{(1)}(1_X)=\sum_{i=1}^d\psi_p(x_i)\psi_p(x_i)^*.
	\end{align}
\end{lemma}	
\begin{proof}
The second equality follows from the reconstruction formula and the formula for  $\psi^{(1)}$. To see the first equality,   
	by uniqueness  of $\alpha^{\psi,\pi}$,  it suffices to prove that
	\begin{itemize}
		\item[(a)]  $\psi^{(1)}(1_X)\psi(x)=\psi(x)$  for all $x\in X$, and
		\item[(b)]  $\psi^{(1)}(1_X)r=0$ for all  $r\in (\psi(X)\HH)^\bot$.
	\end{itemize}	
		
	To see (a), let $x\in X$.   The reconstruction formula for $x$ implies that 
	\[\psi(x)=\psi\Big(\sum_{j=1}^d x_j\cdot \langle x_j,x\rangle\Big)
	=\sum_{j=1}^d\psi(x_j)\pi (\langle x_j,x\rangle).\]
	Using this and  the second equality in \eqref{hilbermodule}, we have
	\begin{align*}
	\psi^{(1)}(1_X)\psi(x)&=\sum_{i,j=1}^d\psi(x_i)\psi(x_i)^*\psi(x_j)\pi (\langle x_j,x\rangle)
	=\sum_{i,j=1}^d\psi(x_i)\pi (\langle x_i,x_j\rangle)\pi (\langle x_j,x\rangle)\\
	&=\sum_{i,j=1}^d\psi(x_i\cdot\langle x_i,x_j\rangle)\pi (\langle x_j,x\rangle).
	\end{align*} 
	Rearranging this and two applications of the reconstruction formula give
	\begin{align*}
	\psi^{(1)}(1_X)\psi(x)&=\sum_{j=1}^d\psi\Big(\sum_{i=1}^d x_i\cdot\langle x_i,x_j\rangle\Big)\pi (\langle x_j,x\rangle)
=\sum_{j=1}^d\psi(x_j)\pi (\langle x_j,x\rangle)\\
	&=\sum_{j=1}^d\psi(x_j\cdot \langle x_j,x\rangle)=\psi(x).
	\end{align*}
This is precisely (a).
	
	For (b), fix $r\in (\psi(X)\HH)^\bot$.    Notice that for all $r'\in \HH$ we have
	\begin{align*}
	\Big(\Big(\sum_{i=1}^d\psi_p(x_i)\psi_p(x_i)^*\Big)r\,\Big|\, r'\Big)=\sum_{i=1}^d\big(r\,\big|\, \psi_p(x_i)\psi_p(x_i)^*r'\big)=0 \text{ for } r'\in \HH.
	\end{align*}
	It follows that $\psi^{(1)}(1_X)r=0$ and we have proved (b).
\end{proof}

Given two right  Hilbert $A$--$A$ bimodules   $X$ and $Y$,  we can form a balanced tensor product of $X$ and $Y$ as follows:
 Let  $X\odot Y$ be the algebraic  tensor product of $X$ and $Y$ and suppose that  
 $X\odot_A Y$   is  the quotient of   $X\odot Y$ by the subspace
\begin{align}\label{notapN}
N:=\lspan\{(x\cdot a)\odot y-x\odot(a\cdot y): x\in X, y\in Y, a\in A\}.
\end{align}
There is a well-defined right action of $A$ on $X\odot_A Y$ such that 
$(x\odot_A y)\cdot a=x\odot_A y\cdot a$  for $x\odot_A y\in X\odot_A Y, a\in A$. By \cite[Proposition~4.5]{La},
we can equip  $X\odot_A Y$  with a right $A$-valued inner product characterised by 
	\begin{align}\label{balinnerproduct}
	\big\langle  x\odot_A y, z\odot_A w  \big\rangle =\big\langle y, \varphi_Y\big(\langle x,z\rangle\big)w\big\rangle\quad \text{ for }x\odot_A y, z\odot_A w\in X\odot_A Y.
	\end{align}
Let $X\otimes_A Y$ be the completion of $X\odot_A Y$ with respect to this inner product \eqref{balinnerproduct}.
Lemma~2.16 of \cite{tfb} implies  that  \eqref{balinnerproduct} extends to a right $A$-valued inner product on $X\otimes_A Y$ as well. Thus $X\otimes_A Y$ is a right Hilbert $A$-module.

For every  $S \in \LL(X)$, \cite[Lemma~I.3]{W} gives  a unique operator $S\otimes 1_Y\in \LL(X\otimes_A Y)$ such that $S\otimes 1_Y(x\otimes y)=S(x)\otimes y$ for $x\otimes y\in X\otimes_A Y$.
Thus the map  $a\mapsto \varphi_X(a)\otimes 1_Y$ gives a left action of $A$ by adjointable operators on $X\otimes_A Y$. Thus   $X\otimes_A Y$ is  a right Hilbert $A$--$A$ bimodule which we call the  \textit{balanced tensor product} of $X$ and $Y$.
Throughout, we use $x\odot y$ for elements of $X\odot Y$ and we  write $x\otimes y$ for elements of both $X\odot_A Y$ and $X\otimes_A Y$.

\subsection{Product systems of Hilbert bimodules}\label{prodsys}
We use  the conventions of \cite{Fo} for the basics of  product systems of Hilbert bimodules. For  convenience, we use  the following equivalent formulation from 
(\cite[page 6]{SY}).

Suppose  that $P$ is a multiplicative semigroup with identity $e$, and let $A$ be a $C^*$-algebra. For each $p\in P$ let  $X_p$  be a right Hilbert $A$--$A$ bimodule and suppose that $\varphi_p:A\rightarrow \LL(X_p)$ is the homomorphism which defines the left action of $A$ on $X_p$.  A \textit{product system  of right Hilbert $A$--$A$ bimodules over $P$} (or  a \textit{product system  with fibres $X_p$ over $P$}) is  the disjoint union  $X:=\bigsqcup_{p\in P}X_p$
	such that:
	\begin{itemize}
		\item[(P1)] The identity fibre $X_e$  is  the standard bimodule $_A{A}_A$.
		\item[(P2)] $X$ is a semigroup  and  for each $p,q \in P\setminus\{e\}$  the map $(x,y)\mapsto xy:X_p\times X_q \rightarrow X_{pq}$, extends to  an isomorphism   $\sigma_{p,q}: X_p\otimes_A X_q\rightarrow X_{pq}$.
		\item[(P3)] The multiplications   $X_e\times X_p\rightarrow X_p$ and $X_p\times X_e\rightarrow X_p$ satisfy
		\[ax=\varphi_p(a)z,\quad xa=x\cdot a\, \text{ for }\, a\in X_e \text { and } x\in X_p.\]
	\end{itemize}
	If each fibre $X_p$ is essential, then we call $X$ a \textit{product system  of essential right Hilbert $A$--$A$ bimodules over $P$}. We write  $1_p$ for the identity operator on the fibre  $X_p$.
	
	The associativity of the multiplication in $X$ implies that for $p,q,r\in P, x\in X_p, y\in X_q, s\in X_r$, we have
\begin{align*}
\sigma_{pq,r}\big(\sigma_{p,q}(x\otimes y)\otimes s\big)=\sigma_{p,qr}\big(x\otimes \sigma_{q,r}(y\otimes s)\big).
\end{align*}
	
Let $p,q \in P\setminus\{e\}$ and $S\in \LL(X_p)$. The isomorphism $\sigma_{p,q}: X_p\otimes_A X_q\rightarrow X_{pq}$  gives    a homomorphism $\iota^{pq}_p:\LL(X_p)\rightarrow \LL(X_{pq})$
defined by
\[\iota^{pq}_p(S)=\sigma_{p,q}\circ (S\otimes 1_{X_q})\circ \sigma_{p,q}^{-1}.\]

Suppose that  $P$ is a subsemigroup of a group $G$ such that $P\cap P^{-1}=\{e\}$. Then 
 \[
 p\leq q  \Leftrightarrow p^{-1}q\in P
 \] 
 defines a partial order on $G$.
Following  \cite[Definition~2.1]{Ni} and \cite[Definition~6]{CL}, we say $(G,P)$ is a  \textit{quasi-lattice ordered group} if  every pair $p,q\in G$ with a common upper bound in $P$ has a least upper bound in $P$, which we denote by $p\vee q$. We write $p\vee q=\infty$ when  $p,q\in G$ have no common upper bound. A well-known example of a quasi-lattice ordered group is $(\Z^k,\N^k)$: for all $m,n \in \N^k$, there is a least upper bound $m\vee n $ with $i$-th coordinate $(m\vee n)_i:=\max\{m_i, n_i\}$.

Let $(G,P)$ be a quasi-lattice ordered group. A product system  of right Hilbert $A$--$A$ bimodules  over $P$ is \textit{compactly aligned}, if for all $p,q\in P$ with $p\vee q<\infty$,   $S\in \KK(X_p)$ and $T\in\KK(X_q)$, we have $\iota^{p\vee q}_p(S)\iota^{p\vee q}_q(T)\in \KK(X_{p\vee q})$.

\subsection{$C^*$-algebras associated to product systems of Hilbert bimodules}\label{NTalg}

	Let $P$ be a semigroup with identity $e$, and let $X$ be a product system  of right Hilbert $A$--$A$ bimodules over $P$.  Let $\psi$   be a function from	$X$ to a $C^*$-algebra $B$. Write $\psi_p$   for the restriction of $\psi$ to $X_p$.  We call  $\psi$
	a  \textit{Toeplitz representation}  of  $X$ if:
	\begin{itemize}
		\item[(T1)]  For each $p\in  P\setminus\{e\}$, $\psi_p:X_p\rightarrow B$ is linear, and $\psi_e:A\rightarrow B$ is a homomorphism,
		\item[(T2)]  $\psi_p(x)^*\psi_p(y)=\psi_e(\langle x,y\rangle)$ for  $p\in P$, and $x,y\in X_p$,
		\item[(T3)] $\psi_{pq}(xy)=\psi_p(x)\psi_q(y)$ for  $p,q\in P$, $x\in X_p$,  and $y\in X_q$.
	\end{itemize}
	Conditions (T1) and (T2) imply that $(\psi_p,\psi_e)$ is a  representation of the fibre $X_p$. Then there is a homomorphism   $\psi^{(p)}:\KK(X_p)\rightarrow B$ such that $\psi^{(p)}(\Theta_{x,y})=\psi_p(x)\psi_p(y)^*$.

The \emph{Toeplitz algebra} $\TT(X)$ is generated by a universal Toeplitz representation of $X$, say $\omega$.  Proposition~2.8 of \cite{Fo} says that there is such an algebra $\TT(X)$. If $T$ is a Toeplitz  representation  of $X$ in a $C^*$-algebra
$B$, then we write $T_*$  for the representation
of $\TT(X)$ in $B$ such that $T_*\circ \omega= T$.

A Toeplitz representation $\psi$ of $X$ is \textit{Cuntz--Pimsner covariant} if
	\begin{align}\label{cunzcon}	
	\psi_e(a)=\psi^{(p)}(\varphi_p(a))\quad\text{for all } p\in P, a\in \varphi_p^{-1}(\KK(X_p)).
	\end{align}
The \textit{Cuntz--Pimsner algebra} $\OO(X)$ is the quotient of $\TT(X)$ by the ideal
\begin{align}\label{quotient2}
\big\langle\omega(a)-\omega^{(p)}(\varphi_p(a)):p\in P, a\in \varphi_p^{-1}(\KK(X_p))\big\rangle.
\end{align}

Let $(G,P)$ be a quasi-lattice ordered group and suppose that $X$ is a product system of essential  right Hilbert $A$--$A$ bimodules over $P$.  Suppose that $\psi$ is  a Toeplitz representation  of $X$ on a Hilbert space  $\HH$. For each $p\in P$, let $\alpha^\psi_p$ be the map $\alpha^{\psi_p,\psi_0}$ associated to the representation $(\psi_p,\psi_0)$ of the fibre $X_p$ as in Lemma~\ref{nicaeasy}.
  A Toeplitz representation $\psi$ of $X$ on a Hilbert space  $\HH$ is \textit{Nica covariant} if  for every $p,q\in P$, we have
	\begin{align*}
	\alpha_p^\psi(1)\alpha_q^\psi(1)=\begin{cases}
	\alpha_{p\vee q}^\psi(1)&\text{if $p\vee q<\infty$}\\
	0&\text{otherwise.}\\
	\end{cases}
	\end{align*}
It follows from  Lemma~\ref{nicaeasy} that if each fibre in $X$ has a Parseval frame, then  a Toeplitz representation $\psi$ is   Nica covariant if and only if  for every $p,q\in P$,  we have
	\begin{align}\label{nicacov-using-lemma}
	\psi^{(p)}(1_p)\psi^{(q)}(1_q)=\begin{cases}
	\psi^{(p\vee q)}(1_{p\vee q})&\text{if $p\vee q<\infty$}\\
	0&\text{otherwise.}\\
	\end{cases}
	\end{align}
	
  Fowler showed in \cite[Proposition~5.6]{Fo} that we can extend the notion of  Nica covariance to representations in  $C^*$-algebras.
 A Toeplitz representation $\psi$ of $X$ in a $C^*$-algebra $B$ is \textit{Nica covariant} if and only for every $p,q\in P$, $S\in \KK(X_p)$,  and $T\in \KK(X_q)$, we have
	\begin{align*}
	\psi^{(p)}(S)\psi^{(q)}(T)=\begin{cases}
	\psi^{(p\vee q)}\big(\iota^{p\vee q}_p(S)\iota^{p\vee q}_q(T)\big)&\text{if $p\vee q<\infty$}\\
	0&\text{otherwise.}\\
	\end{cases}
	\end{align*}
The \textit{Nica-Toeplitz algebra } $\NT(X)$\footnote{In Fowler's paper   the  Nica-Toeplitz algebra is denoted by $\TT_{\cov}(X)$.}
is the   $C^*$-algebra  generated by   a universal  Nica covariant representation  of $X$, which in this paper we denote by $\psi$. 
 It follows from \cite[Theorem~6.3]{Fo} that
\begin{align*}
\NT(X)=\clsp\{\psi_p(x)\psi_q(y)^*:p,q\in P,x\in X_p,y\in X_q\}.
\end{align*}

The Cuntz--Pimsner algebra  $\OO(X)$ is by definition   a quotient of $\TT(X)$.  There is another  Cuntz--Pimsner algebra $\NO(X)$ in \cite{SY}, which is directly defined as a quotient of $\NT(X)$.  In the product systems considered  here,  the left action of $A$ on each fibre is by compact operators,  and hence   Cuntz--Pimsner covariance  implies  Nica covariance \cite[Proposition~5.4]{Fo}.  Also since the left action is injective, by \cite[Proposition~5.1]{SY}, the Cuntz--Pimsner covariance considered in \cite{SY} is equivalent to Fowler's Cuntz--Pimsner covariance given at  \eqref{cunzcon}. Therefore by  \cite[Remark~3.14]{SY} the two Cuntz--Pimsner  algebras coincide.  But we found it easier to work with $\OO(X)$. The next lemma shows that we can still express $\OO(X)$ as a quotient of $\NT(X)$. 

\begin{lemma}\label{fixdef}
	Let $(G, P)$ be a
	quasi-lattice ordered group, and let $X$ be a compactly aligned product system 
	of right Hilbert $A$--$A$ bimodules over $P$. Suppose that every Cuntz--Pimsner-covariant representation of $X$ is a
	Nica-covariant representation. Then $\OO(X)$ is the quotient of $\NT(X)$ by the ideal 
	\begin{align}\label{quotient0}
	\big\langle\psi_e(a)-\psi^{(p)}(\varphi_p(a)):p\in P, a\in \varphi_p^{-1}(\KK(X_p))\big\rangle.
	\end{align}
\end{lemma}

\begin{proof}
Let $\I:=\bigcap\big\{\ker \pi_*:\pi \text{ is a Cuntz--Pimsner-covariant representation of $X$}\big\}$  and let 
$\J:=\bigcap\{\ker \theta_*:\theta \text{ is a Nica-covariant representation of $X$}\}$.  Let $q_{\NT}$ be the quotient map $\TT(X)\rightarrow \TT(X)/\J$. A standard argument shows that $(\TT(X)/\J, q_{\NT}\circ \omega)$ is universal for  Nica-covariant representations, and hence it is canonically isomorphic to $(\NT(X),\psi)$. Similarly,  $\OO(X)$ is isomorphic to $\TT(X)/\I$.  In particular, $\I$ is the   ideal of \eqref{quotient2}.
Since  every Cuntz--Pimsner-covariant representation of $X$ is also a Nica-covariant representation, $\J\subseteq \I$. An application of the third isomorphism theorem in algebra gives a quotient map $q:\NT(X)\rightarrow \OO(X)$ such that $\ker q=\I/\J$. We have 
\begin{align}\label{q0}
	\I/\J=\big\{i+\J: i\in \I\big\}= \big\{q_{\NT}(i): i\in \I\big\}.
\end{align}
Now plugging the generating elements $\omega(a)-\omega^{(p)}(\varphi_p(a))$  of $\I$ into \eqref{q0} and 
using $q_{\NT}\circ \omega=\psi$, we obtain
	$\I/\J=\big\{\psi_e(a)-\psi^{(p)}(\varphi_p(a)):p\in P, a\in \varphi_p^{-1}(\KK(X_p))\big\}$.
Thus $\OO(X)$ is the quotient of $\NT(X)$ by the ideal \eqref{quotient0}
\end{proof}	

Let $X$ be a product system  of right Hilbert $A$--$A$ bimodules over $\N^k$.
A standard argument using the universal property shows that there is a strongly continuous \textit{gauge action} $\gamma:\T^k\rightarrow\NT (X)$ such that for each $x\in X_m$ we have  $\gamma_z(x)=z^{m}x$ (in multi-index
notation,   $z^m = \prod_{i=1}^kz_i^{m_i}$ for $z = (z_1,\dots, z_k) \in\T^k$ and $m\in \Z^k$).
Since $\gamma$ fixes the elements of the set in \eqref{quotient0}, it induces  a natural gauge action of $\T^k$ on $\OO(X)$. Now we can lift these actions to actions of the real line $\R$\footnote{To get the action of $\R$ on $\NT(X)$, we could also apply the argument of the paragraph before \cite [Lemma~3.2]{HLS} to the homomorphism 
$N:\Z^k\rightarrow (0,\infty)$ defined by 
$N(n) =\sum_{i=1}^kn_i r_i$}.

\subsection{The Fock representation}

Let $P$ be a semigroup with identity $e$ and suppose that   $X$ is  a product system  of right  Hilbert $A$--$A$ bimodules over  $P$. 
Define $r:X\rightarrow P$  by $r(x):=p$ for $x\in X_p$.
Following \cite{Fo}, we write $F(X)=\bigoplus_{p\in P}X_p$ and call it the \textit{Fock module}. 
Fowler shows in  \cite[page 340]{Fo} that  for $x\in X$ there is an adjointable operator  $T(x)$ on $F(X)$
determined by
$T(x)(\bigoplus x_p)=\bigoplus (xx_p)$.
The adjoint ${T(x)}^*$ is zero on any summand $X_p$ for which  $p\notin r(x)P$. If $p\in r(x)P$,  say $p=r(x)q$ for some $q\in P$, then there is an isomorphism  $\sigma_{r(x),q}: X_{r(x)}\otimes_A X_{q}\rightarrow X_p$, and the adjoint  ${T(x)}^*$  is  given by 
\begin{align}\label{fockadjoint}
{T(x)}^*\big(\sigma_{r(x),q}(y\otimes z)\big)=\langle x,y\rangle \cdot z.
\end{align}
Furthermore,  $T$ is a Toeplitz representation of $X$   called   the  \textit{Fock representation}.

Let $X$ be a  compactly aligned  product system  of right  Hilbert $A$--$A$ bimodules over  $\N^k$ and suppose that  the left action of $A$ on each fibre is by compact operators. Then  the  homomorphism $T_*:\NT(X)\rightarrow \LL(F(X))$ induced from the Fock representation  is faithful (see \cite[Remark 4.8]{HLS}).

\subsection{KMS states} 
Let $(A,\R,\alpha)$ be a $C^*$-algebraic dynamical system. An element $a\in A$ is \emph{analytic} if $t\mapsto \alpha_t(a)$  is the restriction of an entire function $z\mapsto\alpha_z(a)$ on $\C$.
Following recent conventions \cite{LR, aHLRS1,AaHR,aHLRS},   
we say a state $\phi$ of $(A,\R,\alpha)$ is a \emph{KMS state with inverse temperature $\beta$} (or a KMS$_\beta$ state) if
  $\phi(ab)=\phi(b\alpha_{i\beta}(a))$ for all analytic elements $a,b$. 
  A state $\phi$ is a \textit{KMS$_\infty$ state} if it is  the weak$^*$ limit of a sequence  of KMS$_{\beta_i}$ states as 
	$\beta_i\rightarrow \infty$ (see \cite{CM}).  We distinguish between the ground states and the KMS$_\infty$ states.  But in older literature (for example in  \cite{BRII, Pe}), there was no such a distinction.  For us, ground states are those for which $z\mapsto \phi(a\alpha_z(b))$ is bounded in the upper-half plane  for all analytic elements $a,b$.
The argument of \cite[page 19]{LR} shows that  it  is enough  to check both the ground state condition and the KMS condition  on a set of analytic elements which span  a dense subalgebra of   $A$.

\subsection{Topological graphs} A \emph{topological graph} $E=(E^0,E^1,r,s)$ consists of locally compact Hausdorff spaces $E^0$ and $E^1$, a continuous map $r:E^1\to E^0$ and a local homeomorphism $s:E^1\to E^0$.    Here $E^0$ and $E^1$ will always be compact. We use the convention of  \cite{Ra} for  paths in $E$: for example,  if $s(e)=r(f)$, then we think of $ef$  as a path of length $2$.  Since $E^0$ and $E^1$ are compact, the associated \textit{graph correspondence} $X(E)=C(E^1)$
 is the right-Hilbert $C(E^0)$--$C(E^0)$ bimodule  with module actions and inner product given by $(a\cdot x\cdot b)(z)=a(r(z)) x(z) b(s(z))$ and   $\langle x,y\rangle(z)=\sum_{s(w)=z} \overline{x(w)}y(w)$ for $a,b\in C(E^0)$, $x,y\in X(E)$ and $z\in E^0$.

\subsection{$*$-commuting local homeomorphisms}
Let $f,g$ be commuting maps on a set $Z$. Following  \cite{AR} and \cite[\S 10]{ER} we say
	$f,g$  \textit{$*$-commute}, if for every $x,y\in Z$ satisfying $f(x)=g(y)$, there exists a unique $z\in Z$ such that $x=g(z)$ and $y=f(z)$.
	The following  diagram illustrates this property:
	
	
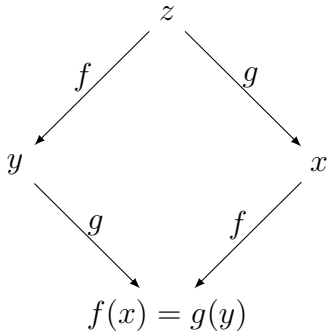
\begin{tikzpicture}[scale=0.2]
	\node  at (0,0) {};
	\node (a) at (35,20) {$z$};
	\node (b) at (25,10) {$y$};
	\node (c) at (45,10) {$x$};
	\node (d) at (35,0) {$f(x)=g(y)$};
	\draw[-latex, black](a) to node[pos=0.4,left] {$f$} (b);
	\draw[-latex, black](c) to node[pos=0.4,left] {$f$} (d);
	\draw[-latex, black ](b) to node[pos=0.4,right] {$g$} (d);
	\draw[-latex, black](a) to node[pos=0.4,right] {$g$} (c);
	\end{tikzpicture}
	
	We also say that a family of maps  $*$-commute  if every two of them $*$-commute. Given  $*$-commuting maps  $f,g, h$  on a space $Z$,  Proposition~10.2 of \cite{ER} implies  that
	$f^i$ and $g^j$  $*$-commute for $i,j\in \N$. 
	Lemma~1.3 of \cite{St}  shows that $f$ and $g\circ h$  $*$-commute.
\begin{remark}
	Let $h_1,\dots,h_k$ be $*$-commuting local homeomorphisms on a space $Z$ and take  $m,n\in \N^k$.  Throughout  we write $m\wedge n$  for the element of $\N^k$ with  $(m\wedge n)_i:=\min\{m_i,n_i\}$. We also define
	\[
	h^m:=h_1^{m_1}\circ\cdots\circ h_k^{m_k}.
	\]
Observe that if  $m\wedge n=0$, then  the  local homeomorphisms appearing in $h_1^{m_1}\circ\cdots\circ h_k^{m_k}$  do not appear in $h_1^{n_1}\circ\cdots\circ h_k^{n_k}$. Then  the argument of the  previous paragraph shows that $h^m$ and $h^n$ $*$-commute.
The  condition $m\wedge n=0$ here  is very   crucial.  In particular,  $*$-commuting local homeomorphisms need not  $*$-commute with themselves. Thus   $h^m$ and $h^n$ may not  $*$-commute in general.
\end{remark}

\subsection{Measures}
All the measures we consider here  are positive in the sense that they take values in $[0,\infty)$.  We write $M(Z)_{+}$ for the set of finite  Borel measures on $Z$. Some of the measures used here are defined by linear functionals on $C(Z)$. Then by the Riesz representation theorem (see \cite[Corollary~7.6]{F} for example) these measures are automatically regular. A \emph{probability measure} is a Borel measure with total mass $1$.

\section{A product system associated to a family of local homeomorphisms}\label{sec3}

In \cite[Lemma~5.2]{AaHR} we proved that for a local homeomorphism $f$ and the associated graph correspondence $X(E)$, there is an isomorphism from $X(E)\otimes_A X(E)$ onto the graph correspondence associated to $f\circ f$. The next lemma generalises this to  graph correspondences of two different local homeomorphisms. There is also a similar result in the dynamics arising from graph algebras (see \cite[Proposition~3.2]{Br}).
\begin{lemma}\label{coresiso}
	Let $f,g$ be surjective local homeomorphisms  on a compact Hausdorff space $Z$. Let  $A:=C(Z)$  and suppose that  $X(E_1)$,   $X(E_2)$ and $X(F)$ are the  graph correspondences related to the topological graphs $E_1=(Z,Z,\id,f)$,  $E_2=(Z,Z,\id,g)$, and   $F=(Z,Z,\id, g\circ f)$. Then there is an isomorphism $\sigma_{f,g}$ from  $X(E_1)\otimes_A X(E_2)$ onto $X(F)$ such that
	\begin{align}\label{isomor51}
	\sigma_{f,g}(x\otimes y)(z)=x(z) y(f(z)) \text{ for all $z\in Z$}.
	\end{align}
\end{lemma}
\begin{proof}
	Define    $\sigma:C(Z)\times C(Z)\rightarrow C(Z) $ by 
	$\sigma(x,y)(z)=x(z)y(f(z))$ for all  $x,y\in C(Z)$.
	Clearly  $\sigma$ is bilinear.
	Taking $y=1$  implies that  $\sigma$ is surjective.  	
	Then the universal property of the  algebraic tensor product  gives a unique surjective linear map   $\tilde{\sigma}:C(Z)\odot C(Z)\rightarrow C(Z)$   satisfying  $\tilde{\sigma}(x\odot y)(z)=x(z) y(f(z))$ for all $x\odot y\in C(Z)\odot C(Z)$. 
	Since $\tilde{\sigma}$ vanishes on the elements of the form \eqref{notapN}, it induces  a surjective linear map   $\sigma_{f,g}:C(Z)\odot_A C(Z)\rightarrow C(Z)$ satisfying \eqref{isomor51}.	
		
	To show that  $\sigma_{f,g}$ preserves the  actions, let  $x\otimes y\in C(Z)\odot_A C(Z)$, $a\in C(Z)$ and  $z\in Z$. It follows from \eqref{isomor51} that  
	\begin{align*}
	\sigma_{f,g}\big(x\otimes y\cdot a\big)(z)	&=x(z) (y\cdot a)(f(z))
	=x(z)y(f(z))a(g\circ f(z))\\
	&=\sigma_{f,g}(x\otimes y)(z)a(g\circ f(z))
	=\big(\sigma_{f,g}(x\otimes y)\cdot a\big)(z).
	\end{align*}
	Similarly   for the left action, we have $\sigma_{f,g}\big(a\cdot(x\otimes y)\big)(z)=\big(a\cdot \sigma_{f,g}(x\otimes y)\big)(z)$. 
	
	To see that   $\sigma_{f,g}$   preserves the  inner products, let $x\otimes y, x'\otimes y'\in C(Z)\odot_A C(Z)$. Since all the range maps are the identity and the source maps are $f,g,$ and $g\circ f$, respectively, we have
	\begin{align}
	\big\langle \sigma_{f,g}(x\otimes y),\sigma_{f,g}(x'\otimes y')\big\rangle(z)\notag&=\sum_{g\circ f(w)=z}\overline{\sigma_{f,g}(x\otimes y)(w)}\sigma_{f,g}(x'\otimes y')(w)\\
	\notag&=\sum_{g\circ f(w)=z}\overline{x(w) y(f(w))}x'(w) y'(f(w))\\
	\notag&=\sum_{g(v)=z}\Big(\sum_{ f(w)=v}\overline{x(w)}x'(w)\Big)\overline{ y(v)} y'(v)\\
\notag&=\sum_{g(v)=z}\langle x,x'\rangle(v)\overline{ y(v)} y'(v)\\
	&=\sum_{g(v)=z}\overline{ y(v)}\big(\langle x,x'\rangle\cdot y'\big)(v)\notag\\
	&=\big\langle y,\langle x,x'\rangle\cdot y'\big\rangle(z)\notag\\
	&=\big\langle x\otimes y,x'\otimes y'\big\rangle(z).\notag
	\end{align}
It follows that $\sigma_{f,g}$ is  an isometry on $C(Z)\odot_A C(Z)$, and hence it   extends  to an isomorphism  $\sigma_{f,g}$ of $X(E_1)\otimes_A X(E_2)$ onto $X(F)$ which satisfies  \eqref{isomor51}.
	\end{proof}
Now we can prove an analogue of \cite[Proposition~3.7]{Br} to build a product system.
\begin{prop}\label{prop1}
	Let $h_1,\dots,h_k$ be surjective  commuting   local homeomorphisms  on a compact Hausdorff space $Z$. For each $m\in \N^k$, let $X_m$ be the  graph correspondence associated to the topological graph $(Z,Z,\id,h^m)$. Suppose that $X:=\bigsqcup_{m\in \N^k}X_m$ and $A:=C(Z)$. Let $\sigma_{m,n}:X_m\otimes_A X_n\rightarrow X_{m+n}$ be the isomorphism obtained by applying  Lemma~\ref{coresiso} with the local homeomorphisms $h^m,h^n$.  Then $X$ is a compactly aligned  product system of essential right Hilbert $A$--$A$  bimodules over $\N^k$  with the multiplication given  by $xy := \sigma_{m,n}(x\otimes y)$ for
	$x\in X_m, y\in Y_n$, so that
\begin{align}\label{multiplicationformula1}
	(xy)(z)=x(z)y(h^m(z)) \text{ for } z\in Z.
\end{align}	
The left action of $A$ on each fibre $X_m$ is by compact operators.
\end{prop}

\begin{proof}
	An easy computation using \eqref{multiplicationformula1} shows that $X$ is a semigroup. Since the source map in the fibre $X_0$ is the identity, $X_0=_A{A}_A$ and therefore (P1) holds. Lemma~\ref{coresiso} gives
	(P2). Let $a\in A$ and $x\in X_m$. Then  (P3) follows from: 
	\[ax(z)=a(z)x(z)=(a\cdot x)(z)\text{ and } xa(z)=x(z)a(h^m(z))=(x\cdot a)(z).\]		
	
	To show that the fibre $X_m$ is essential, notice that $A=C(Z)$ is unital with the  identity $1_{C(Z)}:Z\rightarrow \C$ defined by $1_{C(Z)}(z)=1$ for all $z\in Z$. Since the left action is by pointwise multiplication,   $\varphi_m(1_{C(Z)})x=x$ for all $x\in X_m$. Thus  $X_m$ is essential. 
	
	To  see that the left action of $A$ on   the fibre $X_m$ is by compact operators,  choose an open cover $\{U_j:1\leq j\leq d \}$  of  $Z$ such that  $h^m|_{U_j}$ is injective and choose a partition of unity $\{\rho_j\}$ subordinate to $\{U_j\}$.  Define   $\xi_j:=\sqrt{\rho_j}$. We aim to show that 
	for each $a\in A$, the left action of $a$ on the fibre $X_m$ is by $\sum_{j=1}^{d}\Theta_{a\cdot\xi_j,\xi_j}$.	
	
	Take $x\in X_m$ and $z\in Z$, we have
	\begin{align*}
	\Big(\sum_{j=1}^{d}\Theta_{a\cdot\xi_j,\xi_j}(x)\Big)(z)&=\sum_{j=1}^{d}\big((a\cdot\xi_j)\cdot \langle\xi_j,x \rangle \big)(z)
	=\sum_{j=1}^{d}(a\cdot\xi_j)(z)\langle\xi_j,x\rangle(h^m(z))\\
	&=\sum_{j=1}^{d}a(z)\xi_j(z)\sum_{h^{m}(w)=h^{m}(z)}\overline{\xi_j(w)}x(w).
	\end{align*}	
	Since $h^{m}$ is injective on each $\supp \xi_j$, 
	\begin{align*}
	\Big(\sum_{j=1}^{d}\Theta_{a\cdot\xi_j,\xi_j}(x)\Big)(z)&=	\sum_{j=1}^{d}a(z)\xi_j(z)\overline{\xi_j(z)}x(z)
	=a(z)x(z)\sum_{j=1}^{d}|\xi_j(z)|^2
	=a(z)x(z).
	\end{align*}	
	Thus $\sum_{j=1}^{d}\Theta_{a\cdot\xi_j,\xi_j}=\varphi_m(a)$.  Since the left action on each fibre is by compact operators, \cite[Proposition 5.8]{Fo} implies  that $X$ is a compactly aligned product system over $\N^k$.
\end{proof}

\section{KMS states and the subinvariance relation}

Suppose that $\phi$ is a KMS state on the Toeplitz algebra of a directed graph or a $k$-graph, and that $\{p_v\}$ are the projections associated to the vertices of the graph. We know from \cite{aHLRS1, aHLRS} that the KMS condition implies that the vector $\big(\phi(p_v)\big)$ satisfies an inequality, which is known in Perron-Frobenius theory as a subinvariance relation. The constructions of KMS states in \cite{aHLRS1, aHLRS} depend crucially on being able to find the general solution of this relation. By drawing an analogy between the continuous systems of \cite{AaHR} and those in \cite{aHLRS1}, we found a similar inequality in \cite[\S4]{AaHR}, and were again able to describe the general solution \cite[Proposition~4.2]{AaHR}. Here we find and solve a similar relation by  analogy with the situation for $k$-graphs in \cite[\S4]{aHLRS}.

Throughout this section we consider a family $\{h_i:1\leq i\leq k\}$ of commuting surjective  local homeomorphisms on a  compact space $Z$, and define $h^n:=\prod_{i=1}^kh_i^{n_i}$ for $n\in \N^k$. We denote by $X$ the associated product system, as in Proposition~\ref{prop1}. For each finite Borel measure $\nu$ on $Z$ and each $n\in \N^k$, we define another measure $R^n(\nu)$ by
\begin{align}\label{defR}
	\int a\,d\big(R^n(\nu)\big)
	=\int\sum_{h^n(w)=z}a(w)\,d\nu(z)\quad\text{for } a\in C(Z).
	\end{align}
We then have $R^m\circ R^n=R^{m+n}=R^n\circ R^m$. By using the same formula for signed measures, we can view the $R^n$ as bounded linear operators on the dual space $C(Z)^*$ (see \cite[Theorem~7.17]{F}, for example). This is helpful when we want to make sense of infinite sums involving these operations, which we can do by showing that the sums are absolutely convergent in the Banach space of bounded operators on $C(Z)^*$. 

With this notation we have the following analogue of \cite[Proposition~4.1]{aHLRS}. Recall that we write $\psi$ for the canonical representation of $X$ in $\NT(X)$.

\begin{prop}\label{suninvprop}
Let  $r\in (0,\infty)^k$ and suppose that $\alpha : \R\rightarrow \Aut\NT(X)$ is given in terms of the  gauge action by $\alpha_t=\gamma_{e^{ i t r}}$.
 Suppose that  $\phi$ is a KMS$_\beta$ state  of $(\NT(X),\alpha)$, and $\mu$ is the probability measure on $Z$ such that $\phi(\psi_0(a))=\int a\,d\mu$ for all $a\in C(Z)$. Then for every subset $K$ of  $\{1,\dots, k\}$, 
	\begin{equation}\label{subinvariance}
\int a\,d\Big(\prod_{i\in K}\big(1-e^{-\beta r_i}R^{e_i}\big)\mu\Big)\geq 0 \quad\text{for every positive $a\in C(Z)$.}
\end{equation}
We will call this relation the \textit{subinvariance relation}.
\end{prop}

For the calculations in the proof, we need a lemma.
\begin{lemma}\label{lemmapositive}
 Let $T$ be the Fock representation of $X$ and let $n \in \N^k$. Suppose that $\{\rho_l:1\leq l\leq d\}$ is a  partition of unity  such that $h^n$ is injective on each $\supp \rho_l$, and define  $\tau_l:=\sqrt{\rho_l}$. Then for $x\in X_m$ we have
 \begin{equation*}
 \sum_{l=1}^{d}T_n(\tau_l)T_n(\tau_l)^*(x)=
 \begin{cases} x&\text{if $m\geq n$}\\
 0&\text{otherwise.}
 \end{cases}
 \end{equation*}
\end{lemma}

\begin{proof}
If $m\ngeq n$, then the adjoint formula~\eqref{fockadjoint}  for the Fock representation implies that  $T_n(\tau_l)T_n(\tau_l)^*(x)=0$ for all $l$. So we suppose that  $m\geq n$. We may suppose that $x=\sigma_{n,m-n}(x'\otimes x'')$ for some $x'\in X_n$ and $x''\in X_{m-n}$. Now we compute:
\begin{align*}
\Big(\sum_{l=1}^dT_{n}(\tau_l) {T_{n}( \tau_l)}^*\Big)	\notag	\big(\sigma_{n,m-n}(x'\otimes x'')\big)
&=\sum_{l=1}^dT_{n}(\tau_l)\big(\langle \tau_l, x'\rangle\cdot  x''\big)\\
\notag	&=\sum_{l=1}^d\sigma_{n,m-n}\big(\tau_l\otimes\langle \tau_l, x'\rangle\cdot  x''\big)\\
&=\sigma_{n,m-n}\Big(\!\sum_{l=1}^d\tau_l\cdot\langle \tau_l, x'\rangle\otimes  x''\Big), 
\end{align*}
which is $\sigma_{n,m-n}(x'\otimes x'')=x$ because $\{\tau_l:1\leq l\leq d\}$ is a Parseval frame for the fibre $X_m$ (by Lemma~\ref{pframe}~(a)).
\end{proof}

\begin{proof}[Proof of Proposition~\ref{suninvprop}]
Let  $a$ be a  positive element of $C(Z)$. 	For $K=\emptyset$, we have $\int a\,d\mu\geq 0$ because $a$ is positive. So we assume $K\neq \emptyset$. Following the calculations in \cite[Proposition~4.1]{aHLRS}, we write
\begin{align}\label{productsmeasurs}
\int a\,d\Big(\prod_{i\in K}(1-e^{-\beta r_i}R^{e_i})\Big)\mu=\sum_{\emptyset \subseteq J\subseteq K}(-1)^{|J|}e^{-\beta r\cdot e_J}\int a\,d\big(R^{e_J}\mu\big).
\end{align}
The $J=\emptyset$ term is $\int a\,d\mu\geq 0$. So we take $\emptyset\subsetneq J\subseteq K$, and, following the proof of \cite[Proposition~4.1]{AaHR}, write the integral in the $J$-summand on the right of \eqref{productsmeasurs} in terms of elements of $\NT(X)$. 

Choose an open cover $\{U_l^J:1\leq l\leq d\}$ of $Z$ such that $h^{e_J}|_{U_l^J}$ is injective and choose a partition of unity $\{\rho_l^J:1\leq l\leq d\}$ subordinate to $\{U_l^J\}$. Define $\tau_l^J:=\sqrt{\rho_l}$.
Since the fibre  $X_{e_J}$ is the graph correspondence $(Z,Z,\id , h^{e_J})$, the calculation in the first paragraph of \cite[Proposition~4.1]{AaHR} in  $X_{e_J}$ gives
\[
\int a\,d\big(R^{e_J}\mu\big)=\phi\Big(\sum_{l=1}^d{\psi_{e_J}( \tau_l^J)}^*\psi_{e_J}(a\cdot\tau_l^J)\Big).
\]
Now we apply the KMS relation to get		
\begin{align*}
\int a\,d\big(R^{e_J}\mu\big)&=e^{\beta r\cdot e_J}\phi\Big(\sum_{l=1}^d\!\psi_{e_J}(a\cdot\tau_l^J) {\psi_{e_J}( \tau_l^J)}^*\!\Big)\\
&=e^{\beta r\cdot e_J}\phi\Big(\sum_{l=1}^d\psi_0(a)\psi_{e_J}(\tau_l^J) {\psi_{e_J}( \tau_l^J)}^*\Big).\notag
\end{align*}
Now adding up over $J$ gives 
	\begin{align}
	\eqref{productsmeasurs}&=\int a\,d\mu+\sum_{\emptyset \subsetneq J\subseteq K}(-1)^{|J|}\phi\Big(\sum_{l=1}^d\psi_0(a)\psi_{e_J}(\tau_l^J) {\psi_{e_J}( \tau_l^J)}^*\Big)\notag\\
	&=\phi\Big(\psi_0(a)+\sum_{\emptyset \subsetneq J\subseteq K}(-1)^{|J|}\sum_{l=1}^d\psi_0(a)\psi_{e_J}(\tau_l^J) {\psi_{e_J}( \tau_l^J)}^*\Big).\label{6.3}
	\end{align}

We need to show that  the right-hand side of \eqref{6.3} is positive. Since $\phi$ is  a state, and the Fock representation $T_*$ is faithful, it suffices to prove that
\begin{align}\label{Fock}
	T_0(a)+\sum_{\emptyset \subsetneq J\subseteq K}(-1)^{|J|}T_0(a)\sum_{l=1}^dT_{e_J}(\tau_l^J) {T_{e_J}( \tau_l^J)}^*
	\end{align}
is positive on each summand $X_n$ of the Fock module. For $n=0$, the sum collapses to $T_0(a)$, which is positive because $a$ is. So we take $n\in \N^k$ such that $I_n:=\{i:n_i\not=0\}$ is nonempty, and $x\in X_n$. Lemma~\ref{lemmapositive} implies that the $J$-summand in \eqref{Fock} vanishes on $X_n$ for $n\ngeq e_J$. So only the summands with $n\geq e_J$ survive, and $n\geq e_J$ is equivalent to $J\subset I_n$. So Lemma~\ref{lemmapositive} says that
	\begin{align*}
	\Big(T_0(a)+&\sum_{\emptyset \subsetneq J\subseteq K}(-1)^{|J|}T_0(a)\sum_{l=1}^dT_{e_J}(\tau_l^J) {T_{e_J}( \tau_l^J)}^*\Big)(x)\\
&=T_0(a)(x)+\sum_{\emptyset \subsetneq J\subseteq I_n\cap K}(-1)^{|J|}T_0(a)(x)\\
	&=\sum_{\emptyset \subset J\subseteq I_n\cap K}(-1)^{|J|}T_0(a)(x).
	\end{align*}
If $I_n\cap K=\emptyset$, then the sum is absent, and $\eqref{Fock}=T_0(a)$ on $X_n$, which is positive. If $I_n\cap K\not=\emptyset$, the number of subsets of $I_n\cap K$ with odd cardinality equals the number of subsets with even cardinality, and the sum vanishes.	Thus \eqref{Fock} is positive on all summands of the Fock module, and therefore it is a positive operator  on $F(X)$.  Thus \eqref{Fock} is a positive operator, and \eqref{6.3} is nonnegative. This completes the proof of Proposition~\ref{suninvprop}.
\end{proof}

The next proposition characterises the solutions of the subinvariance relation \eqref{subinvariance}. It is a generalisation of   \cite[Proposition~4.2]{AaHR} and \cite[Theorem~6.1(a)]{aHLRS}.
\begin{prop}\label{series}
For each  $ 1\leq i\leq k$, set
	\begin{equation}\label{criticalpoint}
	\beta_{c_i}:=\limsup_{j\to \infty}\Big(j^{-1}\ln
	\Big(\max_{z\in Z}|h_i^{-j}(z)|\Big)\Big).
	\end{equation}
Let $r\in (0,\infty)^k$, and suppose  that $\beta\in (0,\infty)$ satisfies $\beta r_i> \beta_{c_i}$ for $1\leq i\leq k$.
		\begin{itemize}
		\item[(a)]  The series $\sum_{n\in \N^k}e^{-\beta r\cdot n}|h^{-n}(z)|$ converges uniformly for $z\in Z$ to a continuous function $f_\beta(z)\geq 1.$
		\item[(b)] Let $\varepsilon$ be a finite regular Borel measure on $Z$. Then  the series $\sum_{n\in \N^k}e^{-\beta r\cdot n}R^n\varepsilon$
		converges in norm in the  dual space $C(Z)^*$ with sum $\mu$, say. Then $\mu$ satisfies  the subinvariance  
		relation \eqref{subinvariance}, and we have 
		\begin{equation}\label{recoverep}
\varepsilon=\prod_{i=1}^{k}\big(1-e^{-\beta r_i}R^{e_i}\big)\mu.
\end{equation}
The measure $\mu$ is a probability measure if and only if $\int f_\beta\,d\varepsilon=1.$
	 \item[(c)] Suppose that $\mu$ is a probability measure which satisfies the subinvariance  relation \eqref{subinvariance}. 
	Then $\varepsilon=\prod_{i=1}^{k}\big(1-e^{-\beta r_i}R^{e_i}\big)\mu$  
	is a finite regular Borel measure  satisfying $\sum_{n\in \N^k}e^{-\beta r\cdot n}R^n\varepsilon=\mu$,
	and we have $\int f_\beta\,d\varepsilon=1.$ 
	\end{itemize} 
	\end{prop}

\begin{proof}	
	For part  (a),	take  $1\leq i \leq k$. Applying the calculation of the first paragraph in the proof of  
	\cite[Proposition~4.2]{AaHR} to the local homeomorphism $h_i$ gives us  $\delta_i\in (0,\infty)$ and $M_i\in \N$ such that 
	\begin{equation*}
	l\in \N, l\geq M_i\Longrightarrow e^{-l\beta r_i}\max_z|h^{-l}_i(z)|< e^{-l\delta_i }.
	\end{equation*}
Then with $M=(M_1,\dots,M_k)$ and $N\in\N^k$, we have 
\begin{align*}
	\sum_{M\leq n\leq N} e^{-\beta r \cdot  n}|h^{-n}(z)|&=\sum_{M\leq n\leq N} 
	e^{-\beta r \cdot  n}\big|\big(h_1^{n_1}\circ\cdots\circ h_k^{n_k}\big)^{-1}(z)\big|\\
&=\sum_{M\leq n\leq N} \prod_{i=1}^k \Big(e^{-\beta r_i\cdot n_i} \max_z\big|h_i^{-n_i}(z)\big|\Big)\\
	&\leq\prod_{i=1}^k\sum_{M_i\leq n_i\leq N_i} \Big( e^{-\beta r_i\cdot n_i} \max_z\big|h_i^{-n_i}(z)\big|\Big)\\
	&\leq  \prod_{i=1}^k \sum_{M_i\leq n_i\leq N_i} e^{-\delta_i n_i}.\notag
	\end{align*}
	Now let $N\rightarrow\infty$ in $\N^k$, in the sense that $N_i\rightarrow \infty$ for $1\leq i \leq k$.  Since each series $\sum_{n_i=M_i}^\infty e^{-\delta_i n_i}$ is convergent, it follows that
$\sum_{ n=M}^\infty e^{-\beta r \cdot  n}|h^{-n}(z)|$ converges uniformly for $z\in Z$. Since $h^n$ is a local homeomorphism on $Z$ for all $n\in \N^k$,  \cite[Lemma~2.2]{BRV} implies that  $z\mapsto |h^{-n}(z)|$ is locally constant, and hence continuous. Thus   
\[
f_\beta(z):=\sum_{n\in \N^k} e^{-\beta n}|h^{-n}(z)|
\]
is the uniform limit of a sequence of continuous functions, and is therefore continuous. The term corresponding to $n=0$ is $1$, so $f_\beta\geq 1$.

	For part (b),  let $\varepsilon$ be a finite regular Borel measure on $Z$.	Take $M$ and $\{\delta_i:1\leq i\leq k\}$ as in part (a). We want to   show that
	$\sum_{n\geq M}e^{-\beta r\cdot n}R^n\varepsilon$ converges in the norm of  $C(Z)^*$. We calculate the $N$-th partial sum  using the formula
	\eqref{defR} for $R^n$. Let $g\in C(Z)$, we have	
	\begin{align}\label{YYYY}
	\Big|\sum_{M\leq n\leq N} e^{-\beta r\cdot n}\int g\,d(R^n\varepsilon)\Big|
	\notag&=\Big|\sum_{M\leq n\leq N}e^{-\beta r\cdot n}\int \sum_{h^n(w)=z}g(w)\,d\varepsilon(z)\Big|\\
	\notag&\leq \sum_{M\leq n\leq N} e^{-\beta r\cdot n} |h^{-n}(z)|\,\|\varepsilon\|_{C(Z)^*}\|g\|_\infty\\
	&\leq\|\varepsilon\|_{C(Z)^*}\|g\|_\infty \prod_{i=1}^{k}\sum_{M_i\leq n_i\leq N_i} e^{-\delta_i n_i}.
	\end{align}
	Now when $N\rightarrow \infty$,  all the series $\sum_{ n_i=M_i}^{\infty} e^{-\delta_i n_i}$ are convergent and hence the series $\sum_{n\in \N^k} e^{-\beta r\cdot n}R^n\varepsilon$ converges absolutely in the norm of $C(Z)^*$ with sum a functional $f$, say. The formula  \eqref{defR} for $R^n$ implies that   $f$  is  positive functional on $C(Z)$, and by the
	Riesz representation theorem is given by a  Borel  measure $\mu$ on $Z$.

	  Rearranging the absolutely convergent series that defines $\mu$, we find
	\begin{align}\label{arrangemu}
\mu&=	\sum_{n\in \N^k} e^{-\beta r\cdot n}R^n\varepsilon
=\sum_{n\in \N^k} e^{-\beta r\cdot n}\Big(\prod_{i=k}^{n_i}R^{n_ie_i}\varepsilon\Big)\\
&=\prod_{i=1}^k\Big(\sum_{n_i=0}^\infty e^{-\beta r_in_i}R_i^{n_i}\varepsilon\Big)\notag
\end{align}
with $R_i:=R^{e_i}$. Applying the arguments in the proof of \cite[Proposition~4.2(b)]{AaHR} to the graph $(Z,Z,\id,h_i)$ shows that each series $\sum_{n_i=0}^\infty e^{-\beta r_in_i}R_i^{n_i}\varepsilon$ converges in norm in $C(Z)^*$ to a positive measure $Q_i\varepsilon$, and that we then have
\begin{equation}\label{quotefromk=1}
(1-e^{-\beta r_i}R_i)Q_i\varepsilon=\varepsilon.
\end{equation}
Equation \eqref{arrangemu} says that $\mu=\prod_{i=1}^k Q_i\varepsilon$, and hence the following calculation using $k$ applications of \eqref{quotefromk=1} gives \eqref{recoverep}:
\begin{align*}
\prod_{i=1}^{k}\big(1-e^{-\beta r_i}R^{e_i}\big)\mu&=\Big(\prod_{i=1}^{k}\big(1-e^{-\beta r_i}R^{e_i}\big)\Big)\Big(\prod_{i=1}^k Q_i\varepsilon\Big)\\
&=\Big(\prod_{i=1}^k\big(1-e^{-\beta r_i}R^{e_i}\big)Q_i\Big)\varepsilon=\varepsilon.
\end{align*}

To see that $\mu$ satisfies the  subinvariance  relation \eqref{subinvariance}, we take $K\subseteq \{1,\dots,k\}$ and use \eqref{quotefromk=1} again:
\begin{align*}
\prod_{i\in K}\big(1-e^{-\beta r_i}R^{e_i}\big)\mu&=\Big(\prod_{i\in K}\big(1-e^{-\beta r_i}R^{e_i}\big)\Big)\Big(\prod_{i=1}^k Q_i\varepsilon\Big)\\
&=\Big(\prod_{i\notin K}Q_i\Big)\Big(\prod_{i\in K}\big(1-e^{-\beta r_i}R^{e_i}\big)Q_i\Big)\varepsilon\\
&=\Big(\prod_{i\notin K}Q_i\Big)\varepsilon\geq 0.
\end{align*}

	To see the relation between $\mu$ and $f_\beta$, we compute using  the formula  \eqref{defR} for $R^n$ and then Tonelli's theorem:
	\begin{align*}
	\mu(Z)&=\sum_{n\in \N^k} e^{-\beta r\cdot n}(R^n\varepsilon)(Z)
	=\sum_{n\in \N^k} e^{-\beta r\cdot n}\int 1\,d(R^n\varepsilon)\\
	&=\sum_{n\in \N^k} e^{-\beta r\cdot n}\int |h^{-n}(z)|\,d\varepsilon(z)
	=\int \sum_{n\in \N^k} e^{-\beta r\cdot n}|h^{-n}(z)|\,d\varepsilon(z)\\
	&=\int f_{\beta}(z)\,d\varepsilon(z).
	\end{align*}
Thus $\mu$ is finite, and is  a probability measure if and only if $\int f_\beta\,d\varepsilon=1$. This completes the proof of part (b).
	
For (c), suppose that $\mu$ is a probability measure which satisfies the subinvariance  relation \eqref{subinvariance}. 
Set $\varepsilon=\prod_{i=1}^{k}\big(1-e^{-\beta r_i}R^{e_i}\big)\mu$. 
Then $\varepsilon$  is a bounded functional on $C(Z)$, and the subinvariance relation \eqref{subinvariance} for $\mu$ and  $K=\{1,\dots,k\}$ says that $\varepsilon$ is positive. 
	
	 To check that
	$\sum_{n\in \N^k}e^{-\beta r\cdot n}R^n\varepsilon=\mu$,	
	we calculate as in \eqref{arrangemu}:
	\begin{align*} 
	\sum_{0\leq n\leq N}e^{-\beta r\cdot n}R^n\varepsilon
	\notag&=\Big(\sum_{0\leq n\leq N}e^{-\beta r\cdot n}R^n\prod_{i=1}^{k}\big(1-e^{-\beta r_i}R^{e_i}\big)\Big)\mu\\
	\notag&=\Big(\sum_{0\leq n\leq N}\prod_{i=1}^{k}e^{-\beta r_i n_i}R^{n_ie_i}\big(1-e^{-\beta r_i}R^{e_i}\big)\Big)\mu\\
	&=\Big(\prod_{i=1}^{k}\sum_{n_i=0}^{N_i}\big(e^{-\beta r_i n_i}R^{n_ie_i}-e^{-\beta r_i(n_i+1)}R^{(n_i+1)e_i}\big)\Big)\mu.\notag
	\end{align*}
Observe that for each $i$, the sum over $n_i$ above is telescoping. Now let $N\to\infty$ in the sense that $N_i\to\infty$ for $1\leq i\leq k$. Then
\[
\Big(\lim_{N_k\to\infty}\sum_{n_i=0}^{N_i}\big(e^{-\beta r_i n_i}R^{n_ie_i}-e^{-\beta r_i(n_i+1)}R^{(n_i+1)e_i}\big)\Big)\mu=\mu.
\]
Repeating for  $i=k-1,\dots,1$ we get that $\sum_{n\in \N^k}e^{-\beta r\cdot n}R^n\varepsilon=\mu$, as needed. That $\int f_\beta\,d\varepsilon=1$ now follows from part (b). 
\end{proof}

\section{A characterisation of KMS states}\label{sec4}

Our goal is to construct KMS states on the Nica-Toeplitz algebra $\NT(X)$ of the previous sections, and we need to be able to recognise when a given state is a KMS state. We can of course verify the KMS condition on pairs of elements $b,c$ of the form $\psi_m(x)\psi_n(y)^*$, but this involves working with two products $bc$ and $cb$ of such elements. We seek a characterisation of KMS states which can be verified on a single spanning element, like those of \cite[Proposition~3.1]{AaHR} and \cite[Proposition~3.1]{aHLRS}. Proving that our characterisation works involves simplifying a product of the form $\psi_m(x)^*\psi_n(y)$ using suitably chosen Parseval frames for the Hilbert bimodules in our product (see Proposition~\ref{solelprop}). Our construction of such frames uses the $*$-commuting property.

So we assume from now on that we have a family $\{h_i:1\leq i\leq k\}$ of $*$-commuting local homeomorphisms on a compact space $Z$. We consider the product system $X$ of Proposition~\ref{prop1}. We fix $r\in (0,\infty)^k$ and define $\alpha:\R\to \Aut\NT(X)$ in terms of the gauge action by $\alpha_t=\gamma_{e^{ i t r}}$. We begin by stating our characterisation of KMS states on $(\NT(X),\alpha)$, though proving it will take the rest of the section. Notice that the hypothesis of rational independence in part (b) already featured in \cite[Proposition~3.1]{aHLRS}.

\begin{prop}\label{KMSPROP}
Suppose that  $\beta>0$ and  $\phi$ is a state on $\NT(X)$.
	\begin{itemize}
		\item[(a)] If $\phi$  satisfies
		\begin{align}\label{KMSCHARAC}
		\phi\big(\psi_m(x)\psi_n(y)^*\big)=\delta_{m,n} e^{-\beta r\cdot m}\phi\circ\psi_0(\langle y,x\rangle)
		\text{ for } x\in X_m, y\in X_n,
		\end{align}
		then  $\phi$  is a KMS$_\beta$ state of $(\NT(X),\alpha)$.
		\item[(b)] If $\phi$ is a KMS$_\beta$ state of $(\NT(X),\alpha)$ and $r\in (0,\infty)^k$ has rationally independent coordinates,
		then $\phi$ satisfies \eqref{KMSCHARAC}.
	\end{itemize}
\end{prop}

We now describe the Parseval frames that we will use. As usual, they have the form $\tau_i=\sqrt{\rho_i}$ for some carefully chosen partition $\{\rho_i\}$ of unity of $Z$.

\begin{lemma}\label{pframe}
	Let $f,g$ be $*$-commuting local homeomorphisms on  $Z$ giving graph correspondences $X(E_1)$, $X(E_2)$. Let  $\{\rho_i:1\leq i\leq d\}$ be a  partition of unity  such that $f$ and $g$ are one-to-one on each $\supp \rho_i$, and take $\tau_i:=\sqrt{\rho_i}$. Then
	\begin{itemize}
		\item[(a)] $\{\tau_i\}$,$\{\tau_i\circ g\}$ are  Parseval frames for $X(E_1)$;
		\item[(b)] 	$\{\tau_i\}$,$\{\tau_i\circ f\}$ are Parseval frames for $X(E_2)$; and
		\item[(c)] 	 there is an isomorphism
		$t_{f,g}:X(E_1)\otimes_A X(E_2)\rightarrow X(E_2)\otimes_A X(E_1)$ such that
		\begin{align}\label{flipformula}
		t_{f,g}(\tau_i\circ g\otimes \tau_j)=\tau_j\circ f\otimes\tau_i \text{ for } 1\leq i,j\leq d.
		\end{align}
		We call this  isomorphism the  flip map.
	\end{itemize}	
\end{lemma}

\begin{proof}
Part (b) follows from (a), so we only prove (a) and (c). It is easy to check  that
	$\{\tau_i\}$ is  a Parseval frame for $X(E_1)$ (see, for example, \cite[Proposition~8.2]{EV}).  To see that  $\{\tau_i\circ g\}$ is
	a Parseval frame for $X(E_1)$, we take $x\in X(E_1)$ and check the reconstruction formula. For $z\in Z$, we have
	\begin{align}\label{UFF}
	\sum_{i=1}^{d} (\tau_i\circ g)\cdot\big\langle (\tau_i\circ g),x\big\rangle(z)
	&=\sum_{i=1}^{d} \tau_i\big( g(z)\big)\Big(\sum_{ f(w)=f(z)}\overline{\tau_i(g(w)) }x(w)\Big).
	\end{align}
The  $i$-summand vanishes  unless both $g(z),g(w)$ are in $\supp{\tau_i}$, so we suppose that  $g(z),g(w)\in \supp{\tau_i}$. For $f(w)=f(z)$, we have
	 $g\circ f(z)=g\circ f(w)$ and hence $f\circ g(z)=f\circ g(w)$; since $f$ is one-to-one on  $\supp{\tau_i}$, we have $g(w)=g(z)$. Thus both $w$ and $z$ fit in the box in the diagram

	
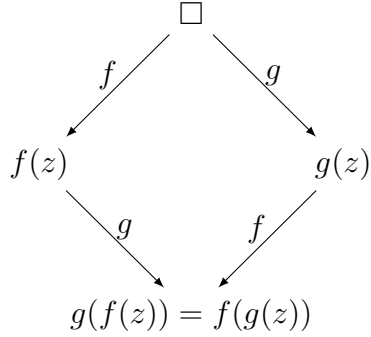
\begin{tikzpicture}[scale=0.2]
	\node  at (0,0) {};
	\node (a) at (35,20) {$\Box$};
	\node (b) at (25,10) {$f(z)$};
	\node (c) at (45,10) {$g(z)$};
	\node (d) at (35,0) {$g(f(z))=f(g(z))$};
	\draw[-latex, black](a) to node[pos=0.4,left] {$f$} (b);
	\draw[-latex, black](c) to node[pos=0.4,left] {$f$} (d);
	\draw[-latex, black ](b) to node[pos=0.4,right] {$g$} (d);
	\draw[-latex, black](a) to node[pos=0.4,right] {$g$} (c);
	\end{tikzpicture}
		
\noindent		and $*$-commutativity of $f$ and $g$  implies that $w=z$.
	Thus  the interior sum in \eqref{UFF} collapses to $\overline{(\tau_i\circ g)(z)}x(z)$, and we recover the reconstruction formula  
	\[\sum_{i=1}^{d} (\tau_i\circ g)\cdot \langle (\tau_i\circ g),x\rangle(z)=\sum_{i=1}^{d}\overline{\tau_i(g(z))} \tau_i(g(z))x(z)
	=x(z)\sum_{i=1}^{d} \tau_i(g(z))^2
	=x(z).\]
	
	For part (c), we take isomorphisms $\sigma_{f,g}$ and $\sigma_{g,f}$ as in Lemma~\ref{coresiso}, and  set $t_{f,g}:=\sigma_{g,f}^{-1}\,\circ\, \sigma_{f,g}$, which is  an isomorphism of  $X(E_1)\otimes_A X(E_2)$ onto $X(E_2)\otimes_A X(E_1)$.
	To check \eqref{flipformula}, note that
	\begin{align*}
	\sigma_{f,g}(\tau_i\circ g\otimes \tau_j)=\sigma_{g,f}(\tau_j\circ f\otimes\tau_i), 
	\end{align*}
and hence
\[t_{f,g}(\tau_i\circ g\otimes \tau_j)=\sigma_{g,f}^{-1}\,\circ\, \sigma_{f,g}(\tau_i\circ g\otimes \tau_j)=\tau_j\circ f\otimes\tau_i.	
\qedhere\]
\end{proof}

To study the KMS sates, similar  results in the literature (for example \cite[Proposition~3.1]{aHLRS} and \cite[Theorem 4.6]{HLS}) show that it  is  crucial  to express elements of the form $\psi_{n}(y)^*\psi_{m}(x)$ in terms of usual spanning elements $\psi_{p}(s)\psi_{q}(t)^*$ of  the algebra $\NT(X)$.  For a general product system over a semigroup,  Fowler provided an approximation \cite[Proposition~5.10]{Fo}, but this is not enough because we need an exact formula; in the
dynamics associated to a $k$-graph \cite{aHLRS} this formula already exists as one of the Toeplitz-Cuntz-Krieger relations; in \cite{HLS}, since each fibre in the product system has an orthonormal basis, it is easier to find such a formula (see \cite[Lemma~4.7]{HLS}). The next proposition provides such a formula for our product system. 
\begin{prop}\label{solelprop}
 Take $m,n\in \N^k$ such that $m\wedge n=0$.  Let  $\{\rho_i:1\leq i\leq d\}$ be a  partition of unity  such that $h^m$ and $h^n$ are one-to-one on $\supp \rho_i$ and take  $\tau_i:=\sqrt{\rho_i}$.
\begin{itemize}	
\item[(a)] Let $\sigma_{m,n}:X_m\otimes_A X_n\rightarrow X_{m+n}$ be the isomorphism induced by the  multiplication in $X$, and similarly for $\sigma_{n,m}$. Then for $x\in X_m$ and $y\in X_n$, we have
\begin{align}\label{HELPFORMULA}
\sigma_{m,n}(x\otimes y)=\sum_{i,j=1}^d\sigma_{n,m}\big(\tau_j\circ h^m \otimes \tau_i \big)\cdot \big\langle \langle x,\tau_i\circ h^n \rangle\cdot\tau_j\, ,\, y\big\rangle.
\end{align}
\item[(b)] For $x\in X_m$ and $y\in X_n$,  we  have
	\begin{align}\label{FORMULA}
	\psi_{n}(y)^*\psi_{m}(x)= \sum_{i,j=1}^d\psi_m\big(\langle y,\tau_j\circ h^m \rangle\cdot\tau_i\big)\psi_n\big(\langle x,\tau_i\circ h^n \rangle\cdot\tau_j\big)^*.
	\end{align}
\end{itemize}	
\end{prop}

\begin{proof}
	For part (a), since $m\wedge n=0$, $h^m$ and $h^n$ are $*$-commuting. Then Lemma~\ref{pframe} implies that $\{\tau_i\circ h^n\}$ and $\{\tau_j\}$ are Parseval frames for $X_m$ and $X_n$. The formula for multiplication in $X$ implies that
\begin{align}\label{tazetaze}
\sigma_{m,n}(\tau_i\circ h^n\otimes \tau_j)=\sigma_{n,m}(\tau_j\circ h^m\otimes\tau_i).
\end{align}
To prove \eqref{HELPFORMULA}, we write $x\otimes y$ in terms of  the elements $\{\tau_i\circ h^n\otimes  \tau_j\}$.
The reconstruction formulas for the Parseval frames $\{\tau_i\circ h^n\}$ and $\{\tau_j\}$ give
\begin{align*}
x\otimes y
&=\Big(\sum_{i=1}^d\tau_i\circ h^n\cdot \big\langle\tau_i\circ h^n,x\big\rangle\Big)\otimes \Big(\sum_{j=1}^d\tau_j\cdot \langle\tau_j,y\rangle\Big).
\end{align*}
Since the tensors are balanced, we have
\begin{align}\label{taze}
x\otimes y
&= \sum_{i,j=1}^d \Big(\tau_i\circ h^n\otimes\langle\tau_i\circ h^n,x\rangle\cdot \tau_j\cdot \langle\tau_j,y\rangle\Big).
\end{align}

We then claim that
\begin{align}\label{quick}
\big\langle\tau_i\circ h^n,x\big\rangle\cdot\tau_j\cdot \langle\tau_j,y\rangle=\tau_j\cdot  \big\langle \langle x,\tau_i\circ h^n \rangle\cdot\tau_j\, ,\, y\big\rangle.
\end{align}
To justify the claim, we evaluate both sides of \eqref{quick} on $z\in Z$. For the left-hand side,  a computation  in the fibre $X_n$ shows that
\begin{align*}
\big(\langle\tau_i\circ h^n,x&\rangle\cdot\tau_j\cdot \langle\tau_j,y\rangle\big)(z)=	\langle\tau_i\circ h^n,x\rangle(z)\tau_j(z) \langle\tau_j,y\rangle(h^n(z))\\
&=\langle\tau_i\circ h^n,x\rangle(z)\tau_j(z)\sum_{h^n(w)=h^n(z)}\overline{\tau_j(w)} y(w)\\
&=\langle\tau_i\circ h^n,x\rangle(z)\tau_j(z)\overline{\tau_j(z)} y(z)\qquad (h^n\text{ is injective on } \supp \tau_j).
\end{align*}
A similar calculation for  the right-hand side of \eqref{quick} gives
\begin{align*}
\big(\tau_j\cdot  \big\langle \langle x,\tau_i\circ &h^n \rangle\cdot\tau_j\, ,\, y\big\rangle\big)(z)=	\tau_j(z) \big\langle \langle x,\tau_i\circ h^n \rangle\cdot\tau_j\, ,\, y\big\rangle(h^n(z))\\
&=	\tau_j(z) \sum_{h^n(w)=h^n(z)}\overline{\langle x,\tau_i\circ h^n \rangle(w)\tau_j(w)} y(w)\\
&=	\tau_j(z)\langle \tau_i\circ h^n,x \rangle(z)\overline{\tau_j(z)} y(z),
\end{align*}
and we have proved the claim.

Now putting \eqref{quick} in \eqref{taze} gives
\begin{align*}
x\otimes y
&= \sum_{i,j=1}^d \big(\tau_i\circ h^n\otimes\tau_j\cdot  \big\langle \langle x,\tau_i\circ h^n \rangle\cdot\tau_j\, ,\, y\big\rangle\big).
\end{align*}
Since $\sigma_{m,n}$ is an isomorphism of correspondences, we have
\begin{align*}
\sigma_{m,n}(x\otimes y)
&= \sum_{i,j=1}^d \sigma_{m,n}(\tau_i\circ h^n\otimes\tau_j)\cdot  \big\langle \langle x,\tau_i\circ h^n \rangle\cdot\tau_j\, ,\, y\big\rangle,
\end{align*}
and applying \eqref{tazetaze} completes the proof of \eqref{HELPFORMULA}.

	For part (b),	we use  the Fock representation $T$ of $X$, and
	it suffices for us to prove that
	\begin{align}\label{FORMULAFOCK}
	T_{n}(y)^*T_{m}(x)= \sum_{i,j=1}^dT_m\big(\langle y,\tau_j\circ h^m \rangle\cdot\tau_i\big)T_n\big(\langle x,\tau_i\circ h^n \rangle\cdot\tau_j\big)^*.
	\end{align}

	To do this, we show that both sides of \eqref{FORMULAFOCK} agree on   each  summand $X_p$ of the Fock module. Let $p\in \N^k, s\in X_p$.
The formula  \eqref{fockadjoint} for the adjoint  shows that the right-hand side of \eqref{FORMULAFOCK} vanishes unless $p\geq n$. The
	left-hand side satisfies
	$\big(T_{n}(y)^*T_{m}(x)\big)(s)=T_{n}(y)^*\big(\sigma_{m,p}(x\otimes s)\big)$, which vanishes unless $m+p\geq n$. Since $m\wedge n=0$, $m+p\geq n$ is equivalent to $p\geq n$. Thus both sides of \eqref{FORMULAFOCK} are zero unless  $p\geq n$.    So  we assume $p\geq n$.
	
	It suffices to check \eqref{FORMULAFOCK} on  $s=\sigma_{n,p-n}(s'\otimes s'')$ where $s'\otimes s''\in X_n\odot_A X_{p-n}$. We first compute
	the right-hand side of \eqref{FORMULAFOCK} using \eqref{fockadjoint}:
	\begin{align}\label{FOCKRIGHT}
	\sum_{i,j=1}^{ d}T_m\big(\langle y,\notag&\tau_j\circ h^m \rangle\cdot\tau_i\big)T_n\big(\langle x,\tau_i\circ h^n \rangle\cdot\tau_j\big)^*(\sigma_{n,p-n}(s'\otimes s''))\\
	\notag&=\sum_{i,j=1}^{ d}T_m\big(\langle y,\tau_j\circ h^m \rangle\cdot\tau_i\big)\big(\big\langle \langle x,\tau_i\circ h^n \rangle\cdot\tau_j\, ,\, s'\big\rangle\cdot s''\big)\\
	&=\sum_{i,j=1}^{ d}\sigma_{m,p-n}\big(\langle y,\tau_j\circ h^m \rangle\cdot\tau_i\otimes \big\langle \langle x,\tau_i\circ h^n \rangle\cdot\tau_j\, ,\, s'\big\rangle\cdot s''\big).
	\end{align}	
For the left-hand side of \eqref{FORMULAFOCK}, we evaluate $\dagger:=T_{n}(y)^*T_{m}(x)\big(\sigma_{n,p-n}(s'\otimes s'')\big)$.  We start by applying the definition of the Fock representation. Then
\begin{align*}
\dagger&=T_{n}(y)^*\big(\sigma_{m,p}\big(x\otimes \sigma_{n,p-n}(s'\otimes s'')\big)\big)\\
&=T_{n}(y)^*\big(\sigma_{m+n,p-n}\big(\sigma_{m,n}(x\otimes s')\otimes s''\big)\big).
\end{align*}
Part (a) gives
\begin{align*}
\sigma_{m,n}(x\otimes s')\otimes s''&= \sum_{i,j=1}^{d}\big(\sigma_{n,m}(\tau_j\circ h^m \otimes \tau_i )\cdot \big\langle \langle x,\tau_i\circ h^n \rangle\cdot\tau_j,s'\big\rangle\big)\otimes s''\\
&= \sum_{i,j=1}^{d}\sigma_{n,m}(\tau_j\circ h^m \otimes \tau_i )\otimes \big\langle \langle x,\tau_i\circ h^n \rangle\cdot\tau_j\, ,\, s'\big\rangle\cdot s'',
\end{align*}
and associativity of the multiplication in $X$ gives
\begin{align*}	
\dagger
\notag&= \sum_{i,j=1}^{d} T_{n}(y)^*\big(\sigma_{n,m+p-n}\big(\tau_j\circ h^m \otimes \sigma_{m,p-n}\big(\tau_i \otimes \big\langle \langle x,\tau_i\circ h^n \rangle\cdot\tau_j\, ,\, s'\big\rangle\cdot s''\big)\big)\\
&= \sum_{i,j=1}^{d} \big\langle y, \tau_j\circ h^m\big\rangle \cdot \sigma_{m,p-n}\big(\tau_i \otimes \big\langle \langle x,\tau_i\circ h^n \rangle\cdot\tau_j\, ,\, s'\big\rangle\cdot s''\big)\\
&= \sum_{i,j=1}^{d}  \sigma_{m,p-n}\big(\langle y, \tau_j\circ h^m\rangle \cdot\tau_i \otimes \big\langle \langle x,\tau_i\circ h^n \rangle\cdot\tau_j\, ,\, s'\big\rangle\cdot s''\big),
		\end{align*}	
which is \eqref{FOCKRIGHT}. Thus we have \eqref{FORMULAFOCK}, and the  injectivity of $T_*$ gives \eqref{FORMULA}.	 
\end{proof}

\begin{lemma}\label{application}
 Suppose that $x\in X_m$, $y\in X_n$, $s\in X_p$, and $t\in X_q$. Then  there exist $\{\xi_{i,j}:1\leq i,j\leq d\}\subset X_{m+p-n\wedge p}$ and $\{\eta_{i,j}:1\leq i,j\leq d\}\subset X_{n+q-n\wedge p}$ such that
	\begin{align}\label{blabla}
	\psi_m(x)\psi_n(y)^*\psi_p(s)\psi_q(t)^*
	=\sum_{i,j=1}^{d}\psi_{m+p-n\wedge p}(\xi_{i,j})\psi_{n+q-n\wedge p}(\eta_{i,j})^*.
	\end{align}
\end{lemma}

\begin{proof}
	Let $N:=n-n\wedge p$ and $P:=p-n\wedge p$. It suffices to prove \eqref{blabla} for  $y=\sigma_{n\wedge p,N}(y''\otimes y')$  and  $s=\sigma_{n\wedge p,P} (s''\otimes s')$.  A calculation using (P2) and (T3) shows that
	\begin{align}\label{coprime}
	\psi_n(y)^*\psi_p(s)&=\psi_{N}(y')^* \psi_{n\wedge p}(y'')^*\psi_{n\wedge p}(s'')\psi_{P}(s')\\
	\notag&=\psi_{N}(y')^* \psi_0(\langle y'',s''\rangle)\psi_{P}(s')\\
	&=\psi_{N}(y')^*\psi_{P}(\langle y'',s''\rangle\cdot s' ).\notag
	\end{align}
Choose an open cover $\{U_i:1\leq i\leq d\}$ of $Z$ such that ${h^N}|_{U_i}$ and ${h^P}|_{U_i}$ are injective, choose a partition of unity $\{\rho_i\}$ subordinate to $\{U_i\}$, and define  $\tau_i:=\sqrt{\rho_i}$. Since $N\wedge P=0$,  Proposition~\ref{solelprop} implies that
	\begin{align*}	
	\psi_{N}(y')^*\psi_{P}\big(\langle y'',s''\rangle\cdot s' \big)=\sum_{i,j=1}^{d}\psi_P\big(\langle y',\tau_j\circ h^P \rangle\cdot\tau_i\big)\psi_N\big(\big\langle \langle y'',s''\rangle\cdot s',\tau_i\circ h^N\big\rangle\cdot \tau_j\big)^*.
	\end{align*}
	We use this to compute: 
	\begin{align*}
	\psi&_m(x)\psi_n(y)^*\psi_p(s)\psi_q(t)^*\\
	&=\psi_m(x)\big(\sum_{i,j=1}^{d}\psi_P\big(\langle y',\tau_j\circ h^P \rangle\cdot\tau_i\big)\psi_N\big(\big\langle \langle y'',s''\rangle\cdot s',\tau_i\circ h^N\big\rangle\cdot \tau_j\big)^*\big)\psi_q(t)^*\\
	&=\sum_{i,j=1}^{d}\psi_{m+P}\big(\sigma_{m,P}\big(x\otimes\langle y',\tau_j\circ h^P \rangle\cdot\tau_i\big)\big)\psi_{q+N}\big(\sigma_{N,q}\big(t\otimes\big\langle \langle y'',s''\rangle\cdot s',\tau_i\circ h^n\big\rangle\cdot\! \tau_j\big)\big)^.
	\end{align*}
	Now $\xi_{i,j}:=\sigma_{m,P}\big(x\otimes\langle y',\tau_j\circ h^P \rangle\cdot\tau_i\big)$ and $\eta_{i,j}:=\sigma_{N,q}\big(t\otimes\big\langle \langle y'',s''\rangle\cdot s',\tau_i\circ h^n\big\rangle\cdot \tau_j\big)$ satisfy \eqref{blabla}.
\end{proof}

\begin{lemma}\label{ADAD}
	Suppose  that $m,n,p,q  \in \N^k$ satisfy $m+p=n+q$ and $n\wedge p=0$. Then
	\[m-m\wedge q=n\, \text{ and }\, q-m\wedge q=p.\]
\end{lemma}
\begin{proof}
	We prove $m-m\wedge q=n$, and the other follows by symmetry. Fix $1\leq i\leq k$. Since $n\wedge p=0$, either $n_i=0$ or $p_i=0$.	If $n_i=0$, then  $m+p=n+q$  implies that $m_i\leq q_i$ and hence  $m_i-(m\wedge q)_i=m_i-m_i=0=n_i$. If $p_i=0$, then $m_i\geq q_i$ and $m_i-(m\wedge q)_i=m_i-q_i=n_i-p_i=n_i$. Thus  $m_i-(m\wedge q)_i= n_i$, as  required.
\end{proof}

\begin{proof}[Proof of Proposition~\ref{KMSPROP}]
	Suppose that  $\phi$ satisfies \eqref{KMSCHARAC}. It suffices to  check the KMS condition
	\begin{align}\label{KMS}
	\phi(bc)=e^{-\beta r\cdot(m-n)}\phi(cb)
	\end{align}
	for $b=\psi_m(x)\psi_n(y)^*$ and $c=\psi_p(s)\psi_q(t)^*$ from $\NT(X)$. Let $M:=m-m\wedge q$, $N:=n-n\wedge p$, $P:=p-n\wedge p$  and $Q:=q-m\wedge q$, and consider
\[x=\sigma_{m\wedge q,M}(x''\otimes x'),\ y=\sigma_{n\wedge p,N}(y''\otimes y'),\ s=\sigma_{n\wedge p, P}(s''\otimes s')\ \text{and}\  t=\sigma_{m\wedge q, Q}(t''\otimes t').
\]
We will need the following equations similar to \eqref{coprime}:
	\begin{align}\label{np}
	\psi_n(y)^*\psi_p(s)&=\psi_{N}(y')^*\psi_{P}\big(\langle y'',s''\rangle\cdot s' \big), \text{ and }\\
	\label{qm}
	\psi_q(t)^*\psi_m(x)
	&=\psi_{Q}(t')^*\psi_{M}\big(\langle t'',x''\rangle\cdot x' \big);
	\end{align}
	
By Lemma~\ref{application} and \eqref{KMSCHARAC}, both $\phi(bc)$ and $\phi(cb)$ vanish when $m+p\not=n+q$.  So we assume $m+p=n+q$. We claim that it suffices to prove \eqref{KMS} when $n\wedge p=0$.  Then \eqref{np} implies that $\phi(bc)=\phi\big(\psi_m(x)\psi_{N}(y')^*\psi_{P}\big(\langle y'',s''\rangle\cdot s' \big)\psi_q(t)^*\big)$; since $N\wedge P=0$, we are back in the other case, and thus
	\begin{align}
	\phi(bc)\notag&=e^{-\beta r\cdot (m-N)}\phi\big(\psi_{P}(\langle y'',s''\rangle\cdot s' )\psi_q(t)^*\psi_m(x)\psi_{N}(y')^*\big).
	\end{align}
Two similar calculations using \eqref{qm} and \eqref{np} give
	\begin{align}
	\phi(cb)
	&=\phi\big(\psi_p(s)\psi_{Q}(t')^*\psi_{M}(\langle t'',x''\rangle\cdot x' )\psi_n(y)^*\big)\label{calcKMS}
	\\
	\notag&=e^{-\beta r\cdot (p-Q)}\phi\big(\psi_{M}(\langle t'',x''\rangle\cdot x' )\psi_n(y)^*\psi_p(s)\psi_{Q}(t')^*)\quad\text{since } Q\wedge M=0\\
	\notag&=e^{-\beta r\cdot (p-Q)}\phi\big(\psi_{M}(\langle t'',x''\rangle\cdot x' )\psi_{N}(y')^*\psi_{P}(\langle y'',s''\rangle\cdot s' \big)\psi_{Q}(t')^*\big)\\
	\notag&=e^{-\beta r\cdot (p-Q+M-N)}\phi\big(\psi_{P}(\langle y'',s''\rangle\cdot s')\psi_{Q}(t')^*\psi_{M}(\langle t'',x''\rangle\cdot x' )\psi_{N}(y')^*\big)\\
	\notag&=e^{-\beta r\cdot (p-Q+M-N)}\phi\big(\psi_{P}(\langle y'',s''\rangle\cdot s')\psi_{q}(t)^*\psi_{m}(x )\psi_{N}(y')^*\big).
	\end{align}
Since $m+p=n+q$, we have  $e^{-\beta r\cdot (m-N)}=e^{-\beta r\cdot (m-n)}e^{-\beta r\cdot (p-Q+M-N)}$, and thus \eqref{calcKMS} implies that $\phi(bc)=e^{-\beta r\cdot(m-n)}\phi(cb)$. So it suffices to prove  \eqref{KMS} when $n\wedge p=0$.
	
We therefore assume that  $m+p=n+q$ and $n\wedge p=0$. Choose an open cover $\{U_i:1\leq i\leq d\}$ of $Z$ such that $h^n$ and $h^p$ are injective on each $U_i$. We take a partition of unity $\{\rho_i\}$ subordinate to $\{U_i\}$,  and define  $\tau_i:=\sqrt{\rho_i}$. To compute  $\phi(bc)$, we first use \eqref{FORMULA} to rewrite $\psi_n(y)^*\psi_p(s)$ to get
	\begin{align*}
	\phi(bc)\notag&=\phi\big(\psi_m(x)\psi_n(y)^*\psi_p(s)\psi_q(t)^*\big)\\ 
	&=\phi\Big(\psi_m(x)\Big(\sum_{i,j=1}^{d}\psi_p\big(\langle y,\tau_j\circ h^p \rangle\cdot\tau_i\big)\psi_n\big(\langle s,\tau_i\circ h^n\rangle\cdot \tau_j\big)^* \Big)\psi_q(t)^*\Big)\\
	&=\sum_{i,j=1}^{d}\phi\big(\psi_{m+p}\circ\sigma_{m,p}\big(x\otimes\langle y,\tau_j\circ h^p \rangle\cdot\tau_i\big)\psi_{q+n}\circ\sigma_{q,n}\big(t\otimes\langle s,\tau_i\circ h^n\rangle\cdot \tau_j\big)^* \big).
	\end{align*}
	By our assumption \eqref{KMSCHARAC}, we get $\phi(bc)=e^{-\beta r\cdot (m+p)}\phi\circ \psi_0(\dag)$, where
	\[{\dag:=\sum_{i,j=1}^{d}\big\langle \sigma_{q,n}\big(t\otimes\langle s,\tau_i\circ h^n\rangle\cdot \tau_j\big),\;\sigma_{m,p}\big(x\otimes\langle y,\tau_j\circ h^p \rangle\cdot\tau_i\big)\big\rangle}.
	\]

	To compute $\phi(cb)$, we first observe that $m+p=n+q$ and $n\wedge p=0$ imply	$Q=p$ and $M=n$ (see Lemma~\ref{ADAD}). Thus
	\begin{align*}
	\phi(cb)\notag&=\phi\big(\psi_p(s)\psi_{Q}(t')^*\psi_{M}\big(\langle t'',x''\rangle\cdot x' \big)\psi_n(y)^*\big)\quad\text{by \eqref{qm}} \\
	&=\phi\big(\psi_p(s)\psi_{p}(t')^*\psi_{n}\big(\langle t'',x''\rangle\cdot x' \big)\psi_n(y)^*\big).
	\end{align*}
	 Now we use \eqref{FORMULA} to rewrite the middle terms:
	\begin{align*}
	\notag&\phi(cb)
	=\phi\Big(\psi_p(s)\Big(\sum_{i,j=1}^d\psi_{n}\big(\langle t',\tau_i\circ h^n \rangle\cdot\tau_j\big)\psi_{p}(\big\langle \langle t'',x''\rangle\cdot x',\tau_j\circ h^p \big\rangle\cdot \tau_i\big)^*
	\Big)\psi_n(y)^*\Big)\\
	\notag&=\phi\Big(\sum_{i,j=1}^d\psi_{p+n}\big(\sigma_{p,n}\big(s\otimes\langle t',\tau_i\circ h^n \rangle\cdot\tau_j\big)\big)\psi_{n+p}\big(\sigma_{n,p}\big(y\otimes\big\langle \langle t'',x''\rangle\cdot x',\tau_j\circ h^p\big\rangle\cdot\tau_i\big)\big)^*\Big)
	\end{align*}
	The assumption 	\eqref{KMSCHARAC} implies that $\phi(cb)\notag=e^{-\beta r\cdot (n+p)}\phi\circ \psi_0(\ddag)$, where
	\[
	\ddag:=\sum_{i,j=1}^{d}\big\langle \sigma_{n,p}\big(y\otimes\big\langle \langle t'',x''\rangle\cdot x' ,\tau_j\circ h^p\big\rangle\cdot\tau_i\big),\;
	\sigma_{p,n}\big(s\otimes\langle t',\xi_i\circ h^n\rangle\cdot\tau_j\big)\big\rangle.
	\]	
	Since $e^{-\beta r\cdot (m+p)}=e^{-\beta r\cdot (m-n)}e^{-\beta r\cdot(n+p)}$, to verify the KMS condition \eqref{KMS} it suffices to prove that $\dag=\ddag$. So we evaluate both on $z\in Z$, doing  the easier calculation for $\ddagger(z)$ first. 
	
The $i$-$j$-summand in $\ddagger(z)$ is
\begin{equation}\label{ijsummand}
\sum_{h^{n+p}(w)=z}\!\!\overline{\sigma_{n,p}\big( y\otimes\big\langle \langle t'',x''\rangle\!\cdot\! x' ,\tau_j\circ h^p\big\rangle\!\cdot\!\tau_i\big)(w)}
	\sigma_{p,n}\big(s\otimes\langle t',\tau_i\circ h^n \rangle\cdot\tau_j\big)(w),
\end{equation}
and the $w$-summand in \eqref{ijsummand} is	
\begin{align*}
&\overline{y(w)\big\langle \langle t'',x''\rangle\cdot x' ,\tau_j\circ h^p\big\rangle (h^n(w) )\tau_i\big(h^n(w)\big)}s(w)\langle t',\tau_i\circ h^n \rangle(h^p(w))\tau_j(h^p(w))\\
&=\overline{y(w)}s(w)\langle t',\tau_i\circ h^n \rangle(h^p(w))   \tau_i (h^n(w))\tau_j\big(h^p(w)\big)\big\langle \tau_j\circ h^p\,,\langle t'',x''\rangle\cdot x' \big\rangle(h^n(w))\\
&=\overline{y(w)}s(w)\overline{\big(\tau_i\circ h^n \cdot\langle \tau_i\circ h^n, t'\rangle \big)(w)}\big(\tau_j\circ h^p\cdot\big\langle \tau_j \circ  h^p,\langle t'',x''\rangle\cdot x' \big\rangle\big)(w).
	\end{align*}
By Lemma~\ref{pframe},  $\{\tau_j\circ h^p\}$ and $\{\tau_i\circ h^n\}$   are Parseval frames for $X_n=X(Z,Z,\id,h^n)$ and $X_p=X(Z,Z,\id,h^p)$. Thus when we sum over $i$, $j$ and $w$, the reconstruction formulas for these frames give
	\begin{align*}
	\ddag(z)
	&=\sum_{h^{n+p}(w)=z}\overline{y(w)}s(w)\overline{t'(w)}(\langle t''\!,\!x''\rangle\cdot\!x') (w).
	\end{align*}
		
	Next we compute $\dag(z)$. Using the formula \eqref{multiplicationformula1} for multiplication in $X$, we find that the $i$-$j$-summand in $\dag(z)$ is
	\begin{align*}
	&\dag_{ij}(z)
	=\sum_{h^{m+p}(w)=z}\overline{\sigma_{q,n}\big(t\otimes\langle s,\tau_i\!\circ\! h^n\rangle\cdot \tau_j\big)(w)}
	\sigma_{m,p}\big(x\otimes\langle y,\tau_j\!\circ\! h^p \rangle\cdot\tau_i\big)(w)\\
	&=\sum_{h^{m+p}(w)=z}\!\!\!\!\!\!\overline{t(w)\langle s,\tau_i\circ h^n\rangle(h^q(w))\tau_j (h^q(w))}
	x(w)\langle y,\tau_j\circ h^p \rangle(h^m(w))\tau_i(h^m(w))\\
	&=\sum_{h^{m+p}(w)=z}\overline{t(w)}x(w)\overline{\langle s,\tau_i\circ h^n\rangle(h^q(w))}\tau_i (h^m(w))
	\tau_j(h^q(w))\langle y,\tau_j\circ h^p \rangle(h^m(w)).
	\end{align*}
Lemma~\ref{ADAD} implies that $q=m\wedge q+p$ and $m=m\wedge q+n$.  Then we can rearrange the $w$-summand in $\dag_{ij}(z)$ as
	\begin{align*}
&\overline{t(w)}x(w)\langle \tau_i\circ h^n,s\rangle(h^{m\wedge q+p}(w))\tau_i (h^{m\wedge q+n}(w))\tau_j(h^{m\wedge q+p}(w))\langle y,\tau_j\circ h^p \rangle(h^{m\wedge q+n}(w))\\
&=\overline{t(w)}x(w)\big(\tau_i\circ h^n \,\cdot\,\big\langle \tau_i\circ h^n\,,\, s \big\rangle\big)(h^{m\wedge q}(w))\overline{\big(\tau_j\circ h^p\cdot\langle \tau_j\circ h^p\,,\,y \rangle\big)(h^{m\wedge q}(w))}.
	\end{align*}
We sum these terms over $i$, $j$ and $w$, and use the reconstruction formulas for the  Parseval frames $\{\tau_i\circ h^n\}$ and $\{\tau_j\circ h^p\}$ to recognise
	\[
	\dag(z)=\sum_{h^{n+p+m\wedge q}(w)=z}\overline{t(w)}x(w)s (h^{m\wedge q}(w))\overline{y(h^{m\wedge q}(w))}.
	\]
Finally, we recall that $t=\sigma_{m\wedge q, Q}(t''\otimes t')$, $x=\sigma_{m\wedge q,M}(x''\otimes x')$  and split the sum into two stages:
	\begin{align*}
	\dag(z)&=\sum_{h^{n+p+m\wedge q}(w)=z}\overline{t''(w)t'(h^{m\wedge q}(w))}x''(w)x'(h^{m\wedge q}(w))s (h^{m\wedge q}(w))\overline{y(h^{m\wedge q}(w))}\\
	&=\sum_{h^{n+p}(u)=z}s(u)x'(u)\overline{y(u)t'(u)}\!\!\!\!\sum_{h^{m\wedge q}(w)=u}\overline{t''(w)}x''(w)\\	
	&=\sum_{h^{n+p}(u)=z}s(u)x'(u)\overline{y(u)t'(u)}\langle t'',x''\rangle(u).
	\end{align*}
Thus $\dag(z)=\ddag(z)$, and $\phi$ satisfies \eqref{KMS}.

To prove part (b), suppose that $\phi$ is a KMS$_\beta$ state  on $\NT(X)$ and that $r$ has rationally independent coordinates. To see that $\phi$ satisfies \eqref{KMSCHARAC}, take $x\in X_m$ and  $y\in X_n$.
	By two applications of the KMS condition, we have
	\begin{align*}
	\phi(\psi_m(x)\psi_n(y)^*)&=\phi(\psi_n(y)^*\alpha_{i\beta}(\psi_m(x)))\\
	&=e^{-\beta r\cdot m}\phi(\psi_n(y)^*\psi_m(x))\\
	&=e^{-\beta r\cdot (m-n)}\phi(\psi_m(x)\psi_n(y)^*).
	\end{align*}
Since $r$  has rationally independent coordinates and $\beta>0$, both sides vanish for $m\neq n$. For $m=n$ the  KMS condition and (T2)  imply that
	\[\phi(\psi_m(x)\psi_m(y)^*)=e^{-\beta r\cdot m}\phi(\psi_m(y)^*\psi_m(x))= e^{-\beta r\cdot m}\phi(\psi_0(\langle y,x\rangle)),\]
which is \eqref{KMSCHARAC}.
\end{proof}

\section{KMS states  at large inverse temperatures}\label{sec6}
In this section we  identify the simplex of  KMS$_\beta$ states of $\NT(X)$ for $\beta$ above the critical inverse temperature. The  rational independency condition on the dynamics in Theorem~\ref{theorem1}(b)  also came up in  \cite{HLS, aHLRS}.
Theorem~6.1 extends  \cite[Theorem~6.1]{AaHR} and  \cite[Theorem~6.1]{aHLRS}.  

\begin{thm}\label{theorem1}
	Let $h_1,\dots,h_k$ be  $*$-commuting and  surjective  local homeomorphisms on a compact Hausdorff space   $Z$.
	Let  $X$ be the associated product system over $\N^k$,   as in Proposition~\ref{prop1}. For  $ 1\leq i\leq k$, let  $\beta_{c_i}$ be as in
	\eqref{criticalpoint}, and
	suppose that $r\in (0,\infty)^k$ satisfies $\beta r_i> \beta_{c_i}$ for all $i$.
	Let $f_\beta$ be  the function in  Proposition~\ref{series}(a) and define $\alpha : \R\rightarrow \Aut\NT(X)$ by $\alpha_t=\gamma_{e^{ i t r}}.$
	
	\begin{itemize}
		\item[(a)]Suppose that $\varepsilon$ is a finite Borel measure on $Z$ such 
		that $\int f_\beta\,d\varepsilon=1,$ and define $\mu:=\sum_{n\in \N^k}e^{-\beta r\cdot n}R^n\varepsilon$. Then there is a KMS$_\beta$ state
		$\phi_\varepsilon$
		on $(\NT(X),\alpha)$ such that
		\begin{align}\label{equkms}
		\phi_{\varepsilon}\big(\psi_m(x){\psi_p(y)}^*\big)=\begin{cases}0  & \text{if $m\neq p$}\\e^{-\beta r\cdot m}\int \langle y,x\rangle\,d\mu &  \text{if  $m=p$.}\end{cases}
		\end{align}
		\item[(b)] If in addition $r$ has rationally independent coordinates, then the map $\varepsilon \mapsto \phi_\varepsilon$  is an affine isomorphism of
		\[\Sigma_\beta:=\Big\{\varepsilon\in M(Z)_+:\int f_\beta\,d\varepsilon=1\Big\}\]
		onto the simplex of KMS$_\beta$ states of $(\NT(X),\alpha)$. Given a state $\phi$, let $\mu$ be the probability measure such that $\phi(\psi_0(a))=\int a\,d\mu$ for $a\in C(Z)$. Then the inverse of $\varepsilon \mapsto \phi_\varepsilon$   takes $\phi$ to $\varepsilon:=\prod_{i=1}^k(1-e^{-\beta r_i}R^{e_i})\mu$.
	\end{itemize}
\end{thm}

\begin{proof}[Proof of Theorem~\ref{theorem1}\textnormal{(a)}]
Let $\varepsilon$ be  a finite Borel measure on $Z$ such 
		that $\int f_\beta\,d\varepsilon=1$. We follow the  structure of the proof of
\cite[Theorem~5.1]{AaHR}. Thus  we aim to construct the KMS state $\phi_\varepsilon$ by using a representation $\theta$ of $X$ on
$H_{\theta}:=\bigoplus_{n\in \N^k} L^2(Z,R^n\varepsilon)$; we  write $\xi=(\xi_n)$ for the elements of the direct
sum.  

We now fix $m\in \N^k$ and  $x\in X_m$, and claim that there is a bounded operator $\theta_m(x)$ on $H_\theta$ such that	
\begin{align*}
(\theta_m(x)\xi)_{n}(z)= \begin{cases}
0  & \text{if $n\ngeq m$}\\
x(z)\xi_{n-m}(h^m(z))& \text{if $n\geq m$.}
\end{cases}
\end{align*}
To see the claim, we take $\xi=(\xi_n)\in\bigoplus_{n\in \N^k}  L^2(Z, R^n\varepsilon)$ and 
 compute using the formula \eqref{defR} for $R^n$: 
 \begin{align} 
 \|\theta_m(x)\xi\|^2\notag&=\sum_{n\in \N^k}\|(\theta_m(x)\xi)_{n}\|^2
 = \sum_{n\geq m} \int |x(z)|^2|\xi_{n-m}(h^m(z))|^2\,d(R^{n}\varepsilon)(z)\\                                                                      
\notag&= \sum_{n\in \N^k} \int |x(z)|^2|\xi_{n}(h^m(z))|^2\,d(R^{n+m}\varepsilon)(z)\\
\notag&\leq\sum_{n\in \N^k}\|x\|_\infty^2\int \sum_{h^m(w)=z}|\xi_n(h^m(w))|^2\,d(R^{n}\varepsilon)(z)\\
\notag&=\sum_{n\in \N^k} \|x\|_\infty^2\int \sum_{h^m(w)=z}|\xi_n(z)|^2\,d(R^{n}\varepsilon)(z),
\end{align}	
which, after setting $c_m:=\textstyle{\max_z|h^{-m}(z)|}$, is
\begin{align*}
  \|\theta_m(x)\xi\|^2&\leq\sum_{n\in \N^k} \|x\|_\infty^2 c_m \int |\xi_n(z)|^2\,d(R^{n}\varepsilon)(z)  
=c_m\|x\|_\infty^2 \|\xi\|^2.
\end{align*}	
Thus $\theta_m(x)$ is bounded, as claimed. 	A similar calculation shows that the  adjoint $\theta_m(x)^*$ satisfies
\begin{equation}\label{equ44}
\big(\theta_m(x)^*\eta\big)_n(z)=\sum_{h^m(w)=z}\overline{x(w)}\eta_{n+m}(w)\quad \text{for $\eta\in H_\theta$.} 
\end{equation}
	
	Next we claim that $\theta$ is a  Toeplitz representation of $X$. We check the conditions (T1)$-$(T3). Since $\theta_0(x)$ is  pointwise multiplication by the function $x$, and $X_0=A$,  (T1) is trivially true. 	
	To check (T2), fix $m$ and  $x_1,x_2\in X_m$. Then
	\begin{align*}
	\big(\theta_0(\langle x_1,x_2\rangle)\xi\big)_n(z)&=\langle x_1,x_2\rangle(z)\xi_n(z)
	=\sum_{h^m(w)=z}\overline{x_1(w)}x_2(w)\xi_n(z)\\
	&=\sum_{h^m(w)=z}\overline{x_1(w)}x_2(w)\xi_n(h^m(w))\\
	&=	\sum_{h^m(w)=z}\overline{x_1(w)}(\theta_m(x_2)\xi)_{n+m}(w).
\end{align*}
	Now \eqref{equ44} implies that	
	$\big(\theta_0(\langle x_1,x_2\rangle)\xi\big)_n(z)=\big(\theta_m(x_1)^*\theta_m(x_2)\xi\big)_n(z),$ which is (T2).
	
	For (T3), let  $x\in X_m$ and $y\in X_p$.
	If $n\ngeq m+p$, then  $\big(\theta_{m+p}(xy)\xi\big)_n(z)=0$ and  
	\[\big(\theta_m(x)\theta_p(y)\xi\big)_{n} (z)=x(z)\big(\theta_p(y)\xi\big)_{n-m}\big(h^m(z)\big)=0.\]
So we assume $n\geq m+p$. The multiplication formula \eqref{multiplicationformula1} for $X$ implies that
	\begin{align*}
	\big(\theta_{m+p}(xy)\xi\big)_n(z)&= 
	x(z)y\big(h^m(z)\big) \xi_{n-(m+p)}\big(h^{m+p}(z)\big)\\
	&=x(z)\big(\theta_p(y)\xi\big)_{n-m} (z)\big(h^m(z)\big)\\
	&=\big(\theta_m(x)\theta_p(y)\xi\big)_{n} (z),
	\end{align*}
which gives (T3). Thus $\theta$ is a Toeplitz representation, as claimed.
	
	Next we show that  $\theta$ is Nica covariant.   Fix $m,p\in \N^k$.   Since $m\vee p<\infty$, we need to show 
	that $\psi^{(m)}(1_m)\psi^{(p)}(1_p)=\psi^{(m\vee p)}(1_{m\vee p})$ as in \eqref{nicacov-using-lemma}. We choose  a partition of unity   $\{\rho_j: 1\leq j\leq d\}$ for $Z$  such that   $h^{m\vee p}$ is injective on each $\supp \rho_j$ and 
take $\tau_j:=\sqrt{\rho_j}\in X_m$. Notice that $\{\tau_j\}$ can be viewed  as a Parseval frame for the fibres $X_m,X_p$ and $X_{m\vee p}$.  By  
Lemma~\ref{nicaeasy}, it suffices to show that
	\begin{align}\label{nicalater2}
	\Big(\sum_{i=1}^{d}\theta_m(\tau_i)\theta_m(\tau_i)^*\Big)\Big(\sum_{j=1}^{d}\theta_p(\tau_j)\theta_p(\tau_j)^*\Big)=
	\Big(\sum_{l=1}^{d}\theta_{m\vee p}(\tau_l)\theta_{m\vee p}(\tau_l)^*\Big).
	\end{align}
To see this, we take  $\xi\in H_\theta$ and evaluate both sides of \eqref{nicalater2} at $\xi$. 
	 For the right-hand side of \eqref{nicalater2}, the definition of $\theta_{m\vee p}$  implies that
$\big(\big(\sum_{l=1}^{d}\theta_{m\vee p}(\tau_l)\theta_{m\vee p}(\tau_l)^*\big)\xi\big)_n$
	vanishes unless $n\geq m\vee p$. So we assume $n\geq m\vee p$. For $z\in Z$, the definition of $\theta_{m\vee p}$ and the adjoint formula  \eqref{equ44} imply that
	\begin{align*}
	\Big(\Big(\sum_{l=1}^{d}\theta_{m\vee p}(\tau_l)\theta_{m\vee p}(\tau_l)^*\Big)\xi\Big)_n(z)&=\sum_{l=1}^{d} \big(\tau_l(z)\big(\theta_{m\vee p}(\tau_l)^*\xi\big)_{n-m\vee p}\big(h^{m\vee p}(z)\big)\big)\\
	&=\sum_{l=1}^{d} \Big(\tau_l(z)\!\!\!\!\sum_{h^{m\vee p}(w)=h^{m\vee p}(z)}\!\!\!\!\overline{\tau_l(w)}\xi_n(w)\Big).
	\end{align*}	
	Since $h^{m\vee p}$ is injective on each $\supp \tau_l$, we have
	\[\Big(\Big(\sum_{l=1}^{d}\theta_{m\vee p}(\tau_l)\theta_{m\vee p}(\tau_l)^*\Big)\xi\Big)_n(z)=\sum_{l=1}^{d} 
	\big(\tau_l(z)\overline{\tau_l(z)}\xi_n(z)\big)=\xi_n(z)\sum_{l=1}^{d}\big|\tau_l(z)\big|^2=\xi_n(z).\]
	Thus
	\begin{align*}
	\Big(\Big(\sum_{l=1}^{d}\theta_{m\vee p}(\tau_l)\theta_{m\vee p}(\tau_l)^*\Big)\xi\Big)_n= \begin{cases}
	\xi_n & \text{if $n\geq m\vee p$}\\
	0& \text{otherwise.}
	\end{cases}  
	\end{align*}
	
	For the left-hand side of \eqref{nicalater2},  notice that $\{\tau_i\}$ is a Parseval frame for $X_m$ and $h^{m}$ is injective 
	on each $\supp \tau_i$. Then applying  the same calculation of the previous paragraph (using the formulas for $\theta_m(\tau_i)$ and $\theta_m(\tau_i)^*$), we have
		\begin{align*} 
		\Big(\Big(\sum_{i=1}^{d}\theta_m(\tau_i)\theta_m(\tau_i)^*\sum_{j=1}^{d}\theta_p(\tau_j)\theta_p(\tau_j)^*\Big)\xi\Big)_n= \begin{cases}       
		\big(\big(\sum_{j=1}^{d}\theta_p(\tau_j)\theta_p(\tau_j)^*\big)\xi\big)_n & \text{if $n\geq m$}\\     
		0& \text{otherwise.}
		\end{cases} 
		\end{align*}
A similar computation shows that
	\begin{align*}
	\Big(\Big(\sum_{i=1}^{d}\theta_m(\tau_i)\theta_m(\tau_i)^*\sum_{j=1}^{d}\theta_p(\tau_j)\theta_p(\tau_j)^*\Big)\xi\Big)_n= \begin{cases}
	\xi_n & \text{if $n\geq m\vee p$}\\
	0& \text{otherwise.}
	\end{cases}   \end{align*}
Thus the left-hand side of \eqref{nicalater2} vanishes unless $n\geq m$ and $n\geq p$, which is equivalent to $n\geq m\vee p$, and in that case is $\xi_n$. Thus we have equality in \eqref{nicalater2}, and  $\theta$ is Nica covariant.
	
	Now the universal property of $\NT(X)$ (\cite[Theorem~6.3]{Fo})  gives us a 
	homomorphism $\theta_*:\NT(X)\to B(H_\theta)$ such that $\theta_*\circ\psi=\theta$. For each $q\in \N^k$,  we choose a finite partition $\{Z_{q,i}: 1\leq i\leq I_q\}$ of $Z$ by Borel sets such that $h^q$ is one-to-one on each $Z_{q,i}$.  We take  $Z_{0,1}=Z$, set $I_0=1$, write $\chi_{q,i}:=\chi_{Z_{q,i}}$, and
	define $\xi^{q,i}\in \bigoplus_{n\in \N^k} L^2(Z,R^n\varepsilon)$ by
	\begin{equation*}
	\xi_n^{q,i}= \begin{cases}
	0  & \text{if $n\neq q$}\\                            
	\chi_{q,i}& \text{if $n=q$.}
	\end{cases}
	\end{equation*}

	We now claim that there is a well-defined function $\phi_{\varepsilon}:\NT(X)\to \C$ such that 
	\begin{equation}\label{defphi}
	\phi_{\varepsilon}(b)=\sum_{q\in \N^k}\sum_{i=1}^{I_q} e^{-\beta r\cdot q}\big(\theta_*(b)\xi^{q,i}\,\big|\,\xi^{q,i}\big)\quad\text{for  $b\in \NT(X)$.}
	\end{equation}
We first need to show that  the series converges. Since every element of  a $C^*$-algebra is a linear combination of positive elements, and every positive element $b$ satisfies $b\leq \|b\|1$,  it suffices  to show that the series
	defining $\phi_{\varepsilon}(1)$ is convergent.   By definition, 
	\begin{align*}
	\phi_\epsilon(1)=\sum_{q\in \N^k}\sum_{i=1}^{I_q}e^{-\beta r\cdot q}\big(\chi_{Z_{q,i}}\,\big|\,\chi_{Z_{q,i}}\big)&= \sum_{q\in \N^k}\sum_{i=1}^{I_q}e^{-\beta r\cdot q}\int
	\overline{\chi_{Z_{q,i}}(z)}\chi_{Z_{q,i}}(z)\,d(R^q\varepsilon)(z)\\
	&=\sum_{q\in \N^k}\sum_{i=1}^{I_q}e^{-\beta r\cdot q}R^q\varepsilon(Z_{q,i}).
	\end{align*}
	Since the  $Z_{q,i}$  partition $Z$, we have	
	\begin{align*}
	\sum_{q\in \N^k}\sum_{i=1}^{I_q}e^{-\beta r\cdot q}\big(\chi_{Z_{q,i}}\,\big|\,\chi_{Z_{q,i}}\big)=
	\sum_{q\in \N^k}e^{-\beta r\cdot q}R^q\varepsilon(Z).
	\end{align*}
	By Proposition~\ref{series}(b), the sum  $\sum_{q\in \N^k}e^{-\beta r\cdot q}R^q\varepsilon$	converges to a measure $\mu$, which  is a probability measure because $\int f_\beta\,d\varepsilon=1$. Then
	\begin{align*}
	\sum_{q\in \N^k}\sum_{i=1}^{I_q}e^{-\beta r\cdot q}\big(\chi_{Z_{q,i}}\,\big|\,\chi_{Z_{q,i}}\big)=\mu(Z)=1.
	\end{align*}
	Thus $\phi_\epsilon(1)=1$, and the formula \eqref{defphi} gives us a well-defined state on  $\TT(X(E))$.

	To see that $\phi_\varepsilon$ satisfies \eqref{equkms},
	take $x\in X_m$,  $y\in X_p$ and $b=\psi_m(x)\psi_p(y)^*$. Since $\xi^{q,i}$ is zero in all except the 
	$q$-summand,
	$\theta_*(b)\xi^{q,i}=\theta_m(x)\theta_p(y)^*\xi^{q,i}$
	is zero in all but the $(q-p+m)$-summand. Thus 
	\[\big(\theta_*(b)\xi^{q,i}\,\big|\,\xi^{q,i}\big)=0\quad \text{for all $q,i$ for $p\not=m$.}\]
Thus $\phi_{\varepsilon}(b)=0$  giving \eqref{equkms} when $p\not=m$.  Next we suppose that $p=m$. If $q\ngeq m$, then $\theta_m(x)\theta_m(y)^*\xi^{q,i}=0$, so we suppose also that $q\geq m$. Since  $h^q$ is injective on $Z_{q,i}$, $h^m$ is injective on each $Z_{q,i}$. Then  
	\begin{align*}
	\big(\theta_m(x)\theta_m(y)^*\xi^{q,i}\,\big|\,\xi^{q,i}\big)
	&=\int \Big(x(z)\sum_{h^m(w)=h^m(z)}\overline{y(w)}\chi_{q,i}(w)\Big)\overline{\chi_{q,i}(z)}\,d(R^q\varepsilon)(z)\\
	&=\int x(z)\overline{y(z)}\overline{\chi_{q,i}(z)}\chi_{q,i}(z)\,d(R^q\varepsilon)(z).
	\end{align*}
	Since the $Z_{q,i}$ partition $Z$, summing over $i$ gives 
	\[ \sum_{i=1}^{I_q} \big(\theta_*(\psi_m(x)\psi_m(y)^*)\xi^{q,i}\,\big|\,\xi^{q,i}\big) =\int x(z)\overline{y(z)}\,d(R^q\varepsilon)(z).\]  
	Now  using the  formula \eqref{defR} for $R^m$, we have
	\begin{align*}
	\phi_{\varepsilon}\big(\psi_m(x)\psi_m(y)^*\big)&=\sum_{q\geq m} e^{-\beta r\cdot q}\int x(z)\overline{y(z)}\,d(R^q\varepsilon)(z)\\  
&=\sum_{q\geq m}  e^{-\beta r\cdot q}\int\sum_{h^m(w)=z} x(w)\overline{y(w)}\,d(R^{q-m}\varepsilon)(z)\\
&=\sum_{q\in \N^k}  e^{-\beta r\cdot  (m+q)}\int\langle y, x\rangle(z)\,d(R^{q}\varepsilon)(z)\\
&=e^{-\beta r\cdot m}\int \langle y, x\rangle\,d\Big(\sum_{q\in \N^k}  e^{-\beta q}R^{q}\varepsilon\Big)\\
&=e^{-\beta r\cdot  m}\int\langle y, x\rangle \,d\mu, 
\end{align*}
which is \eqref{equkms}.
	
	To see that $\phi_\varepsilon$ is a KMS$_\beta$ state,
	we apply Proposition~\ref{KMSPROP} with $m=p=0$  and  $x=y=b\in A$ to get
	\[\phi_\varepsilon\big(\psi_0(bb^*)\big)=\phi_\varepsilon\big(\psi_0(b)\psi_0(b)^*\big)=\int\langle b,b^*\rangle_A\,d\mu=\int bb^*\,d\mu.\]
	This implies that $\phi_\varepsilon(\psi_0(a))=\int a\,d\mu$ for all positive $a\in A$,  hence for all  $a\in A$, and in particular $a= \langle y,x\rangle$.
	Now
	\[\phi_{\varepsilon}\big(\psi_m(x)\psi_n(y)^*\big)=\delta_{m,n}e^{-\beta r\cdot m}\phi_{\varepsilon}\big(\psi_0 \big(\langle y,x\rangle\big)\big),\]
	and Proposition~\ref{KMSPROP}(a)  implies that  $\phi_\varepsilon$ is a  KMS$_\beta$ state.
\end{proof}

\begin{proof}[Proof of Theorem~\ref{theorem1}\textnormal{(b)}]
	Now suppose that $r$ has rationally independent coordinates. Since $\Sigma_\beta$ is a weak* compact subset of $C(Z)^*$ in the  weak$^*$ norm, to prove that $\varepsilon\mapsto \phi_\varepsilon$ is an isomorphism, it suffices to show that it is injective, surjective and weak* continuous.
	
	Let $\phi$ be a KMS$_\beta$ state, and take $\mu$ to be the probability measure such that $\phi\big(\psi_0(a)\big)=\int a\,d\mu$ for $a\in C(Z)$.  By  Proposition~\ref{suninvprop},  $\mu$ satisfies the subinvariance  relation \eqref{subinvariance}. Since $r$ has rationally independent coordinates,   Proposition~\ref{KMSPROP} implies that
	\begin{align}\label{KMSCHARAC1}
	\phi\big(\psi_m(x){\psi_n(y)}^*\big)=\delta_{m,n}	e^{-\beta r\cdot m}\phi\big(\psi_0\big(\langle y,x\rangle\big)\big)=
	e^{-\beta r\cdot m}\int\langle y,x\rangle\,d\mu.
	\end{align}
	 Then   Proposition~\ref{series}(c) implies that $\varepsilon:=\prod_{i=1}^{k}\big(1-e^{-\beta r_i}R^{e_i}\big)\mu$ belongs to $\Sigma_\beta$ and satisfies 
	$\sum_{n\in \N^k}e^{-\beta r\cdot n }R^n\varepsilon=\mu$. Now applying part~(a)  to $\varepsilon$ gives a KMS$_\beta$ state
	$\phi_\varepsilon$ such that
	\begin{align}\label{equkms1}
	\phi_{\varepsilon}\big(\psi_m(x){\psi_n(y)}^*\big)=\begin{cases}
	0  & \text{if $m\neq n$}\\
	e^{-\beta r\cdot m}\int\langle y,x\rangle\,d\mu &  \text{if  $m=n$.}
	\end{cases}
	\end{align}
	Comparing  equations \eqref{equkms1} and \eqref{KMSCHARAC1} gives $\phi=\phi_\varepsilon$. Thus $\varepsilon\mapsto \phi_\varepsilon$ is surjective.
	
	Next suppose that $\epsilon_i\in \Sigma_\beta$ and $\phi_{\varepsilon_1}=\phi_{\varepsilon_2}$. Define $\mu_i$ by 
	$\phi_{\varepsilon_i}\circ\psi_0(a)=\int a\,d\mu_i$ for $a\in A$. Then  $\mu_1=\mu_2$.  The construction of the  previous paragraph shows that
	\[
	\varepsilon_1=\prod_{i=1}^{k}\big(1-e^{-\beta r_i}R^{e_i}\big)\mu_1=
	\prod_{i=1}^{k}\big(1-e^{-\beta r_i}R^{e_i}\big)\mu_2=\varepsilon_2.
	\]
	Thus  $\varepsilon\mapsto \phi_{\varepsilon}$ is one-to-one.
	
	Finally, suppose that $\varepsilon_j\rightarrow \varepsilon$ in $\Sigma_\beta$. The calculation \eqref{YYYY} implies that  
\[
\mu_j:=\sum_{n\in \N^k}e^{-\beta r\cdot n}R^n\varepsilon_j\to \mu:=\sum_{n\in \N^k}e^{-\beta r\cdot n}R^n\varepsilon
\]  
in the weak* topology, and then \eqref{equkms} implies that   $\phi_{\varepsilon_j}\rightarrow\phi_{\varepsilon}$ in the weak* topology.
\end{proof}

The next corollary extends \cite[Corollary~5.3]{AaHR}.

\begin{cor}\label{cor6}
	Let $h_1,\dots,h_k$ be  $*$-commuting and  surjective  local homeomorphisms on a compact Hausdorff space   $Z$, and let  $X$ be the associated product system over $\N^k$,   as in Proposition~\ref{prop1}. Define  $\beta_{c_i}$ by \eqref{criticalpoint} and suppose that $r\in (0,\infty)^k$ has rationally independent coordinates. Define $\bar{\alpha}: \R\rightarrow \Aut\OO(X)$ in terms of the gauge action $\bar{\gamma}$ by $\bar{\alpha}_t=\bar{\gamma}_{e^{ i t r}}$. If there is a KMS$_\beta$ state of  $(\OO(X),\bar{\alpha})$, then $\beta r_i\leq \beta_{c_i} $ for at least one $i$.
\end{cor}

\begin{proof}
	Suppose that $\phi$ is  a KMS$_\beta$ state of  $(\OO(X),\bar{\alpha})$. We suppose that   $\beta r_i>\beta_{c_i}$	 for $1\leq i\leq k $, and aim for a contradiction. Let $q: \NT(X)\rightarrow\OO(X)$ be the quotient map of Lemma~\ref{fixdef}. Then $\phi\circ q$ is a KMS$_\beta$ state for the  system $(\NT(X), \alpha)$ in Theorem~\ref{theorem1}. Since $r$ has rationally independent coordinates,  part~(b) of Theorem~\ref{theorem1} gives a measure $\varepsilon$ on $Z$ such that $\int f_\beta\,d\varepsilon=1$ and $\phi\circ q=\phi_\varepsilon.$ Since  $f_\beta\geq 1$, $\int f_\beta\,d\varepsilon=1$ implies that  $\varepsilon(Z)>0.$
	
	Set $K:=\{1,\dots,k\}$ and for $J\subset K$, set $e_J:=\sum_{i\in J}e_i$. Choose an open cover $\{U_l:1\leq l\leq d \}$  of  $Z$ such that    $h^{e_J}|_{U_l}$  is injective for every $J\subset K$ and $1\leq l\leq d$.
	By \cite[Lemma~4.32]{tfb} there is an open cover $\{V_l:1\leq l\leq d \}$ of $Z$ such that $\overline{V_l}\subset U_l$ for each $l$. Since $\varepsilon(Z)>0 $, there exists $l$ such that $\varepsilon(V_l)>0.$ By Urysohn's Lemma, there is a function $f\in C(Z)$ such that $f(z)=1$ for $z\in V_l$ and $\supp f\subset U_l$.
	
	Next for each $J\subset K$,  take $f_J:=f\in X_{e_J}$ and view $|f|^2$ as an element of $A= C(Z)$. We aim  to set up a contradiction by showing that
	\[b:=\psi_0(|f|^2)+\sum_{\emptyset \subsetneq J\subseteq K}(-1)^{|J|} \psi_J(f_J)\psi_J(f_J)^*\]	
	belongs to $\ker q$ but $\phi_\varepsilon(b)=\phi\circ q(b)\not=0$. Since the left action of $|f|^2$ on each fibre $X_{e_J}$ is implemented by the finite-rank operator $\Theta_{f_J,f_J}$, we have
	\begin{align*}
	b&=
	\psi_0(|f|^2)+\sum_{\emptyset \subsetneq J\subseteq K}(-1)^{|J|} \psi^{(e_J)}\big(\Theta_{f_J,f_J}\big)\\
	&=\sum_{\emptyset \subsetneq J\subseteq K}(-1)^{(|J|+1)}\psi_0(|f|^2)+\sum_{\emptyset \subsetneq J\subseteq K}(-1)^{|J|} \psi^{(e_J)}\big(\varphi_{e_J}(|f_J|^2)\big)\\
	&=\sum_{\emptyset \subsetneq J\subseteq K}(-1)^{|J|}\big(\psi_0(|f|^2)- \psi^{(e_J)}\big(\varphi_{e_J}(|f|^2)\big)\big).
	\end{align*}
Since each summand is in $\ker q$, so is $b$.
	
	Next we compute  $\phi_\varepsilon(b)$ using the measure $\mu$ in part (b) of Theorem~\ref{theorem1}:
	\begin{align*}
	\phi_\varepsilon(b)
	&=\int |f|^2(z)\, d\mu(z)+\sum_{\emptyset \subsetneq J\subseteq K}(-1)^{|J|} e^{-\beta r\cdot e_J}\int \big\langle f_J,\overline{f_J}\big\rangle(z)\,d\mu(z)\\
	&=\int |f|^2(z)\, d\mu(z)+\sum_{\emptyset \subsetneq J\subseteq K}(-1)^{|J|} e^{-\beta r\cdot e_J}\int \sum_{h^{e_J}(w)=z}|f|^2(w)\,d\mu(z).
	\end{align*}
Recalling the the definition of $R$ from \eqref{defR},  we have
	\begin{align*}
	\phi_\varepsilon(b)	
	&=\int |f(z)|^2(z)\, d\mu(z)+\sum_{\emptyset \subsetneq J\subseteq K}(-1)^{|J|} e^{-\beta r\cdot e_J}\int|f^2(z)|\,d(R^{e_J}\mu)(z)\\
&=\int |f(z)|^2\,d\Big(\prod_{i\in K}(1-e^{-\beta r_i}R^{e_i})\mu\Big)(z)\\
&=\int |f(z)|^2\, d\varepsilon(z)\geq \varepsilon(V_l)>0.
	\end{align*}
Thus we have our contradiction, and the proof is complete.
\end{proof}

In Theorem~\ref{theorem1}, we first chose an $r\in \N^k$ and then characterised KMS states of the dynamical system $\big(\NT(X),\alpha\big)$  for $\beta$ satisfying $\beta > \max_{i}r_i^{-1}\beta_{c_i}$.
Thus the range of possible inverse temperature  is dependent on the choice of $r\in \N^k$. When $r$ is a multiple of  $(\beta_{c_1},\dots,\beta_{c_k} )$, following the recent conventions for   $k$-graph algebras (see \cite{Ya1,Ya2,aHLRS}), we call the common value $\beta_c:=r_i^{-1}\beta_{c_i}$  the \textit{critical inverse temperature}. In particular, we are interested in  $r:=(\beta_{c_1},\dots,\beta_{c_k} )$ which gives the critical inverse temperature $\beta_c=1$. In this case, we refer to the associated dynamics $\alpha:t\mapsto\gamma_{e^{itr}}$ as  the  \textit{preferred dynamics}. Our next corollary is about the preferred dynamics.

\begin{cor}\label{cor-rescue}
	Let $h_1,\dots,h_k$ be  $*$-commuting and  surjective  local homeomorphisms on a compact Hausdorff space   $Z$, and let  $X$ be the associated product system over $\N^k$,   as in Proposition~\ref{prop1}. Define  $\beta_{c_i}$ by \eqref{criticalpoint} and let $r:=(\beta_{c_1},\dots,\beta_{c_k})$. 
\begin{enumerate}
	\item\label{cor-rescue-a} Then there is a KMS$_1$ state of $\big(\NT(X),\alpha\big)$.
	\item\label{cor-rescue-b} Suppose that  $r$ has rationally independent coordinates. Define $\bar{\alpha}: \R\rightarrow \Aut\OO(X)$ in terms of the gauge action $\bar{\gamma}$ by $\bar{\alpha}_t=\bar{\gamma}_{e^{ i t r}}$. If there is a KMS$_\beta$ state of  $(\OO(X),\bar{\alpha})$, then $\beta \leq 1$.
\end{enumerate}
\end{cor}

Motivated by \cite[Theorem~6.1]{AaHR}, for example, we suspect that at least one KMS$_1$ state of $\big(\NT(X),\alpha\big)$  should factor through a state of $(\OO(X),\bar{\alpha})$, but we have been unable to prove this. 

\begin{proof}[Proof of Corollary~\ref{cor-rescue}]
Choose a decreasing sequence $\{\beta_j\}$ such that $\beta_j\rightarrow 1$ and a probability measure $\nu$ on $Z$. Then $K_j:=\int f_{\beta_j}\,d\nu$ belongs to $[1,\infty)$, and $\varepsilon_j:=K_j^{-1}\nu$ satisfies $\int f_{\beta_j}\,d\varepsilon_j=1$. Thus for each $j$, part $(a)$ of Theorem~\ref{theorem1} gives  a KMS$_{\beta_j}$ state  $\phi_{\varepsilon_j}$ on $\big(\NT(X),\alpha\big)$.  Since  $\{\phi_{\varepsilon_j}\}$ is a sequence  in the compact unit ball of $C(Z)^*$,  by passing to a subsequence and relabeling, we may assume that  $\phi_{\varepsilon_j}\rightarrow \phi$ for some state $\phi$.
Now \cite[Proposition~5.3.23]{BRII} implies that $\phi$ is a KMS$_{1}$ state of $\big(\NT(X),\alpha\big)$. This gives (\ref{cor-rescue-a}). 

For (\ref{cor-rescue-b}), suppose that $r$ has rationally independent coordinates.  Then Corollary~\ref{cor6} implies that there exists $i$ such that $\beta\leq r_i^{-1}\beta_{c_i}$. Since $r_i=\beta_{c_i}$, we have $\beta\leq 1$.
\end{proof}

\section{Ground states and KMS$_\infty$ states}\label{sec8}

The next proposition is an   extension of    \cite[Proposition~8.1]{aHLRS} and \cite[Proposition~5.1]{aHLRS1} from dynamical systems of graph algebras  to  the dynamical system $(\NT(X),\alpha)$.
\begin{prop}\label{groundprop}
	Let $h_1,\dots,h_k$ be $*$-commuting and surjective local homeomorphisms on a compact Hausdorff space $Z$ and let $X$ be the associated product system  as in Proposition~\ref{prop1}.
	Suppose that $r\in (0,\infty)^k$  and $\alpha : \R\rightarrow \Aut\NT(X) $ is given in terms of the gauge action
 by $\alpha_t=\gamma_{e^{i t r}}$.  For each probability  measure $\varepsilon$ on $Z$ there is a unique KMS$_\infty$ state $\phi_\varepsilon$ such that
	
	\begin{align}\label{groundformula3}
	\phi_\varepsilon\big(\psi_m(x)\psi_n(y)^*\big)=\begin{cases}
	\int\langle y,x\rangle\,d\varepsilon  &\text{if $m=n=0$ }\\
	0  &\text{otherwise.}\\
	\end{cases}
	\end{align}	
	The map $\varepsilon\mapsto \phi_\varepsilon$ is an affine isomorphism of the simplex of probability measures on $Z$ onto the ground states of 
	$(\NT(X),\alpha)$, and  every ground state of $(\NT(X),\alpha)$ is a KMS$_\infty$ state.
\end{prop}
For the proof of this proposition,  we  need the  following  generalisation of \cite[Proposition~3.1(c)]{aHLRS1} and \cite[Proposition~2.1(b)]{aHLRS}. 
\begin{lemma}\label{groundlemma}
	Let $h_1,\dots,h_k$ be $*$-commuting and surjective local homeomorphisms on a compact Hausdorff space $Z$ and let $X$ be the associated product system  as in Proposition~\ref{prop1}.
	Suppose that $r\in (0,\infty)^k$  and $\alpha : \R\rightarrow \Aut\NT(X) $ is given in terms of the gauge action by $\alpha_t=\gamma_{e^{
			i t r}}$. Suppose   that $\beta>0$ and let $\phi$
	be a state on $\NT(X)$. Then $\phi$ is a ground state  of $(\NT(X),\alpha)$ if and only if
	\begin{align}\label{groundformula1}
	\phi\big(\psi_m(x)\psi_n(y)^*\big)=0 \text{ whenever  $r\cdot m>0$ or $r\cdot n>0$. }
	\end{align}	
\end{lemma}
\begin{proof}
	First notice that for every state $\phi$, $a+ib\in \C$ and $m,n,p,q\in \N^k$,  the definition of $\alpha$ implies that
	\begin{align}\label{groundformula2}
	\phi\big(\psi_m(x)\psi_n(y)^*\alpha_{a+ib}\big(\psi_p(s)\psi_q(t)^*\big)\big)\notag&
	=\big|e^{i(a+ib)r\cdot(p-q)}\phi\big(\psi_m(x)\psi_n(y)^*\psi_p(s)\psi_q(t)^*\big)\big|\\
	&=e^{-br\cdot(p-q)}\big|\phi\big(\psi_m(x)\psi_n(y)^*\psi_p(s)\psi_q(t)^*\big)\big|.
	\end{align}
	Now suppose that $\phi$ is a ground  state. Then
	\[ \big|\phi\big(\psi_m(x)\alpha_{a+ib}\big(\psi_n(y)^*\big)\big)\big|=e^{br\cdot n}\big|\phi\big(\psi_m(x)\psi_n(y)^*\big)\big|\]
	is bounded on the upper half plane $b>0$. Thus   $\phi\big(\psi_m(x)\psi_n(y)^*\big)=0$ whenever
	$r\cdot n>0$. Since
	$\phi\big(\psi_n(y)\psi_m(x)^*\big)=\overline{\phi\big(\psi_m(x)\psi_n(y)^*\big)}$, a symmetric calculation shows that $\phi\big(\psi_n(y)\psi_m(x)^*\big)=0$ whenever $r\cdot m>0$.
	
	Next suppose that $\phi$ satisfies \eqref{groundformula1}.   It follows from Lemma~\ref{application} that  there exist  
	$\{\xi_{i,j}:1\leq i,j\leq d\}\subset X_{m+p-n\wedge p}$ and $\{\eta_{i,j}:1\leq i,j\leq d\}\subset X_{q+n-n\wedge p}$ such that
	\begin{align*}
	\psi_m(x)\psi_n(y)^*\psi_p(s)\psi_q(t)^*
	=\sum_{ i,j=1}^d\psi_{m+p-n\wedge p}(\xi_{i,j})\psi_{q+n-n\wedge p}(\eta_{i,j})^*.
	\end{align*}
	Putting this in \eqref{groundformula2}, we have
	\begin{align*}
	\phi\big(\psi_m(x)\psi_n(y)^*\alpha_{a+ib}\big(\psi_p(s)\psi_q(t)^*\big)\big)
	&=e^{-br\cdot(p-q)}\Big|\phi\Big(\sum_{ i,j=1}^d\psi_{m+p-n\wedge p}(\xi_{i,j})\psi_{q+n-n\wedge p}(\eta_{i,j})^*\Big)\Big|.
	\end{align*}
	By assumption  \eqref{groundformula1}, this is zero (consequently is bounded) unless
	$r\cdot (m+p-n\wedge p)=0=r\cdot(q+n-n\wedge p).$
	So  suppose that $r\cdot (m+p-n\wedge p)=0=r\cdot(q+n-n\wedge p)$. Since $r\in(0,\infty)^k$, it follows that
	$ m+p-n\wedge p=0=q+n-n\wedge p.$
	Then
	\begin{align*}
	\phi\big(\psi_m(x)\psi_n(y)^*\alpha_{a+ib}\big(\psi_p(s)\psi_q(t)^*\big)\big)
	&=e^{-br\cdot (p-q)}\Big|\sum_{i,j=1}^ d\phi\big(\psi_{0}\big(\langle \eta_{i,j},\xi_{i,j}\rangle\big)\big)\Big|.
	\end{align*}
	Notice that  $q$ and $n-n\wedge p$ are both positive.   Then  $q+n-n\wedge p=0$ implies that
	$q=0.$ Now we have
	\begin{align*}
	\phi\big(\psi_m(x)\psi_n(y)^*\alpha_{a+ib}\big(\psi_p(s)\psi_q(t)^*\big)\big)
	&=e^{-br\cdot p}\Big|\sum_{i,j=1}^d\phi\big(\psi_{0}\big(\langle \eta_{i,j},\xi_{i,j}\rangle\big)\big)\Big|.
	\end{align*}
	Thus $\phi$  is bounded on the upper half plane $b>0$, and hence it is a ground state.
\end{proof}

\begin{proof}[Proof of Proposition~\ref{groundprop}]
	Suppose that $\varepsilon$ is a probability measure on $Z$. For each $1\leq i\leq k$, let $\beta_{c_i}$ be as in \eqref{criticalpoint}. Choose a  sequence $\{\beta_j\}$ such that $\beta_j\rightarrow \infty$ and each $\beta_j > \max_{i}r_i^{-1}\beta_{c_i}$.  For each $\beta_j$, let $f_{\beta_j}(z)$ be the function
	in Proposition~\ref{series}(a) and set $K_j:=\int f_{\beta_j}\,d\varepsilon$. Then $K_j$ belongs to $[1,\infty)$, and $\varepsilon_j:=K_j^{-1}\varepsilon$ satisfies $\int f_{\beta_j}\,d\varepsilon_j=1$.
	Now part $(a)$ of Theorem~\ref{theorem1} gives  a KMS$_{\beta_j}$ state  $\phi_{\varepsilon_j}$ on $\big(\NT(X),\alpha\big)$.  Since  $\{\phi_{\varepsilon_j}\}$ is a sequence  in the compact unit ball of $C(Z)^*$, by passing to a subsequence and relabeling, we may assume that 
	$\phi_{\varepsilon_j}\rightarrow \phi_\varepsilon$. Then $\phi_\varepsilon$ is  a KMS$_\infty$ state.

	We now show that $\phi_\varepsilon$ satisfies \eqref{groundformula3}. For each $\phi_{\varepsilon_j}$, we have
	\[\phi_{\varepsilon_j}\big(\psi_m(x)\psi_n(y)^*\big)=\delta_{m,n}e^{-\beta_j r\cdot n}\int\langle y,x\rangle\,d\varepsilon_j\,\text{ for all $m,n\in \N^k$}.\]	
	Thus $\phi_{\varepsilon_j}\big(\psi_m(x)\psi_n(y)^*\big)=0$ for $m\neq n$ and hence $\phi_\varepsilon\big(\psi_m(x)\psi_n(y)^*\big)=0$ if $m\neq n$.
	So we suppose that  $m=n$. If $n\neq0$, then $r\in (0,\infty)^k$ implies that $e^{-\beta_j r\cdot n}\rightarrow 0$, and
	again $\phi_\varepsilon\big(\psi_m(x)\psi_n(y)^*\big)=\lim_{j\rightarrow \infty} \phi_{\varepsilon_j}\big(\psi_m(x)\psi_n(y)^*\big)=0$.
	So we assume that  $m=n=0$.
	
	We aim to show that $K_j\rightarrow 1$ when $j\rightarrow \infty$. Fix $z\in Z$ and let $f_{\beta_j}(z)=\sum_{p\in \N^k}e^{-\beta_j r\cdot p}|h^{-p}(z)|$ as in Proposition~\ref{series}(a).  For each $p\in \N^k$ let	
	\begin{align*}
	g(p)=\begin{cases}
	1  &\text{if $p=0$ }\\
	0  &\text{if $p\neq0$ }.\\
	\end{cases}
	\end{align*}
	Then  $e^{-\beta_j r\cdot p}|h^{-p}(z)|\rightarrow g(p)$ as $j\rightarrow \infty$. 			
Notice that for each $j$, $e^{-\beta_j r\cdot p}|h^{-p}(z)|$ is dominated by $e^{-\beta_0 r\cdot p}|h^{-p}(z)|$. The dominated convergence theorem implies that
	\[f_{\beta_j}(z)=\sum_{p\in \N^k}e^{-\beta_j r\cdot p}|h^{-p}(z)|\rightarrow \sum_{p\in \N^k}g(p)=1 \qquad \text{as }j\rightarrow \infty.\]
		Also notice that   each  $f_{\beta_j}$ is dominated by $f_{\beta_0}$ and $\varepsilon$ is a probability measure. Then  another application of the  dominated convergence theorem implies that
	\[K_j=\int f_{\beta_j}\,d\varepsilon\rightarrow \int1\,d\varepsilon=1\qquad \text{as }j\rightarrow \infty.\]	
 Next we compute using the formula \eqref{equkms} for $\phi_{\varepsilon_j}$:
	\begin{align*}
	\phi_\varepsilon\big(\psi_m(x)\psi_n(y)^*\big)&=\lim_{j\rightarrow \infty} \phi_{\varepsilon_j}\big(\psi_m(x)\psi_n(y)^*\big)
	=\lim_{j\rightarrow \infty}\int\langle y,x\rangle\,d\varepsilon_j.
	\end{align*}
Since $\varepsilon_j=K_j^{-1}\varepsilon$,	
	\begin{align*}
	\phi_\varepsilon\big(\psi_m(x)\psi_n(y)^*\big)&=\lim_{j\rightarrow \infty}
	{K_j}^{-1}\int \langle y,x\rangle(z)\,d\varepsilon(z)=
	\int \langle y,x\rangle(z)\,d\varepsilon(z).
	\end{align*}
	Thus $\phi_\epsilon$ satisfies \eqref{groundformula3}.
	Since $\phi_\varepsilon\big(\psi_m(x)\psi_n(y)^*\big)$ vanishes for all $m\neq 0$ or $n\neq 0$, it also does for $r\cdot m\neq0$ or $ r\cdot n\neq0$. Then Lemma~\ref{groundlemma} says that $\phi_\varepsilon$ is a ground state.
	
	Next let $\phi$ be a ground state and suppose that $\varepsilon$ is the probability measure satisfying $\phi(\psi_0(a))=\int a\,d\varepsilon$ for all $a\in A$. Then the formulas \eqref{groundformula1} and  \eqref{groundformula3} for $\phi$ and $\phi_\varepsilon$  imply that $\phi=\phi_\varepsilon$. Thus $\varepsilon\mapsto \phi_\varepsilon$ maps the simplex of the probability measures of $Z$ onto the ground states, and it is clearly affine and injective. Since each $\phi_\varepsilon$ is by construction a KMS$_\infty$ state, it follows that every ground state is a KMS$_\infty$ state.
\end{proof}

\section{Shifts on the infinite-path spaces of $1$-coaligned $k$-graphs }\label{sec9}

In this section we apply our results to a finite $k$-graph $\Lambda$ with no sources. Its infinite path space $\Lambda^\infty$ is then compact. If $\Lambda$ is $1$-coaligned (see below), then the shift maps on $\Lambda^\infty$ $*$-commute.  In this case, our main results say that the Toeplitz algebra $\TT C^*(\Lambda)$ of $\Lambda$ is isomorphic to a subalgebra of  the Nica--Toeplitz algebra $\NT(X(\Lambda^\infty))$ of $X(\Lambda^\infty))$, and that every KMS$_\beta$ state of $\TT C^*(\Lambda)$ is the restriction of  a  KMS$_\beta$ state of $\NT(X(\Lambda^\infty))$. We start with some background on $k$-graphs.

Let $k\geq 1$. 	Suppose that  $\Lambda$ is a $k$-graph with vertex set $\Lambda^0$ and degree map $d:\Lambda\rightarrow\N^k$ in the sense of  \cite{KP}. 
For any $n\in \N^k$, we write $\Lambda^n:=\{\lambda\in \Lambda^*: d(\lambda)=n\}$.  All $k$-graphs considered here are finite in the sense that
 $\Lambda^n$ is finite for all $n\in \N^k$. Given $v,w\in \Lambda^0$,   $v\Lambda^n w$ denotes  $\{\lambda\in \Lambda^n: r(\lambda)=v \text { and } s(\lambda)=w\}$. We  say $\Lambda$ has  \textit{no sinks} if $\Lambda^n v\neq\emptyset$ for every $v\in \Lambda^0$ and $n\in \N^k$. Similarly,  $\Lambda$ has   \textit{no sources} if $v\Lambda^n\neq\emptyset$ for every $v\in \Lambda^0$ and  $n\in \N^k$.
For $\mu,\nu\in \Lambda$,  we write 
\[\Lambda^{\min}(\mu,\nu):=\{(\xi,\eta)\in \Lambda\times \Lambda: \mu\xi=\nu \eta \text{ and } d(\mu\xi)=d(\mu)\vee d(\nu)\}.\]  

Let  $\Omega_k:=\{(m,n)\in \N^k\times \N^k:m\leq n\}$. The set $\Omega_k$ becomes a $k$-graph with  $r(m,n)=(m,m), s(m,n)=(n,n),$  $(m,n)(n,p)=(m,p)$ and
$d(m,n)=n-m$. We identify $\Omega_k^0$ with $\N^k$ by $(m,m)\mapsto m$.
We call 
\[\Lambda^\infty:=\{z:\Omega_k\rightarrow \Lambda: \text{ $z$ is a functor intertwining  the degree maps}\}\]
the \textit{infinite-path space}  of $\Lambda$. 
For $p\in \N^k$, the shift map $\sigma^p:\Lambda^\infty\rightarrow \Lambda^\infty $ is defined by
$\sigma^p(z)(m,n)= z(m+p,n+p)$ for all $z\in \Lambda^\infty$ and $(m,n)\in \Omega_k$. Clearly $\sigma^p\circ \sigma^q=\sigma^q\circ \sigma^p$ for $p,q\in \N^k$. Observe that   $z=z(0,p)\sigma^p(z)$ for  $z\in \Lambda^\infty$ and $p\in \N^k$. 

For each $\lambda\in \Lambda$, let $Z(\lambda):=\{ z\in \Lambda^\infty: z(0, d(\lambda))=\lambda \}.$ Lemma~2.6 of \cite{KP}
says that $\Lambda^\infty$ is compact in the topology which has $\{Z(\lambda): \lambda\in \Lambda\}$ as a basis. Also for each $p\in \N^k$, the shift map $\sigma^p$ is a local homeomorphism on $\Lambda^\infty$ (see \cite[Remark 2.5]{KP}). 

Following \cite{RS,aHLRS}, a \textit{Toeplitz-Cuntz-Krieger $\Lambda$-family} in a $C^*$-algebra $B$ is a set  of partial isometries $\{S_\lambda: \lambda\in \Lambda\}$ such that
	\begin{enumerate}[(TCK1)]
		\item $\{S_v: v\in \Lambda^0\}$ is a set of mutually orthogonal projections,
		\item $S_\lambda S_\mu=S_{\lambda \mu}$ whenever $s(\lambda)=r(\mu)$,
		\item $S_\lambda^*S_\lambda=S_{s(\lambda)}$ for all $\lambda$,
		\item  $S_{v}\geq \sum_{\lambda\in v\Lambda^n}S_\lambda S_\lambda^*$ for all $v\in \Lambda^0$ and $n\in \N^k$, and 
		\item $S_\mu^*S_\nu=\sum_{(\xi,\eta)\in\Lambda^{\min}(\mu,\nu)} S_\xi S_\eta ^*$ for all $\mu,\nu\in \Lambda$. 
			\end{enumerate}
				We interpret empty sums as $0$. 	A Toeplitz-Cuntz-Krieger  $\Lambda$-family $\{S_\lambda: \lambda\in \Lambda\}$  is a \textit{Cuntz-Krieger $\Lambda$-family} if we also have 
	\begin{itemize}	
	 \item[(CK)] $\quad  S_{v}= \sum_{\lambda\in v\Lambda^n}S_\lambda S_\lambda^*$ for all $v\in \Lambda^0$ and $n\in \N^k$.
	\end{itemize}		
	   Lemma~3.1 of \cite{KP} says that  (TCK1)$-$(TCK3) together with (CK)	implies (TCK5).

		The Toeplitz algebra $\TT C^*(\Lambda)$ is  generated by a universal Toeplitz-Cuntz-Krieger $\Lambda$-family $\{s_\lambda: \lambda\in \Lambda\}$.  
	The Cuntz-Krieger algebra  $C^*(\Lambda)$ is the  quotient of $\TT C^*(\Lambda)$ by the ideal 
	$\langle s_{v}- \sum_{\lambda\in v\Lambda^n}s_\lambda s_\lambda^*:v\in \Lambda^0\rangle.$
	
	There is a strongly continuous \textit{gauge action} $\tilde{\gamma}:\T^k\rightarrow\TT C^*(\Lambda)$ such that $\tilde{\gamma}_z(s_\lambda)=z^{d(\lambda)}s_\lambda$. 
	Since $\tilde{\gamma}$ fixes the kernel of the quotient map, it induces a natural gauge action of $\T^k$ on $C^*(\Lambda)$. 
	
The next lemma  shows that it suffices to check (TCK5) for a subset of $\Lambda\times \Lambda$.
\begin{lemma}\label{remark62}
	Let $\Lambda$ be a finite $k$-graph. Suppose that $\{S_\lambda:
	\lambda\in \Lambda\}$ is a set of partial isometries	in a $C^*$-algebra $B$ which  satisfies (TCK1)$-$(TCK3). Suppose that for all $\mu,\nu\in \Lambda$ with $d(\mu)\wedge d(\nu)=0$ we have $S_\mu^*S_\nu=\sum_{(\xi,\eta)\in\Lambda^{\min}(\mu,\nu)} S_\xi S_\eta ^*$. Then $\{S_\lambda: \lambda\in \Lambda\}$ satisfies (TCK5).
\end{lemma}	

\begin{proof}
	Fix $\mu,\nu\in \Lambda$.  Factor $\mu=\mu'\mu''$ and $\nu=\nu'\nu''$ such that $d(\mu')=d(\nu')=d(\mu)\wedge d(\nu)$. Notice that
	\begin{align}\label{usingprim}
	 d(\mu'')=d(\mu)-d(\mu)\wedge d(\nu),\, 
	& d(\nu'')=d(\nu)-d(\mu)\wedge d(\nu)\text{ and } d(\mu'')\wedge d(\nu'')=0.
	\end{align}
	Now using (TCK2), (TCK3) and the identity $S_{r(\lambda)}S_{\lambda}=S_{\lambda}$, we have
	\begin{align*}
	S_\mu^*S_\nu\notag&=S_{\mu''}^*S_{\mu'}^*S_{\nu'}S_{\nu''}
=S_{\mu''}^*\delta_{\mu',\nu'}S_{s(\mu')}S_{\nu''}\\
	\notag&=\delta_{\mu',\nu'}S_{\mu''}^*S_{r(\mu'')}S_{\nu''}&\text{since } s(\mu')=r(\mu'')\\
&=\delta_{\mu',\nu'}S_{\mu''}^*S_{\nu''}.
	\end{align*}
	Since   $d(\mu'')\wedge d(\nu'')=0$, applying (TCK5) for $\mu'', \nu''$ gives
	\begin{align*}
	S_\mu^*S_\nu=& \delta_{\mu',\nu'}\sum_{(\xi,\eta)\in\Lambda^{\min}(\mu'',\nu'')} S_\xi S_\eta ^*.
	\end{align*}	
	Now it suffices to prove that 
	\begin{align*}
	\Big(\mu'=\nu'\, \text { and } \,(\xi,\eta)\in\Lambda^{\min}(\mu'',\nu'')\Big)&\Longleftrightarrow (\xi,\eta)\in\Lambda^{\min}(\mu,\nu).
	\end{align*}	
	To see this, suppose that  $\mu'=\nu'$ and $(\xi,\eta)\in\Lambda^{\min}(\mu'',\nu'')$.  Then $\mu''\xi=\nu''\eta $ implies that $\mu\xi=\nu\eta$.
	Since  $d(\mu''\xi)=d(\mu'')\vee d(\nu'')$ and $d(\mu'')\wedge d(\nu'')=0$, it follows  
	that $d(\mu''\xi)=d(\mu'')+ d(\nu'')$. This says    $d(\xi)=d(\nu'')$. Now  \eqref{usingprim} implies that
	\[d(\mu\xi)=d(\mu)+d(\nu'')=d(\mu)+d(\nu)-d(\mu)\wedge d(\nu)=d(\mu)\vee d(\nu).\]
	Thus $(\xi,\eta)\in\Lambda^{\min}(\mu,\nu)$.
	
	Next let $(\xi,\eta)\in\Lambda^{\min}(\mu,\nu)$. Since $\mu\xi=\nu\eta$, the factorisation property implies that $\mu'=\nu'$. Notice that
	\[d(\mu''\xi)=d(\mu\xi)-d(\mu')=d(\mu)\vee d(\nu)-d(\mu')=d(\mu)+d(\nu)-d(\mu)\wedge d(\nu)-d(\mu').\]
	Rearranging this and using \eqref{usingprim} we have
	\[d(\mu''\xi)=(d(\mu)-d(\mu)\wedge d(\nu))+(d(\nu)-d(\mu'))=d(\mu'')+d(\nu'').\]
	Thus $(\xi,\eta)\in\Lambda^{\min}(\mu'',\nu'')$.
\end{proof}

Following \cite[Definition~2.2]{MW}, we say 
	 a $k$-graph $\Lambda$ is \textit{$1$-coaligned} if for all $1\leq i\neq j\leq k$ and $(\lambda,\mu)\in \Lambda^{e_i}\times \Lambda^{e_j}$ with $s(\lambda)=s(\mu)$ there exists a unique pair $(\eta,\zeta)\in \Lambda^{e_j}\times \Lambda^{e_i}$ such that $\eta\lambda=\zeta\mu$.

The next lemma is contained in \cite[Theorem~2.3]{MW}.  Since  \cite{MW} has not been  published, we provide a brief  proof.
\begin{lemma}\label{shiftmaps}
	Let $\Lambda$ be a finite $1$-coaligned $k$-graph. Suppose that $0\leq i\neq j \leq k$. Then the shift maps $\sigma^{e_i}$ and $\sigma^{e_j}$ $*$-commute.
\end{lemma}	
\begin{proof}
	Let $w,z\in \Lambda^\infty$ such that
	$\sigma^{e_i}(z)=\sigma^{e_j}(w).$
	Notice  that $z=z(0,e_i)\sigma^{e_i}(z)$ and $w=w(0,e_j)\sigma^{e_j}(w)$.  It  follows that
	$z(0,e_i)$ and $w(0,e_j)$ have the same sources. Since $\Lambda $ is $1$-coaligned there exists a unique pair  $(\eta,\zeta)\in \Lambda^{e_j}\times \Lambda^{e_i}$ such that
	$\eta z(0,e_i)=\zeta w(0,e_j)=\lambda$, say.	Then  $x:=\lambda\sigma^{e_i}(z)\in \Lambda^\infty$ satisfies
	$\sigma^{e_j}(x)=z \text{ and } \sigma^{e_i}(x)=w$.
For the uniqueness, suppose that $y\in \Lambda^\infty$ also satisfies $\sigma^{e_j}(y)=z \text{ and } \sigma^{e_i}(y)=w$. Then $\mu:=y(0,e_i+e_j)$ satisfies $\mu(e_j,e_i+e_j)=y(e_j,e_i+e_j)=z(0,e_i)$ and $\mu(e_i,e_j+e_i)=w(0,e_j)$, and hence $\mu=\lambda$ by $1$-coalignedness. Thus $y=\mu \sigma^{e_i+e_j}(y)=\lambda\sigma^{e_i}(z)=x$, and $\sigma^{e_i}$ and $\sigma^{e_j}$ $*$-commute.
\end{proof}

Lemma~\ref{shiftmaps} makes it relatively easy to identify families of graphs that are $1$-coaligned. For example:

\begin{example}
Maloney and Willis showed in \cite[\S3]{MW} that \cite{PRW} provides many examples of $1$-coaligned $2$-graphs. Each graph $\Lambda$ in \cite{PRW} is determined by a set of ``basic data", and the condition on the basic data that is equivalent to $1$-coalignedness of $\Lambda$ also implies that $\Lambda$ is aperiodic (compare \cite[Corollary~3.4]{MW} and \cite[Theorem~5.2]{PRW}). Theorem~6.1 of \cite{PRW} then says further that $C^*(\Lambda)$ is simple, nuclear and purely infinite.
\end{example}

These examples are reassuring rather than unexpected: the whole point of \cite{PRW} was to find $2$-graphs such that subshifts on path space give dynamical systems of algebraic origin. One naturally wonders if there are other familiar examples. The next lemma helps resolve this for the important family of graphs with a single vertex.

\begin{lemma}\label{1coal}
Suppose that $\Lambda$ is a finite $2$-graph with one vertex.  For $f\in \Lambda^{e_2}$, define $\rho_f:\Lambda^{e_1}\to \Lambda^{e_1}$ by $\rho_f(e)=(ef)(e_2,e_1+e_2)$. Then $\Lambda$ is $1$-coaligned if and only if $\rho_f$ is a bijection for every $f\in \Lambda^{e_2}$.
\end{lemma} 

\begin{proof}
Suppose that $\Lambda$ is $1$-coaligned, and fix $f\in \Lambda^{e_2}$. Suppose that $e, e'\in \Lambda^{e_1}$ and that $\rho_f(e)=\rho_f(e')$. Then $(ef)(e_2, e_1+e_2)= (e'f)(e_2, e_1+e_2)=h$, say. Let  $l=(ef)(0, e_2)$ and $m=(e'f)(0, e_2)$. Then $l, m\in \Lambda^{e_2}$ such that $ef=lh$ and $e'f=mh$. But now $1$-coalignedness for the pair $(h,f)$ implies that $(l,e)=(m,e')$. Thus $e=e'$, and $\rho_f$ is one-to-one. Since $\rho_f:\Lambda^{e_1}\to \Lambda^{e_1}$ and $\Lambda^{e_1}$ is finite, it is also onto.

For the converse, suppose that $\rho_f$ is a bijection for all $f\in \Lambda^{e_2}$. Fix $(e, f)\in \Lambda^{e_1}\times  \Lambda^{e_2}$.  Since $\rho_f$ is a bijection, there exists unique $h\in \Lambda^{e_1}$ such that $\rho_f(h)=e$, that is, $(hf)(e_2, e_1+e_2)=e$.  Take $l=(hf)(0, e_2)$. Then $(l,h)\in\Lambda^{e_2}\times\Lambda^{e_1}$ satisfies $le=hf$. For uniqueness, suppose also that $(k,g)\in\Lambda^{e_2}\times\Lambda^{e_1}$ satisfies $ke=gf$. Then $\rho_f(g)=e=\rho_f(h)$ implies $g=h$. 
Now uniqueness of factorisation implies $l=k$. So $(l,h)$ is unique, and $\Lambda$ is $1$-coaligned.
\end{proof}

\begin{example}
Suppose that $\Lambda$ is a $2$-graph with a single vertex, $\Lambda^{e_1}=\{e,f\}$ and $\Lambda^{e_2}=\{g,h\}$, and factorisations 
\[
eg=he,\quad eh=hf,\quad fg=gf,\quad\text{and}\quad fh=ge.
\]
Then $\rho_h$ flips $e$ and $f$, and $\rho_g$ is the identity. So, either by inspection or by Lemma~\ref{1coal}, $\Lambda$ is $1$-coaligned.
\end{example}

	Let $\Lambda$ be a finite $1$-coaligned $k$-graph with no sinks. Then the shift maps  $\sigma^{e_1},\dots,\sigma^{e_k}$ are surjective , and they $*$-commute by Lemma~\ref{shiftmaps}. 
	We write $X(\Lambda^\infty)$ for the  product system associated to $\sigma^{e_1},\dots,\sigma^{e_k}$.
	We use $\psi$  for the universal Nica-covariant representation. We write $X_m(\Lambda^\infty)$ for the  fibre associated to $m\in \N^k$. We write $\varphi_m$ for the left action of $A$ on the fibre $X_m(\Lambda^\infty)$.
	Recall that the multiplication formula in $X(\Lambda^\infty)$ is\footnote{In previous sections we wrote the multiplication in terms of isomorphisms $\sigma$ between fibres, for example, $xy(z)=\sigma(x\otimes y)(z)$. Unfortunately, in this section we use the letter $\sigma$ for the shifts. Therefore we suppress $\sigma$ when  writing products.}
	\begin{align} \label{multiplicationformula}
	xy(z)=x(z)y(\sigma^m(z))\text{ for } x\in X_m, y\in X_n, z\in \Lambda^\infty.
	\end{align}
The next proposition is an analogue of \cite[Proposition~7.1]{AaHR}.
\begin{prop}\label{example}
	Let $\Lambda$ be a finite $1$-coaligned $k$-graph with no sinks or sources.  For  each $\lambda\in \Lambda$, 
	let $S_\lambda:=\psi_{d(\lambda)}(\chi_{Z(\lambda)})$. Then
	\begin{itemize}
		\item[(a)] The set $\{S_\lambda:\lambda\in \Lambda\}$ is a
		Toeplitz-Cuntz-Krieger $\Lambda$-family in $\NT(X(\Lambda^\infty))$. The  homomorphism $\pi_S:\TCL\rightarrow\NT(X(\Lambda^\infty))$ is injective and intertwines  the respective gauge actions of $\T^k$ (that is, $\pi_S\circ \tilde{\gamma}=\gamma \circ \pi_S$).
		\item[(b)] Let $q:\NT(X(\Lambda^\infty))\rightarrow \OO(X(\Lambda^\infty))$ be the quotient map as in Lemma~\ref{fixdef}. Then $\{q\circ S_\lambda:\lambda\in \Lambda\}$ is a Cuntz-Krieger $\Lambda$-family in $\OO(X(\Lambda^\infty))$. The corresponding homomorphism $\pi_{q\circ S}:C^*(\Lambda)\rightarrow\OO(X(\Lambda^\infty))$ is an  isomorphism and intertwines the respective gauge actions of $\T^k$.
	\end{itemize}	
\end{prop}
To prove this, we need the following  results.

\begin{prop}\label{approximation-shift}
Let $\Lambda$ be a finite $1$-coaligned $k$-graph with no sources. Suppose  that $m\wedge n=0$. Then 
	
\begin{itemize}
		\item[(a)]   $\{\chi_{Z(\xi)}:\xi\in \Lambda^{m}\}$ and $\{\chi_{Z(\xi)}\circ \sigma^n:\xi\in \Lambda^{m}\}$ are  Parseval frames for $X_m(\Lambda^\infty)$;
		\item[(b)]   $\{\chi_{Z(\eta)}:\eta\in \Lambda^{n}\}$ and $\{\chi_{Z(\eta)}\circ \sigma^m:\eta\in \Lambda^{n}\}$ are  Parseval frames for $X_n(\Lambda^\infty)$; and
		\item[(c)] for $\mu\in \Lambda^n$ and $\nu\in \Lambda^m$, we have	
		\begin{align}\label{swap-shift}
\psi_n(\notag&\chi_{Z(\mu)})^*\psi_m(\chi_{Z(\nu)})\\
&=\sum_{\xi\in \Lambda^m,\,\eta\in \Lambda^n}\psi_m\big(\big\langle \chi_{Z(\mu)},\chi_{Z(\eta)}\circ\sigma^m\big\rangle \cdot \chi_{Z(\xi)}\big)\psi_n\big(\big\langle \chi_{Z(\nu)},\chi_{Z(\xi)}\circ\sigma^n\big\rangle\cdot\chi_{Z(\eta)}\big)^*.
	\end{align}
	\end{itemize}	
\end{prop}
\begin{proof}
For (a), recall that the fibre $X_m(\Lambda^\infty)$ is the topological graph $(\Lambda^\infty,\Lambda^\infty,\id,\sigma^m)$. By  \cite [Remark~2.5]{KP},
    $\{Z(\xi): d(\xi)=m\}$ forms a partition of $\Lambda^\infty$ and therefore the set $\{\chi_{Z(\xi)}:\xi\in \Lambda^{m}\}$ is a partition of unity. Furthermore,   $\sigma^m$ is injective on each $\supp \chi_{Z(\xi)}$. Since the proof of Lemma~\ref{pframe}(a) only requires $\sigma^m$  to be injective on $\supp \chi_{Z(\xi)}$,  we can apply 
Lemma~\ref{pframe}(a) with $X(E_1):=X_m(\Lambda^\infty)$ and $g:=\sigma^n$ to get (a).

Part (b) follows from part (a).

Part (c) is very similar to  Proposition~\ref{HELPFORMULA}(b). The only difference is that  Proposition~\ref{HELPFORMULA}(b) considers one common  partition of unity  
$\{\tau_i\}$ and  corresponding  Parseval frames $\{\tau_i\circ h^n\}$ and  $\{\tau_j\circ h^m\}$ for the fibres $X_m$ and $X_n$. Here we use different partitions of unity $\{\chi_{Z(\xi)}:\xi\in \Lambda^{m}\}$ and $\{\chi_{Z(\eta)}:\eta\in \Lambda^{n}\}$ to get the Parseval frames 
$\{\chi_{Z(\xi)}\circ \sigma^n\}$ and $\{\chi_{Z(\eta)}\circ \sigma^m\}$ for the fibres $X_m(\Lambda^\infty)$ and $X_n(\Lambda^\infty)$. The proof of Proposition~\ref{HELPFORMULA}(b)  uses the reconstruction formulas of Parseval frames and  that $\sigma^m$ and $\sigma^n$  are  injective on  $\supp \chi_{Z(\xi)}$ and $\supp \chi_{Z(\eta)}$. So the proof of \eqref{swap-shift} caries over  from the proof of the Proposition~\ref{HELPFORMULA}(b).
\end{proof}

\begin{remark}
We will need \eqref{swap-shift} to prove that $\{S_\lambda:\lambda\in \Lambda\}$ in  Proposition~\ref{example}(a) satisfies (TCK5). To see (TCK5), we try to rewrite $S_\mu^*S_\nu$ as  $\sum_{\alpha\in \Lambda^p,\beta\in \Lambda^q} S_\alpha S_\beta^*$ for suitable $p,q$, and then reduce this to 
\[S_\mu^*S_\nu=\sum_{(\xi,\eta)\in\Lambda^{\min}(\mu,\nu)}S_\xi S_\eta^*.\] 
We first thought that we would do this using Proposition~\ref{HELPFORMULA}(b): Check that the set $\{\chi_{Z(\mu)}:\mu\in \Lambda^{m+n}\}$ is a partition of unity  such that  $\sigma^m|_{\supp \chi_{Z(\mu)}}$ and $\sigma^n|_{\supp \chi_{Z(\mu)}}$ are injective for all $\mu\in \Lambda^{m+n}$, and then  apply Proposition~\ref{HELPFORMULA}(b) to get
\[S_\mu^*S_\nu\notag=\sum_{\xi,\eta\in \Lambda^{m+n}}\psi_m\big(\big\langle \chi_{Z(\mu)},\chi_{Z(\eta)}\circ\sigma^m\big\rangle \cdot \chi_{Z(\xi)}\big)\psi_n\big(\big\langle \chi_{Z(\nu)},\chi_{Z(\xi)}\circ\sigma^n\big\rangle\cdot\chi_{Z(\eta)}\big)^*.\]
But we could not  reduce this sum  to $\sum_{(\xi,\eta)\in\Lambda^{\min}(\mu,\nu)}S_\xi S_\eta^*$. The new formula \eqref{swap-shift} is easier to simplify. 
\end{remark}

\begin{proof}[Proof of Proposition~\ref{example}(a)]
For (TCK1),   we note that $\{\chi_{Z(v)}:v\in \Lambda^0\}$ are mutually orthogonal projections in $C(\Lambda ^\infty)$. Since $\psi_0$ is a homomorphism and $S_v=\psi_0(\chi_{Z(v)})$, it follows that    $\{S_v: v\in \Lambda^0\}$ are mutually orthogonal projections in $\NT(X(\Lambda^\infty))$. 
	  
Next we show that for each $\lambda\in \Lambda$,  $S_\lambda$ is a partial isometry.
Since $\chi_{Z(\lambda)}\in X_{d(\lambda)}(\Lambda^\infty)$, a calculation in the fibre $X_{d(\lambda)}(\Lambda^\infty)$ shows that
	\begin{align*} 
	\langle\chi_{Z(\lambda)},\chi_{Z(\lambda)}\rangle(z)
	&=
	\sum_{\sigma^{d(\lambda)}(w)=z}
	\overline{\chi_{Z(\lambda)}(w)}\chi_{Z(\lambda)}(w)\\
	&=
	\big|\big\{w: \sigma^{d(\lambda)}(w)=z \text{ and } w\in Z(\lambda)\big\}\big|\\
	&=\begin{cases}
	0  &\text{if $z\notin Z\big(s(\lambda)\big)$}\\
	1  &\text{if $z\in Z\big(s(\lambda)\big)$}\\
	\end{cases}\\
	&=\chi_{Z(s(\lambda))}(z).
	\end{align*}
Using  (T2) we get
\[
S_{\lambda}^*S_\lambda=\psi_{d(\lambda)}(\chi_{Z(\lambda)})^*\psi_{d(\lambda)}(\chi_{Z(\lambda)})=\psi_0(\langle\chi_{Z(\lambda)},\chi_{Z(\lambda)}\rangle)=\psi_0 (\chi_{Z(s(\lambda))})=S_{s(\lambda)}.
\]
Since $S_{s(\lambda)}$ is a projection,  $S_{\lambda}$ is a partial isometry.  Notice that this also establishes (TCK3).

For  (TCK2), let $\lambda, \mu\in \Lambda$ such that $s(\lambda)=r(\mu)$. The multiplication formula \eqref{multiplicationformula}  for $\chi_{Z(\lambda)}\in X_{d(\lambda)}(\Lambda^\infty)$ and  $\chi_{Z(\mu)}\in X_{d(\mu)}(\Lambda^\infty)$ implies that
\[
	\big(\chi_{Z(\lambda)}\chi_{Z(\mu)}\big)(z)=\chi_{Z(\lambda)}(z)
	\chi_{Z(\mu)}\big(\sigma^{d(\lambda)}(z)\big)
	=\chi_{Z(\lambda\mu)}(z).
\]
Now, using (T2), 
\[
S_\lambda S_\mu =\psi_{d(\lambda)}(\chi_{Z(\lambda)})\psi_{d(\mu)}(\chi_{Z(\mu)})=\psi_{d(\lambda)+d(\mu)}(\chi_{Z(\lambda)}\chi_{Z(\mu)})=\psi_{d(\lambda\mu)}(\chi_{Z(\lambda\mu)})=S_{\lambda\mu}
\]
which is (TCK2).
		
We will need (TCK5) for the proof of  (TCK4). So we first check (TCK5).   Let $\mu, \nu\in \Lambda$. By Lemma~\ref{remark62} we may assume that $d(\mu)\wedge d(\nu)=0$. For  convenience,  let $m:=d(\nu)$ and $n:=d(\mu)$.  Applying Proposition~\ref{approximation-shift}(c) to $\mu,\nu$ gives
	\begin{equation}\label{T5}
	S_\mu^*S_\nu
	=\sum_{\xi\in \Lambda^m,\,\eta\in \Lambda^n}\psi_m\big(\big\langle \chi_{Z(\mu)},\chi_{Z(\eta)}\circ\sigma^m\big\rangle \cdot \chi_{Z(\xi)}\big)\psi_n\big(\big\langle \chi_{Z(\nu)},\chi_{Z(\xi)}\circ\sigma^n\big\rangle\cdot\chi_{Z(\eta)}\big)^*.
	\end{equation}
	We now consider a summand for fixed $\xi$ and $\eta$. We have
	\begin{align*}
	\big(\big\langle \chi_{Z(\mu)},\chi_{Z(\eta)}\circ\sigma^m\big\rangle \cdot \chi_{Z(\xi)}\big)(z)
	&=
	\big\langle \chi_{Z(\mu)},\chi_{Z(\eta)}\circ\sigma^m\big\rangle(z ) \chi_{Z(\xi)}(z)\\
	&=
	 \chi_{Z(\xi)}(z)\sum_{\sigma^n(w)=z }\chi_{Z(\mu)}(w)\chi_{Z(\eta)}\big(\sigma^m(w)\big)\\
	&=
	\begin{cases}
	1  &\text{if $z\in Z(\xi)$,  and  $\exists\alpha\in \Lambda^m$ such that $\mu\xi=\alpha\eta$}\\
	0  &\text{otherwise}
	\end{cases}\\
	&=
	\begin{cases}
	\chi_{Z(\xi)}(z)&\text{if $\exists\alpha\in \Lambda^m$ such that $\mu\xi=\alpha\eta$} \\
	0  &\text{otherwise.}
	\end{cases}
	\end{align*}
	A similar calculation shows that
	\begin{align*}
	\big(\big\langle \chi_{Z(\nu)},\chi_{Z(\xi)}\circ\sigma^n\big\rangle\cdot\chi_{Z(\eta)}\big)(z)
	&=
	\begin{cases}
	\chi_{Z(\eta)}(z) &\text{if } \exists\beta\in \Lambda^n \text { such that } \nu\eta=\beta\xi\\
	0  &\text{otherwise}.
	\end{cases}
	\end{align*}
	Fix a nonzero $\xi$-$\eta$-summand. Then there exist $\alpha, \beta$ such that $\mu\xi=\alpha\eta \,\text{ and }\,\nu\eta=\beta\xi$. In particular 	 $s(\xi)=s(\eta)$ and $(\xi,\eta)\in\Lambda^{\min}(\mu,\nu)$. Then  $1$-coalignedness gives   $\alpha=\nu$ and $\beta=\mu$. Now the 
 sum in \eqref{T5} collapses to
	\begin{align*}
	S_\mu^*S_\nu
	&=\sum_{(\xi,\eta)\in\Lambda^{\min}(\mu,\nu)}\psi_m\big(\chi_{Z(\xi)}\big)\psi_n\big(\chi_{Z(\eta)}\big)^*
	=\sum_{(\xi,\eta)\in\Lambda^{\min}(\mu,\nu)}S_\xi S_\eta^*,
	\end{align*}
	which completes our proof of (TCK5).
	
	To see (TCK4),
	let $v\in \Lambda^0$ and $n\in \N^k$. Suppose that $\lambda, \mu\in v\Lambda^n$ and $\lambda\neq \mu$.  Since  $\lambda$ and $\mu$ are two different paths with the same degree and the same range, $\Lambda^{\min}(\lambda,\mu)=\emptyset$. Applying (TCK5) shows that $S_\lambda^*S_\mu=0$, and hence 
	$S_\lambda(S_\lambda^*S_\mu)S_\mu^*=0$.
	By (TCK2) we have $S_vS_\lambda S_{\lambda}^*=S_\lambda S_{\lambda}^*$ and hence $S_v\geq S_\lambda S_{\lambda}^*$. Thus 
	$S_{v}\geq \sum_{\lambda\in v\Lambda^n}S_\lambda S_\lambda^*$.	
	We  have now proved (TCK4) and therefore   $\{S_\lambda:\lambda\in \Lambda\}$ is a
	Toeplitz-Cuntz-Krieger $\Lambda$-family in $\NT(X(\Lambda^\infty))$.
	
	To see that the induced  homomorphism $\pi_S$ is injective, by \cite[Theorem~8.1]{RS}, it suffices to fix  $v\in \Lambda^0$ and  $n\in \N^k_+$, and  check
	that $S_{v}\gneq\sum_{\lambda\in v\Lambda^n}S_\lambda S_\lambda^*$.  Let $T$ be the Fock representation  of $X(\Lambda^\infty)$.  Then  $T_*:\NT(X(\Lambda^\infty))\rightarrow \LL(F(X(\Lambda^\infty)))$ satisfies
	\begin{align*}
	T_*\Big( S_{v}- \sum_{\lambda\in v\Lambda^n}S_\lambda S_\lambda^*\Big)
	&=T_0(\chi_{Z(v)})- \sum_{\lambda\in v\Lambda^n}T_{n}(\chi_{Z(\lambda)}) T_{n}(\chi_{Z(\lambda)})^*.
	\end{align*}
	The adjoint formula \eqref{fockadjoint} for the Fock representation  says that  $T_{n}(\chi_{Z(\lambda)})^*$ vanishes  on the $0$-summand of the Fock module $F(X(\Lambda^\infty))$.   The left action of $C(Z)$ on each fibre in injective, and it follows that  $T_0$ is injective.  Since $\Lambda$ has no  sources,    $T_0(\chi_{Z(v)})\neq 0$.
	Thus
	$T_0(\chi_{Z(v)})\neq \sum_{\lambda\in v\Lambda^n}T_{n}(\chi_{Z(\lambda)}) T_{n}(\chi_{Z(\lambda)})^*$.
	An application of the injectivity of  $T_*$ gives $S_{v}\gneq \sum_{\lambda\in v\Lambda^n}S_\lambda S_\lambda^*$, as  required.
		
	Finally, since the gauge actions on $\TCL$ and $\NT(X(\Lambda^\infty))$ satisfy  ${\tilde{\gamma}}_z(s_\lambda)=z^{d(\lambda)}s_\lambda$  and $\gamma_z(\psi_m(x))=z^m\psi_m(x)$, respectively, we have $\pi_{S}\circ \tilde{\gamma}=\gamma\circ \pi_{S}$.
\end{proof}

\begin{proof}[Proof of Proposition~\ref{example}(b)]
	It suffices to check (TCK1)$-$(TCK3) and (CK).
	Since the quotient map $q$  is a $*$-homomorphism, and   $\{ S_\lambda:\lambda\in \Lambda\}$ satisfies (TCK1)$-$(TCK3), so does $\{q\circ S_\lambda:\lambda\in \Lambda\}$. 
	
	Towards (CK), note that  $q\circ \psi$ is a universal Cuntz--Pimsner covariant 
	representation of $X(\Lambda^\infty)$ (see \cite[Proposition~2.9]{Fo}).   For convenience let $\rho:=q\circ \psi$. Then the restriction $\rho$ on each fibre $X_n$ is $\rho_n=q\circ \psi_n$. 
	Let $\mu\in \Lambda^n, n\in \N^k$. Next we show that the
	left action of $\chi_{Z(\mu)}$ on the fibre $X_n$ is by the finite rank operator $\Theta_{\chi_{Z(\mu)},\chi_{Z(\mu)}}$.
	To see this, take  $x\in X_n(\Lambda^\infty)$ and $z\in \Lambda^\infty$.  Then
	\begin{align*}
	(\Theta_{\chi_{Z(\mu)},\chi_{Z(\mu)}}(x))(z)\notag&=(\chi_{Z(\mu)}\cdot \langle\chi_{Z(\mu)}, x\rangle)(z)\\
	&=\chi_{Z(\mu)}(z)\langle\chi_{Z(\mu)}, x\rangle(\sigma^n(z))\\
	\notag&=\chi_{Z(\mu)}(z) \sum_{\sigma^n(w)=\sigma^n(z)}\overline{\chi_{Z(\mu)}(w)} x(w).
	\end{align*}
	This vanishes unless $z,w\in Z(\mu)$. Since $\mu\in \Lambda^n$ and  $w,z\in Z(\mu)$, the equation $\sigma^n(w)=\sigma^n(z)$ has a  unique  solution $z$ and therefore  the sum collapses to $\overline{\chi_{Z(\mu)}(z)} x(z)$. It follows that $\big(\Theta_{\chi_{Z(\mu)},\chi_{Z(\mu)}}(x)\big)(z)=\chi_{Z(\mu)}(z)x(z)$. Thus 
	\begin{align}\label{leftaction11}
	\Theta_{\chi_{Z(\mu)},\chi_{Z(\mu)}}&=\varphi_{d(\mu)}(\chi_{Z(\mu)})\text{ for } \mu\in \Lambda^n, n\in \N^k,
	\end{align}
	 and the action of $\chi_{Z(\mu)}$ on $X_n$ is by $\Theta_{\chi_{Z(\mu)},\chi_{Z(\mu)}}$.
		
	Now we can verify (CK).  Let $v\in \Lambda^0$ and $n\in \N^k$. Then a routine calculation shows that
	\begin{align*}
	\sum_{\lambda\in v\Lambda^n}(q\circ S_\lambda) (q\circ S_\lambda)^*&=\sum_{\lambda\in v\Lambda^n}\rho_{d(\lambda)}(\chi_{Z(\lambda)}) \rho_{d(\lambda)}(\chi_{Z(\lambda)})^*\\
	&=\sum_{\lambda\in v\Lambda^n} \rho^{(d(\lambda))}(\Theta_{\chi_{Z(\lambda)},\chi_{Z(\lambda)}})\\
	&= \sum_{\lambda\in v\Lambda^n} \rho^{(d(\lambda))}(\varphi_{d(\lambda)}(\chi_{Z(\lambda)}))\quad \text{by $\eqref{leftaction11}$.}
	\end{align*}
	Since $\rho$ is Cuntz--Pimsner-covariant,
	\begin{align*}
	\sum_{\lambda\in v\Lambda^n}(q\circ S_\lambda) (q\circ S_\lambda)^*
	&=\sum_{\lambda\in v\Lambda^n} \rho_0(\chi_{Z(\lambda)})
	=\rho_0 \Big(\sum_{\lambda\in v\Lambda^n} \chi_{Z(\lambda)}\Big)
	=q\circ \psi_0 \big( \chi_{Z(v)}\big)
	=q\circ (S_v).
	\end{align*}
	Thus (CK) holds and $\{q\circ S_\lambda:\lambda\in \Lambda\}$ forms a Cuntz-Krieger $\Lambda$-family in $\OO(X(\Lambda^\infty))$. 
	
	The universal property  gives a homomorphism $\pi_{q\circ S}:C^*(\Lambda)\rightarrow\OO(X(\Lambda^\infty))$, and  
$\pi_{q\circ S}$ intertwines the gauge actions. Since $\Lambda$ has no sources and  $\rho_0$ is injective (by \cite[Lemma~3.15]{SY}), it follows that $\rho_0 ( \chi_{Z(v)})\neq 0$ for all $v\in \Lambda^0$. Now the gauge-invariant uniqueness theorem (see \cite[Theorem~3.4]{KP}) implies that $\pi_{q\circ S}$ is injective.
	
	To show that  $\pi_{q\circ S}$ is surjective, note that $\OO(X(\Lambda^\infty))$ is generated by $\rho(X(\Lambda^\infty))$. Also, by the Stone-Weierstrass theorem,  $\{\chi_{Z(\lambda)}: \lambda\in \Lambda\}$ spans  a dense $*$-subalgebra of $C(\Lambda^\infty)$.
	Since the norm of   $X(\Lambda^\infty)$
	is equivalent to  $\|\cdot\|_\infty $, the elements $\{\chi_{Z(\lambda)}:\lambda\in \Lambda\}$ span  a dense subspace of  $X(\Lambda^\infty)$.
	Thus it is enough for us  to show that  $\rho_m(\chi_{Z(\mu)})$ lies in the range of $\pi_{q\circ S}$ for all $m,n\in \N^k$ and $\mu\in \Lambda^n$.
	
	We first check this for $m=0$ and  for all $\mu\in \Lambda^n$. Since $\rho$ is Cuntz--Pimsner covariant, a routine calculation using \eqref{leftaction11} shows that
	\begin{align}\label{casezero}
	\rho_0(\chi_{Z(\mu)})&=\rho^{(d(\mu))}\big(\varphi_{d(\mu)}(\chi_{Z(\mu)})\big)
	=\rho^{(d(\mu))}\big(\Theta_{\chi_{Z(\mu)},\chi_{Z(\mu)}}\big)
	=\rho_{d(\mu)}(\chi_{Z(\mu)}) \rho_{d(\mu)}(\chi_{Z(\mu)})^*.
		\end{align}
	This equals $(q\circ S_\mu )(q\circ S_\mu) ^*$ which belongs to the range of $\pi_{q\circ S}$.
	
	Now let $m\neq 0$ and take $\mu\in \Lambda^n$. Notice that $\chi_{Z(\mu)}= \sum_{\nu\in s(\mu)\Lambda^m}\chi_{Z(\mu\nu)}$. Each $\nu$-summand is the pointwise multiplication of  $\chi_{Z(\mu\nu(0,m)}$ and $\chi_{Z(\mu\nu(m,m+n))}\circ \sigma^m$. 
	This is exactly the  right action of $\chi_{Z(\mu\nu(m,m+n))}$ on
	 $\chi_{Z(\mu\nu(0,m))}\in X_m(\Lambda^\infty)$. It follows
	\begin{align*}
	&\rho_m(\chi_{Z(\mu)})
	=\rho_m\Big(\sum_{\nu\in s(\mu)\Lambda^m}\chi_{Z(\mu\nu(0,m))}\cdot\,  \chi_{Z(\mu\nu(m,m+n))}\Big)\\
	&=\sum_{\nu\in s(\mu)\Lambda^m}\rho_m(\chi_{Z(\mu\nu(0,m))})\rho_0(\chi_{Z(\mu\nu(m,m+n))})\\
	&=\sum_{\nu\in s(\mu)\Lambda^m} (q\circ S_{\mu\nu(0,m)})\rho_0(\chi_{Z(\mu\nu(m,m+n))}),
	\end{align*}
	which lies in the range of $\pi_{q\circ S}$ by \eqref{casezero}, as required.
\end{proof}

\subsection*{KMS states on the Toeplitz algebras}
Here we want  to see the relationship between  KMS states of the $C^*$-algebras $\TCL$ and $\NT(X(\Lambda^\infty))$. The KMS states of $\TCL$ are described thoroughly in \cite[Theorem~6.1]{aHLRS}. We now apply Theorem~\ref{theorem1} to characterise  KMS states of $\NT(X(\Lambda^\infty))$.
It follows from \cite[Proposition~7.3]{AaHR} that for the shift maps $\sigma^{e_i} (1\leq i\leq k)$ on $\Lambda^\infty$, each $\beta_{c_i}$ in Theorem~\ref{theorem1} is exactly  the critical inverse temperature $\ln \rho(A_i)$ used in \cite[Theorem~6.1]{aHLRS}.
Thus  the range of possible inverse temperatures studied in Theorem~\ref{theorem1} is the same as that of \cite[Theorem~6.1]{aHLRS}. Now when we use Proposition~\ref{example}(a) to view $\TCL$ as a $C^*$-subalgebra of $\NT(X(\Lambda^\infty))$,  restricting KMS states of $\NT(X(\Lambda^\infty))$ gives KMS states of $\TCL$ with the same inverse temperature.
We expect from our results in \cite[Corollary~7.6]{AaHR} to see that for the common inverse temperatures described in Theorem~\ref{theorem1} and  \cite[Theorem~6.1]{aHLRS}, all KMS states of $\TCL$ arise as restrictions of KMS states of $\NT(X(\Lambda^\infty))$. To see this, we need  generalisations of \cite[Proposition~7.4, Corollary~7.5]{AaHR}, which are stated in Proposition~\ref{restriction} and Corollary~\ref{restrictioncor} below. Their proofs  are very similar to those of \cite{AaHR} (using our operator $R$ from \eqref{defR}  in place of the $R$ from \cite{AaHR}), and so we omit their proofs and  refer the reader to \cite{AaHR}. 

We keep the notation Theorem~\ref{theorem1} to emphasis the parallels with \cite[Theorem~6.1]{aHLRS}. To avoid a clash, we write   $\delta$ for the measure $\varepsilon$ in Theorem~\ref{theorem1}, and choose $\varepsilon$ for the vectors in $[1,\infty)^{\Lambda^0}$ appearing in \cite[Theorem~6.1]{aHLRS}. Otherwise we keep the notation of  Theorem~\ref{theorem1}.

\begin{prop}\label{restriction}
	Suppose that  $\Lambda$ is a finite  $1$-coaligned $k$-graph with no sources and no sinks.  For $1 \leq i \leq k$, let $A_i\in M_{\Lambda^0}\big([0,\infty)\big)$ be the matrix with entries $A_i(v,w)=|v\Lambda^{e_i}w|$. Suppose that $r\in (0,\infty)^k$
	satisfies $\beta r_i> \ln \rho(A_i)$ for all $1 \leq i \leq k$. Let $\alpha : \R\rightarrow \Aut\NT(X(\Lambda^\infty))$ and $\tilde{\alpha}: \R\rightarrow \Aut\TT C^*(\Lambda) $ be given in terms of the  gauge actions by $\alpha_t=\gamma_{e^{ i t r}}$ and $\tilde{\alpha}_t={\tilde{\gamma}}_{e^{ i t r}}$.
	Let $\delta$ be a finite regular Borel measure on $\Lambda^\infty$ such that $\int f_\beta\,d\delta=1$. Define $\varepsilon=(\varepsilon_v)\in [0,\infty)^{\Lambda^0}$ by $\varepsilon_v=\delta(Z(v))$ and take  $y=(y_v)\in [0,\infty)^{\Lambda^0}$ as in \cite[Theorem~6.1]{aHLRS}. Then $y \cdot \varepsilon=1,$ and the restriction of the state $\phi_\delta$ of
	Theorem~\ref{theorem1} to $(\TCL,\tilde{\alpha})$ is the state $\phi_\varepsilon$ of \cite[Theorem~6.1]{aHLRS}.
\end{prop}

\begin{cor}\label{restrictioncor}
	Let  $\Lambda, \beta, \alpha$ and $\tilde{\alpha}$ be as in Proposition~\ref{restriction}.  Suppose that $\delta_1, \delta_2$ are regular Borel measures on $\Lambda^\infty$
	satisfying $\int f_\beta\, d\delta_i=1$. Then $\phi_{\delta_1}|_{\TCL}=\phi_{\delta_2}|_{\TCL}$ if and only if $\delta_1(Z(v))=\delta_2(Z(v))$ for all $v\in \Lambda^0.$
\end{cor}

\begin{prop}
		Let  $\Lambda, \beta, \alpha$ and $\tilde{\alpha}$ be as in Proposition~\ref{restriction}. Then every KMS$_\beta$ state of $(\TCL,\tilde{\alpha})$ is the restriction of a KMS$_\beta$ state  of $\NT(X(\Lambda^\infty),\alpha)$.
\end{prop}
\begin{proof}
	Suppose that $\phi$ is a KMS$_\beta$ state of $(\TCL,\alpha)$. Then \cite[Theorem~6.1(c)]{aHLRS} implies that there is a vector $\varepsilon \in [0,\infty)^{\Lambda^0}$ such that $y\cdot \varepsilon=1$ and $\phi=\phi_\varepsilon$. If $\delta$ is a measure on $\Lambda^\infty$ such that
	$\delta(Z(v))=\varepsilon_v$ for all $v\in \Lambda^0$ and $\int f_\beta\,d\delta=1$, then  Proposition~\ref{restriction} implies that  $\phi_\delta|_{\TCL}=\phi_\varepsilon$. So it suffices to show that there is such a measure.
	
	To see this, let $D:=(1,\dots,1)$ and $M:=\{lD:l\in \N\}$.  For each  $m,n\in M$ with $m\leq n$, define  $r_{m,n}: \Lambda^n\rightarrow \Lambda^m$ by $r_{m,n}(\lambda)=\lambda(0,m)$.   Remark~2.2 of \cite{KP} shows that we can view   $\Lambda^\infty$ as
	the inverse limit of the system $(\{\Lambda^m\},\{r_{m,n}\})_{m,n\in M}$. Let $\pi_m:\Lambda^\infty\rightarrow \Lambda^m$ be the canonical map
	defined by	$\pi_m(z)=z(0,m)$ for $z\in \Lambda^\infty$. Now any family of measures $\delta_m$ on $\Lambda^m$ with finite $\delta_0$ and with
	\begin{align}\label{formula78}
	\int (f\circ r_{m,n})\,d\delta_n=\int f\,d\delta_m\, \text{ for } m\leq n \text {  and } f\in C(\Lambda^n),
	\end{align} 
	gives a unique measure $\delta$ on $\Lambda^\infty$ such that 
	$\int (f\circ \pi_m)\,d\delta=\int f\,d\delta_m$  for  $f\in C(\Lambda^m)$
	(see, for example \cite[Lemma~5.2]{aHKR}). So we aim to build such a family of measures: 
	
 We recursively choose  weights  $\{w_{\eta}: \eta\in \Lambda  \, \text{ with } d(\eta)=lD \text { for some }l\geq 1\}$  such that
	$\sum_{\lambda\in v\Lambda^D} w_\lambda=\varepsilon_v,$
	and
	\begin{align}\label{5FINAL}
	\sum_{ \lambda\in s(\mu)\Lambda^D}w_{\mu\lambda }&=w_\mu,
	\end{align}
	for $v\in \Lambda^0$ and  $\mu\in \Lambda^{lD} \,(l\geq 1)$. Then we  set $\delta_0:=\varepsilon$ and  $\delta_{m}(\mu )=w_{\mu}$ for all $\mu\in \Lambda^{lD}$.
	
	Next we check \eqref{formula78} for these measures. Let $m\in M$. Since the characteristic functions of singletons span $C(\Lambda^m)$, it is enough to prove  \eqref{formula78} for $f=\chi_{\{\mu\}}$ and $\mu\in \Lambda^m$. First notice  that
	$\chi_{\{\mu\}}\circ r_{m,m+D}=\sum_{\lambda\in s(\mu)\Lambda^{D}}\chi_{\{\mu\lambda\}}.$
	Then we have
	\begin{align}\label{thiscalculations}
	\int \chi_{\{\mu\}}\circ r_{m,m+D}\,d\delta_{m+D}\notag&=\int
	\sum_{\lambda\in s(\mu)\Lambda^{D}}\chi_{\{\mu\lambda\}}\,d\delta_{m+D}
	=\sum_{\lambda\in s(\mu)\Lambda^{D}}\delta_{m+D}(\mu\lambda)\\
	&=\delta_m(\mu)=\int \chi_{\{\mu\}}\,d\delta_m\qquad\text{using }\eqref{5FINAL}.
	\end{align}
For each $n\in M$ with $m\leq n$, we have
$r_{m,n}=r_{m,m+D}\circ r_{m+D,m+2D}\circ\cdots \circ r_{n-D,n}.$
Then applying the calculation \eqref{thiscalculations} finitely many times gives
\[\int \chi_{\{\mu\}}\circ r_{m,m+n}\,d\delta_{m+n}=\int \chi_{\{\mu\}}\,d\delta_m.\]
This is precisely \eqref{formula78}. Thus  there is a unique measure $\delta$ on $\Lambda^\infty$ such that	
	\[\int \chi_{\{v\}}\circ \pi_{0}\, d\delta=\int \chi_{\{v\}}\, d \delta_0 \text{ for } v\in \Lambda^0.\]
Notice that  $\int \chi_{\{v\}}\circ \pi_{0}\, d\delta=\delta(Z(v))$ and $\int \chi_{\{v\}}\, d \delta_0=\delta_0(v)=\varepsilon_v$.
Thus $\delta(Z(v))=\varepsilon_v$.
 A similar calculation to the first paragraph of the proof of \cite[Proposition~7.4]{AaHR} shows that $\int f_\beta\,d\delta=y\cdot \varepsilon=1$. Thus $\delta$ has the required properties.
\end{proof}

\appendix
\section{Realising the universal  Nica-covariant representation as a doubly commuting representation}
In \cite{S}, Solel used different notation to study product systems
over $\N^k$. He defined a ``doubly commuting relation'' in \cite [Definition~3.8]{S} and  showed  that it is equivalent  to Fowler's Nica-covariance relation \cite[Remark~3.12]{S}. Equation  \eqref{FORMULA} in Proposition~\ref{solelprop}  is a translation of Solel's doubly commuting relation from his notation to Fowler's notation. The difficulty with this translation was that the doubly commuting relation contains a flip map between fibres which Solel used without an explicit formula.
We found a nice formula for this flip map in Lemma~\ref{pframe}(c), and  now  the translation seems trivial. In this  appendix,  we use our flip map to show that the universal Nica-covariant representation $\psi$ satisfies the  doubly commuting relation.  
 We first need to understand Solel's notation.

Let $A$ be a $C^*$-algebra and  $Y$ be a  right Hilbert $A$--$A$ bimodule. Suppose that $\pi$ is  representation  of $A$ on  $B(\HH)$ for a Hilbert space $\HH$.
Let $Y\odot \HH$ be the algebraic tensor product of $Y$ and $\HH$. It follows from \cite[Proposition~2.6]{tfb} that the formula
\begin{align}\label{appinnerproduct}
\big(y\odot r\, \big|\, y'\odot r'\big)=\big( r\,\big|\, \pi(\langle y , y'\rangle)r'\big)\text{ for } y\odot r, y'\odot r'\in Y\odot \HH
\end{align}
defines a  semi-definite inner product on $Y\odot\HH$. 
 Let $Y\otimes_\pi \HH$ be the completion
of $Y\odot \HH$ with respect to this semi-definite inner product (see \cite[Lemma~2.16]{tfb}). Observe that
 $Y\odot \HH$ is balanced over $A$ in the sense that 
\begin{align}\label{hil-balanced}
y\cdot a\otimes r=y\otimes \pi(a)r \text{ for }  y\in Y, a\in A, r\in\HH.
\end{align}

Given   $S\in \LL(Y)$ and  $U\in \pi(A)'$,  an argument similar to that of \cite[Proposition~2.66]{tfb} shows that there is a well-defined bounded operator
 $S\otimes U$ on $Y\otimes_\pi \HH$ such that
 \[S\otimes U(y\otimes r)=S(y)\otimes U(r) \text{ for } y\otimes r\in Y\otimes_\pi \HH.\]

Let $X$ be a product system of right Hilbert $A$--$A$ bimodules over $\N^k$ and  let $\theta$ be a Toeplitz representation of 
$X$ on  $B(\HH)$ for a Hilbert space $\HH$. Muhly and Solel showed in  \cite[Lemmas~ 3.4--3.6]{MS} that  for each $m\in \N^k$ and  associated fibre $X_m$, there is a well-defined map 
$\widetilde{\theta}_m:  X_m\otimes_{\theta_0} \HH\rightarrow \HH$ such that
\[\widetilde{\theta}_m(x\otimes r)=\theta_m(x)r \text{ for all }x\otimes r\in X_m\otimes_{\theta_0} \HH.\]

 Let $1_{m}, 1_{\HH}$ be the identity maps on $X_{m}$ and $\HH$ respectively, and $\widetilde{\theta}_{m}^{*}$ be the adjoint of $\widetilde{\theta}_{m}$. Suppose that 
 $t_{m,n}$  is the flip map between fibres  $X_{m}$ and $X_{n}$ as in Lemma~\ref{pframe}.  
 A representation $\theta$ is a \textit{doubly commuting representation} if  for every $1\leq i\neq j\leq k$, we have
\begin{align*}
{\widetilde{\theta}_{e_j}}^{*}\widetilde{\theta}_{e_i}=(1_{e_j}\otimes
\widetilde{\theta}_{e_i})(t_{e_i,e_j}\otimes 1_{H})(1_{e_i}\otimes\widetilde{\theta}_{e_j}^{*}).
\end{align*}
It is observed in \cite[Lemma~3.9(i)]{S} that a doubly commuting representation $\theta$ satisfies  
	\begin{align*}
	(1_{n}\otimes\widetilde{\theta}_{m})(t_{m,n}\otimes  1_{H})(1_{m}\otimes\widetilde{\theta}_{n} ^*)=\widetilde{\theta}_{n} ^*\widetilde{\theta}_{m}.
	\end{align*}  
for $m,n\in \N^k_+$ with $m\wedge n=0$. 
\begin{prop}\label{AASOLELPROP}
	Let $h_1,\dots,h_k$ be $*$-commuting and surjective  local homeomorphisms on a compact Hausdorff space $Z$ and let  $X$ be the associated product system  as in Proposition~\ref{prop1}. Take $m,n\in \N^k_+$ such that $m\wedge n=0$. Then	
	\begin{align}\label{AASOLEL}
	(1_{n}\otimes\widetilde{\psi}_{m})(t_{m,n}\otimes  1_{H})(1_{m}\otimes\widetilde{\psi}_{n} ^*)=\widetilde{\psi}_{n}^*\widetilde{\psi}_{m}.
	\end{align} 
\end{prop}
Before starting the proof, notice that  we can view  $\psi$  as a representation on a Hilbert space $\HH$ (see \cite[Remark 3.12]{S}). So \eqref{AASOLEL} makes sense. We also need to compute the adjoint $\widetilde{\psi}_{n}^*: \HH\rightarrow X_{n}\otimes_{\psi_0} \HH$. The next lemma gives a formula for $\widetilde{\psi}_{n}^*$ in terms of a Parseval frame of $X_n$.
\begin{lemma}
	Let $\{\eta_j\}_{j=1}^d$ be a Parseval frame for the fibre $X_n$.   Then
	\begin{align}\label{AAframe}
	\widetilde{\psi}_n ^*(r)=\sum_{j=1}^{d}\eta_j\otimes \psi_{n}(\eta_j)^{*}r\text{ for $r\in \HH$}.
	\end{align}
\end{lemma}
\begin{proof}
	Fix $r\in \HH$ and let $y \otimes s\in X_{n}\otimes_{\psi_0} \HH$. Now we compute using \eqref{appinnerproduct}: 
		\begin{align*}
	\Big(\sum_{j=1}^{d}&\eta_j\otimes \psi_{n}(\eta_j)^{*}r\,\Big|\,y \otimes s\Big)=\sum_{j=1}^{d}\,\big(\psi_{n}(\eta_j)^{*}r\,\big|\, \psi_0 \big(\langle \eta_j,y\rangle\big)s \big)\\
	&= \Big(r\,\Big|\sum_{j=1}^{d}\psi_{n}(\eta_j)\psi_0\big(\langle \eta_j,y\rangle\big)s \Big)
	= \Big(r\,\Big|\,\sum_{j=1}^{d}\psi_{n}\big(\eta_j\cdot \langle \eta_j,y\rangle\big)s \Big)
	=\big(r\,\big|\,\psi_{n}(y)s\big).
	\end{align*}
	This is  precisely 	$\big( r\,\big|\,\widetilde{\psi}_{n}(y \otimes s)\big)$. Thus 
$\widetilde{\psi}_n^*(r)=\sum_{j=1}^{d}\eta_j\otimes \psi_{n}(\eta_j)^{*}r.$
\end{proof}
\begin{proof}[Proof of Proposition~\ref{AASOLELPROP}]
Let $x\otimes r\in X_m\otimes_{\psi_0}\HH$. We evaluate both sides of \eqref{AASOLEL} on $x\otimes r$. We will need to have Parseval 
frames for the fibres $X_m,X_n$. 	
 Let  $\{\rho_i:1\leq i\leq d\}$ be a  partition of unity  such that ${h^m}|_{\supp \rho_i},{h^n}|_{\supp \rho_i}$ are injective and suppose that  $\tau_i:=\sqrt{\rho_i}$. Note that $\{\tau_i\}$ forms a Parseval frame for both fibres $X_m,X_n$. Also since $m\wedge n=0$, $\{\tau_i\circ h^n\}$ and $\{\tau_i\circ h^m\}$ are Parseval frame for the fibres $X_m,X_n$, respectively.  
   
 We start computing  the left-hand side of  \eqref{AASOLEL} by applying the adjoint formula \eqref{AAframe} with the Parseval frame 
$\{\tau_j\} \subset X_n$. For convenience, set
\[\dagger:=(1_{n}\otimes\widetilde{\psi}_{m})(t_{m,n}\otimes  1_{H})(1_{m}\otimes\widetilde{\psi}_{n} ^*).\]
  Then  $\dagger(x\otimes r)=
(1_{n}\otimes\widetilde{\psi}_{m})(t_{m,n}\otimes 1_{H})\big(x\otimes\sum_{j=1}^{d}\tau_j\otimes \psi_{n}(\tau_j)^{*}r\big).$
Using the reconstruction formula for the Parseval frame $\{\tau_i\circ h^n\}\subset X_m$ gives
\begin{align*}
	\dagger(x\otimes r)&=\sum_{j=1}^{d}(1_{n}\otimes\widetilde{\psi}_{m})(t_{m,n}\otimes 1_{H})\Big(\Big(\sum_{i=1}^d\tau_i\circ h^n\cdot \big\langle\tau_i\circ h^n,x\big \rangle\Big)\otimes \tau_j\otimes \psi_{n}(\tau_j)^{*}r\Big)\\
	&=\sum_{i,j=1}^{d}(1_{n}\otimes\widetilde{\psi}_{m})(t_{m,n}\otimes 1_{H})\big(\tau_i\circ h^n\otimes \big\langle\tau_i\circ h^n,x\big \rangle \cdot \tau_j\otimes \psi_{n}(\tau_j)^{*}r\big).
	\end{align*}
Using  the  reconstruction formula for  the element $\langle\tau_i\circ h^n,x\rangle\cdot \tau_j\in X_n $ with the Parseval frame
 $\{\tau_l\} \subset X_n$, we have 
\[\dagger(x\otimes r)=\sum_{i,j=1}^{d}(1_{n}\otimes\widetilde{\psi}_{m})(t_{m,n}\otimes 1_{H})\Big(\tau_i\circ h^n\otimes \Big(\sum_{l=1}^d\tau_l\cdot \big\langle\tau_l,\big\langle\tau_i\circ h^n,x\big \rangle \cdot \tau_j\big\rangle\Big)\otimes \psi_{n}(\tau_j)^{*}r\Big).\]
Since the tensors are balanced (see \eqref{hil-balanced}), we have
\[\dagger(x\otimes r)=\sum_{i,j,l=1}^d(1_{n}\otimes\widetilde{\psi}_{m})(t_{m,n}\otimes 1_{H})\big(\tau_i\circ h^n\otimes \tau_l \otimes \psi_0\big(\big\langle\tau_l,\big\langle\tau_i\circ h^n,x\big \rangle \cdot \tau_j\big\rangle\big)\psi_{n}(\tau_j)^{*}r\big).\]
Now we apply (T3) and rearrange  this to get that
\[\dagger(x\otimes r)=\sum_{i,l=1}^{d}(1_{n}\otimes\widetilde{\psi}_{m})(t_{m,n}\otimes 1_{H})\Big(\tau_i\circ h^n\otimes \tau_l \otimes \psi_n\Big(\sum_{j=1}^d\tau_j\cdot \big\langle \tau_j,\big\langle x,\tau_i\circ h^n\big \rangle \cdot\tau_l\big\rangle\Big)^{*}r\Big).\]
We continue by using the flip map \eqref{flipformula} to get that
\begin{align}
\dagger(x\otimes r)\notag&=\sum_{i,l=1}^ d(1_{n}\otimes\widetilde{\psi}_{m})\Big(\tau_l\circ h^m\otimes \tau_i \otimes \psi_n\Big(\sum_{j=1}^d\tau_j\cdot \big\langle \tau_j,\big\langle x,\tau_i\circ h^n\big \rangle \cdot\tau_l\big\rangle\Big)^{*}r\Big)\\
\notag&=\sum_{i,l=1}^d\tau_l\circ h^m\otimes \psi_m(\tau_i) \psi_n\Big(\sum_{j=1}^d\tau_j\cdot \big\langle \tau_j,\big\langle x,\tau_i\circ h^n\big \rangle \cdot\tau_l\big\rangle\Big)^{*}r.
\end{align}
The reconstruction formula for the frame $\{\tau_j\}\subset X_n$ implies that
\begin{align}\label{Aleft}
\dagger(x\otimes r)=\sum_{i,l=1}^{d}\tau_l\circ h^m\otimes \psi_m(\tau_i) \psi_n\big(\big\langle x,\tau_i\circ h^n\big \rangle \cdot\tau_l\big)^{*}r.
\end{align}

Next we compute the right-hand side of \eqref{AASOLEL} by applying the adjoint formula \eqref{AAframe} with the Parseval frame $\{\tau_l\circ h^m\} \subset X_n$ to get that
\begin{align}
\widetilde{\psi}_{n} ^*\widetilde{\psi}_{m}(x\otimes r)\notag&=\widetilde{\psi}_{n} ^*\psi_m(x)r=\sum_{l=1}^d\tau_l\circ h^m\otimes 
\psi_n(\tau_l\circ h^m)^*{\psi}_{m}(x)r.
\end{align}
We use the formula  \eqref{FORMULA}	to rewrite $\psi_n(\tau_l\circ h^m)^*{\psi}_{m}(x)$.  We have 
\begin{align}
\widetilde{\psi}_{n} ^*\widetilde{\psi}_{m}(x\otimes r)\notag&=\sum_{l=1}^d\tau_l\circ h^m\otimes \Big(\sum_{i,j=1}^d
\psi_m\big(\big\langle \tau_l\circ h^m,\tau_j\circ h^m \big\rangle\cdot \tau_i\big)\psi_{n}\big(\langle x, \tau_i\circ h^n\rangle\cdot\tau_j\big)^*\Big)r\\
\notag&=\sum_{i,j,l=1}^d
	\tau_l\circ h^m\otimes \psi_0\big(\langle \tau_l\circ h^m,\tau_j\circ h^m \big\rangle\big)\psi_m(\tau_i)\psi_{n}\big(\langle x, \tau_i\circ h^n\rangle\cdot\tau_j\big)^*r\\
\notag&=\sum_{i,j,l=1}^d\tau_l\circ h^m\cdot\langle \tau_l\circ h^m,\tau_j\circ h^m \big\rangle\otimes\psi_m(\tau_i)\psi_{n}\big(\langle x, \tau_i\circ h^n\rangle\cdot\tau_j\big)^*r.
\end{align}
Now applying reconstruction formula for the  Parseval frame $\{\tau_l\circ h^m\}\subset X_n$ gives
\begin{align}\label{Aright}
\widetilde{\psi}_{n} ^*\widetilde{\psi}_{m}(x\otimes r)=\sum_{i,j=1}^d\tau_j\circ h^m\otimes\psi_m(\tau_i)\psi_{n}\big(\langle x, \tau_i\circ h^n\rangle\cdot\tau_j\big)^*r.
\end{align}
Comparing \eqref{Aright} and \eqref{Aleft} completes our proof of \eqref{AASOLEL}.
\end{proof}


\end{document}